\documentclass[a4paper,10pt]{article}

\usepackage[pdftex,unicode]{hyperref}
\usepackage{wrapfig}
\usepackage{bm}
\usepackage[usenames]{color}

\usepackage[warn]{mathtext}
\usepackage[T1]{fontenc}
\usepackage[utf8]{inputenc}
\usepackage[english]{babel}
\usepackage{amsmath}
\usepackage[left=2cm,right=2cm,top=2cm,bottom=2cm,bindingoffset=0cm]{geometry}
\usepackage{amsfonts}
\usepackage{amssymb}
\usepackage{bbm}
\usepackage{mathrsfs}
\usepackage{amsthm}
\usepackage{euscript}
\usepackage{bigints}
\usepackage{xcolor}
\usepackage{graphicx}
\graphicspath{}
\DeclareGraphicsExtensions{.pdf,.png,.jpg}
\graphicspath{{C:\Users\TS\Desktop\kdipl}}

\theoremstyle{plain}
\theoremstyle{plain}\newtheorem{Le}{Lemma}[section]
\theoremstyle{plain}
\theoremstyle{definition}\newtheorem{Def}[Le]{Definition}
\theoremstyle{plain}\newtheorem{St}[Le]{Proposition}
\theoremstyle{plain}\newtheorem{Th}[Le]{Theorem}
\theoremstyle{plain}\newtheorem{Cor}[Le]{Corollary}
\theoremstyle{plain}\newtheorem{Rem}[Le]{Remark}
\theoremstyle{plain}
\theoremstyle{plain}
\theoremstyle{plain}
\theoremstyle{plain}
\theoremstyle{plain}\newtheorem{Ex}[Le]{Example}
\numberwithin{equation}{section}

\newtheorem*{theorem*}{Theorem}{\bf}{\it}
\newtheorem*{proposition*}{Proposition}{\bf}{\it}
\newtheorem*{observation*}{Observation}{\bf}{\it}
\newtheorem*{lemma*}{Lemma}{\bf}{\it}
\theoremstyle{definition}

\theoremstyle{remark}

\DeclareMathOperator{\dir}{dir}
\DeclareMathOperator{\fin}{fin}
\DeclareMathOperator{\ri}{ri}
\DeclareMathOperator{\ray}{ray}
\DeclareMathOperator{\rb}{rb}

\DeclareMathOperator{\dist}{dist}
\DeclareMathOperator{\aff}{aff}
\DeclareMathOperator{\conv}{conv}

\DeclareMathOperator{\clos}{clos}
\DeclareMathOperator{\intX}{int}
\newcommand{\BMO}{\mathrm{BMO}}

\newcommand{\BG}{\mathcal{B}}
\newcommand{\re}{\overline{\mathbb{R}}}
\newcommand{\supfu}[1]{\overline{#1}}

\begin{document}

\title{Structure of Minimal Locally Concave Functions}

\author{Egor Dobronravov\footnote{Supported by the Russian Science Foundation grant 24-71-10011}}

\maketitle

\begin{abstract}
We propose a theory describing the structure of minimal locally concave functions on an arbitrary subdomain of $\mathbb{R}^d$. We introduce the notion of an extremal set, that is, a set on which the function is affine, and prove that extremal sets foliate the domain. We show that extremal sets are covered by simplices with vertices on the boundary of the domain, and that these simplices themselves lie in the extremal set. We also prove that one-dimensional extremal sets can approach the free boundary only tangentially, not transversally.
\end{abstract}

\section{Introduction}\label{introd}
The main object of study in this paper is the class of minimal locally concave functions. Let $\Omega$ be a subset of $X$, where $X$ is a vector space over $\mathbb{R}$ (except in Part~\ref{prinnasl}, we always have $X=\mathbb{R}^d$). A function defined on $\Omega$ is called locally concave if its restriction to every segment lying entirely in $\Omega$ is concave. Let $g$ be a function defined on $\Omega$ or on a subset of $\Omega$. We call the pointwise infimum of locally concave functions on $\Omega$ that are pointwise greater than or equal to $g$ the minimal locally concave function on $\Omega$ with obstacle $g$. Minimal locally concave and maximal locally convex functions appear in several areas of mathematics. In convex domains they are concave and, respectively, convex solutions of the degenerate Monge--Amp\`ere equation~\cite{CafNirSpr, LiWang, RaTa, TruUrb}; in nonconvex domains they are locally concave or locally convex solutions of the same equation~\cite{Guan}. They also occur as solutions of a number of optimization problems in harmonic analysis~\cite{r7, r6, StaVaZa1, StZa, r40, Vas1, VaZaZl2}. Theorem 1.3.4 of~\cite{StaVaZaMRTR} establishes a relation between minimal locally concave functions in the parabolic strip and minimal biconcave functions in the diagonal strip, which also solve a number of optimization problems in harmonic analysis (see~\cite{StaVaZaMRTR}).

For minimal locally concave functions, as for solutions of partial differential equations, it is natural to ask how their smoothness depends on the smoothness of the boundary values. For the degenerate Monge--Amp\`ere equation this question was studied for convex domains in~\cite{CafNirSpr, TruUrb}, and in nonconvex domains under the additional assumption that there exists a strictly concave $C^2$-smooth majorant with the same boundary values in~\cite{Guan}. The existence of such a strictly concave $C^2$-smooth majorant imposes strong restrictions on the structure of a minimal locally concave function (see Sec.~\ref{c2slmn}); in particular, it excludes tangent extremal sets, although these are the most typical extremal sets in the Bellman function theory, as Sec.~\ref{emprul} shows. In the Bellman function theory the exact computation of optimization functions is crucial, and it is often based on guessing the structure of the minimal locally concave function. The present paper develops a theory of extremal sets, that is segments, and domains in which minimal locally concave functions are affine.

In~\cite{CafNirSpr} and~\cite{DePhiFig}, the authors studied the smoothness of the maximal convex function defined, respectively, on a convex compact set with a boundary majorant and with a majorant given on the whole domain. The main tool in their study of maximal convex functions is a certain consequence of Carath\'eodory's theorem (see Lemma 2 in~\cite{CafNirSpr} and Lemma 3.2 in~\cite{DePhiFig}). For nonconvex domains Carath\'eodory's theorem is not applicable, and to prove a similar result, Theorem~\ref{mainth}, one has to develop a new theory.

The main result of the paper is Theorem~\ref{mainth}. Informally, it says that every interior point of the domain of a minimal locally concave function lies in a simplex with vertices on the boundary, and the function is affine on this simplex. This theorem, like Lemma 2 of~\cite{CafNirSpr} and Lemma 3.2 of~\cite{DePhiFig}, will be the key tool in studying how the smoothness of minimal locally concave functions depends on the smoothness of the boundary data. These results will be presented in subsequent papers.

Two further topics are considered in the paper. The first is the range of applicability of a basic tool for studying minimal locally concave functions, called in the previous papers the hereditary principle (see Part~\ref{prinnasl}). The second is the formalization and proof of the empirical rule that one-dimensional extremal sets can approach the free boundary only tangentially (see Sec.~\ref{emprul}).

\subsection{Structure of the paper}
The paper is divided into three parts.

The first part consists of two sections. In the first, we recall the definition of faces of a convex set and prove several properties of faces that will be needed in the main part of the paper. In the second, we recall the definition of generalized simplices in the sense of Rockafellar and prove several technical lemmas about them.

The second part is devoted to the hereditary principle for minimal locally concave functions. It says that the restriction of a minimal locally concave function to a subdomain coincides with the minimal locally concave function whose boundary data is given by the initial minimal locally concave function on the boundary of the new subdomain.

The third part develops the theory describing the structure of minimal locally concave functions. Its first section contains preliminary material. In the second section of this part we define the contact zone, that is, the part of the domain in which the function coincides with the obstacle. The third section discusses the coincidence set of a minimal locally concave function with a locally supporting affine functional; this is an intermediate notion needed for the definition of extremal sets. In the fourth section we introduce the main object of the paper, namely, extremal sets of a locally concave function, and prove their basic properties. Informally, an extremal set is a domain in an affine subspace in which the minimal locally concave function is itself affine. Thus, extremal sets of a minimal locally concave function $\BG$ are determined by two principal parameters: an affine subspace $l$ containing the extremal set and an affine functional $L$ which equals the minimal locally concave function on the extremal set, as well as by an auxiliary parameter, a point $x$ that selects the connectivity component of the subdomain of $l$ where $\BG=L$. The most important property of extremal sets proved in this section is Lemma~\ref{extremal boundary} saying that the boundary of an extremal set, except for the boundary of the domain $\Omega$, is a union of extremal sets of smaller dimension. This will later allow us to use induction on the dimension of extremal sets in the proof of the main theorem. In the fifth section we prove the theorem on the structure of extremal sets (see Theorem~\ref{mainth}). It asserts that, although an extremal set may be nonconvex, its points can still be covered by simplices lying in the extremal set and having vertices on the boundary of the domain of the minimal locally concave function. The sixth section is devoted to formalizing and proving the empirical rule that one-dimensional extremal sets can approach the inner boundary only tangentially. In the penultimate section we show which restrictions on the structure of a minimal locally concave function are imposed by the existence of a smooth strictly locally concave majorant with the same boundary values. The last chapter says a few words about possible generalizations of the results of the work.

\subsection{Description of the ideas of the proof of the main Theorem~\ref{mainth}}
In this subsection we describe the ideas of the proof of the main theorem, Theorem~\ref{mainth}, which says that the domain of a minimal locally concave function is covered by pieces of affinity of the function, called extremal sets. These extremal sets, in turn, are covered by lines and by generalized simplices with vertices on the boundary of the domain, and these simplices lie entirely inside the extremal set.

Let $x\in\intX\Omega$. We want to find either a generalized simplex containing $x$, lying in the domain, with vertices on the boundary of the domain, on which $\BG$ is affine, or a line lying entirely in the domain, containing $x$, on which $\BG$ is affine again.
\begin{enumerate}
\item
We prove that there exists an extremal set $E(x,\,L,\,l)$ passing through $x$ (see Definition~\ref{opredext}); that is, we choose an affine functional $L$ and an affine subspace $l$ such that $\BG\leqslant L$ in a neighborhood of $x$, while $\BG=L$ in a neighborhood of $x$ inside $l$. The extremal set $E(x,\,L,\,l)$ is a certain subset of $l$ on which $\BG=L$.
\item
We see that the boundary of the extremal set, relative to $l$, consists of points of the boundary of the domain and extremal sets of smaller dimension (see Lemma~\ref{extremal boundary}).
\item
We define the set $V(x,\,L,\,l)$ of candidates for the vertices of the desired simplex: these are the points of $\partial\Omega$ (and the directions of rays) that are visible from $x$ inside $E(x,\,L,\,l)$ (see Definitions~\ref{canconver},~\ref{cannapver}).
\item
In Lemma~\ref{leprovipob} we prove by induction on the dimension of $l$ that
\begin{enumerate}
\item
 either $x\in\conv(V(x,\,L,\,l))$,
\item
 or a line $m''\subseteq E(x,\,L,\,l)$ passes through $x$.
\end{enumerate}
\item
It follows from the previous item that there exists a simplex $C$ with vertices in $V(x,\,L,\,l)$ containing $x$, and among such simplices, by Corollary~\ref{tsorsim}, we may choose an inclusion-minimal one, denoted by $C_{\min}$.
\item
If $C_{\min}$ does not have the required properties, we obtain a contradiction to minimality as follows:
\begin{enumerate}
\item
If it does not lie entirely in $\Omega$, then it contains a point of $\partial\Omega$;
\item
Gradually expanding the point $x$ to the simplex $C_{\min}$, we find a point $z\in V(x,\,L,\,l)$ in it;
\item
We can decrease $C_{\min}$ by replacing one of its vertices with $z$.
\end{enumerate}
\end{enumerate}

\subsection{Notation}

\noindent By $\re$ we denote the extended real line, that is, $\re=\mathbb{R}\cup\{-\infty,\,+\infty\}$.

\noindent Let $X$ be a topological space, and let $V\subseteq U\subseteq X$. We denote by $\clos_U V$, $\intX_U V$, and $\partial_U V$ the closure, interior, and boundary of $V$ in the topology induced from $X$ on $U$.

\begin{Rem}
The closure in the induced topology is given by the formula $\clos_U V=(\clos V)\cap U$. If $V$ is closed in $U$ \textup{(}i.e., $V=\clos_U V$\textup{)}, then, for any set $W\subseteq U$ that is closed in $X$, the intersection $W\cap V$ is closed in $X$.
\end{Rem}
\begin{Rem}
The boundary in the induced topology is given by the formula $\partial_U V=\partial V\cap\partial(U\setminus V)\cap U$, and the chain of inclusions $\partial V\setminus\partial U\subseteq \partial_U V\subseteq\partial V$ holds.
\end{Rem}

\noindent If $X$ is a vector space over $\mathbb{R}$, then we assume that its finite-dimensional subspaces are equipped with the standard topology of the Euclidean space. For a nonempty set $U\subseteq X$, we denote by $\aff(U)$ and $\conv(U)$ its affine and convex hulls, respectively. For $a,\,b\in X$, we denote by $[a,\, b]$ the segment with the endpoints $a$ and $b$.

\noindent For $x\in\mathbb{R}^d$, we denote by $B_{r}(x)$ the closed ball with center $x$ and radius $r$; in particular, $B_0(x)=\{x\}$. The length of the vector $x$ will be denoted by $|x|$, that is,
\begin{equation}
|x|=\sqrt{\sum_{j=1}^dx_j^2}.
\end{equation}

\noindent Let $U\subseteq\mathbb{R}^d$ be a convex set. We denote by $\ri U$ and $\rb U$ the relative interior and relative boundary of $U$, respectively: $\ri U=\intX_{\aff U}U$, $\rb U=\partial_{\aff U}U$.

\noindent The set of natural numbers, including zero, is denoted by $\mathbb{N}_0$. By $\#A$ we denote the cardinality of the set $A$.

\noindent For a function $f\colon\Omega\subseteq\mathbb{R}^d\to\re$, we denote by $\supfu{f}$ its upper semicontinuous majorant, that is, the function defined on $\clos\Omega$ by the formula
\begin{equation}
\supfu{f}(y)
=
\limsup\limits_{z\in\Omega,z\to y}f(z).
\end{equation}

{\bf Acknowledgments.} The author expresses his deep gratitude to his supervisor Dmitry Stolyarov for suggesting the problem, constant attention to the work, and valuable advice on the exposition. The author is also sincerely grateful to Mikhail Novikov for carefully reading the final version of the paper and for valuable remarks that helped eliminate misprints and inaccuracies.

\part{On convex geometry}\label{convgeom}
\section{Faces of a convex set}\label{secfacomn}
In this paper we will need the notion of a face of a convex set. The definition of a face of a convex set and proofs of its basic properties are given in Rockafellar~\cite{Rockafellar} (p. 162). We state the facts we need.
\begin{Def}
The dimension of a convex set is the dimension of its affine hull. The dimension of the empty set is formally defined as $-1$.
\end{Def}
\begin{Def}
A convex subset $F$ of a convex set $U\subseteq\mathbb{R}^d$ is called a face of $U$ if, for any points $x,\,y\in U$ such that $(x,\,y)\cap F\ne\varnothing$, the inclusion $[x,\,y]\subseteq F$ holds. By the dimension of a face we mean its dimension as a convex set.
\end{Def}
\begin{Def}
A point $x$ of a convex set $U$ is called extreme if it is not an interior point of any segment with endpoints in $U$. In other words, for any points $y,\,z\in U$, the inclusion $x\in [y,\,z]$ implies either $x=y$ or $x=z$.
\end{Def}
\begin{Rem}
Zero-dimensional faces are one-point sets whose elements are extreme points. Thus, in what follows we will identify a zero-dimensional face with the corresponding extreme point.
\end{Rem}

\begin{Ex}\label{cilindr}
Let $U\subseteq\mathbb{R}^3$ be the cylinder
\begin{equation}
U=\{(x,\,y,\,z)\in\mathbb{R}^3|\,x^2+y^2\leqslant 1,\ z\in[0,\,1]\}.
\end{equation}
The faces of $U$ can be listed in decreasing order of their dimensions.
\begin{enumerate}
\item[3.]
There is one three-dimensional face, namely the set $U$ itself.
\item[2.]
The two-dimensional faces are the upper and lower bases of the cylinder $F_0=\{(x,\,y,\,0)\in\mathbb{R}^3|\,x^2+y^2\leqslant 1\}$, $F_1=\{(x,\,y,\,1)\in\mathbb{R}^3|\,x^2+y^2\leqslant 1\}$.
\item[1.]
One-dimensional faces correspond to vertical segments on the lateral surface, $F_{x,\,y}=\{(x,\,y,\,z)\in\mathbb{R}^3|\,z\in[0,\,1]\}$, where $x^2+y^2=1$.
\item[0.]
Zero-dimensional faces correspond to extreme points of the circles that form the upper and lower bases: $F_{x,\,y,\,0}=\{(x,\,y,\,0)\}$ and $F_{x,\,y,\,1}=\{(x,\,y,\,1)\}$, where $x^2+y^2=1$.
\item[-1.]
The formal empty face.
\end{enumerate}
\end{Ex}
\begin{figure}[h]\centering
\def\svgwidth{5cm}
%% Creator: Inkscape 1.3.2 (091e20e, 2023-11-25, custom), www.inkscape.org
%% Accompanies image file 'cilindr.pdf' (pdf, eps, ps)
\begingroup%
  \makeatletter%
  \providecommand\color[2][]{%
    \errmessage{(Inkscape) Color is used for the text in Inkscape, but the package 'color.sty' is not loaded}%
    \renewcommand\color[2][]{}%
  }%
  \providecommand\transparent[1]{%
    \errmessage{(Inkscape) Transparency is used (non-zero) for the text in Inkscape, but the package 'transparent.sty' is not loaded}%
    \renewcommand\transparent[1]{}%
  }%
  \providecommand\rotatebox[2]{#2}%
  \newcommand*\fsize{\dimexpr\f@size pt\relax}%
  \newcommand*\lineheight[1]{\fontsize{\fsize}{#1\fsize}\selectfont}%
  \ifx\svgwidth\undefined%
    \setlength{\unitlength}{192.75590551bp}%
    \ifx\svgscale\undefined%
      \relax%
    \else%
      \setlength{\unitlength}{\unitlength * \real{\svgscale}}%
    \fi%
  \else%
    \setlength{\unitlength}{\svgwidth}%
  \fi%
  \global\let\svgwidth\undefined%
  \global\let\svgscale\undefined%
  \makeatother%
  \begin{picture}(1,1.27941176)%
    \lineheight{1}%
    \setlength\tabcolsep{0pt}%
    \put(0,0){\includegraphics[width=\unitlength,page=1]{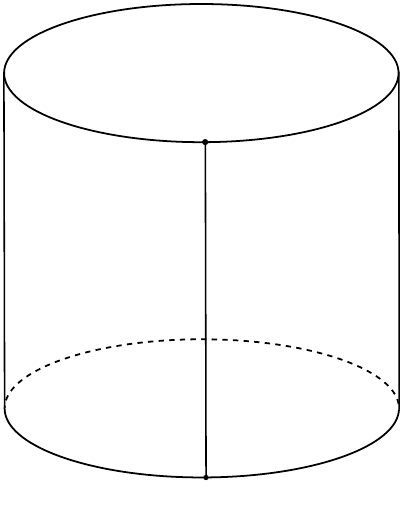}}%
    \put(0.26119983,0.70743242){\color[rgb]{0,0,0}\makebox(0,0)[lt]{\lineheight{1.25}\smash{\begin{tabular}[t]{l}$U$\end{tabular}}}}%
    \put(0.47902729,1.13476241){\color[rgb]{0,0,0}\makebox(0,0)[lt]{\lineheight{1.25}\smash{\begin{tabular}[t]{l}$F_1$\end{tabular}}}}%
    \put(0.29866066,0.28703954){\color[rgb]{0,0,0}\makebox(0,0)[lt]{\lineheight{1.25}\smash{\begin{tabular}[t]{l}$F_0$\end{tabular}}}}%
    \put(0.51926292,0.58672552){\color[rgb]{0,0,0}\makebox(0,0)[lt]{\lineheight{1.25}\smash{\begin{tabular}[t]{l}$F_{x,\,y}$\end{tabular}}}}%
    \put(0.48318963,0.95300849){\color[rgb]{0,0,0}\makebox(0,0)[lt]{\lineheight{1.25}\smash{\begin{tabular}[t]{l}$F_{x,\,y,\,1}$\end{tabular}}}}%
    \put(0.49706396,0.02396949){\color[rgb]{0,0,0}\makebox(0,0)[lt]{\lineheight{1.25}\smash{\begin{tabular}[t]{l}$F_{x,\,y,\,0}$\end{tabular}}}}%
  \end{picture}%
\endgroup%

\caption{Example~\ref{cilindr}}
\end{figure}

We will need the following properties of faces of a convex set; we state them without proof.
\begin{St}[see~\cite{Rockafellar} p. 162]\label{graff}
Let $C\subseteq\mathbb{R}^d$ be a nonempty convex set, and let $L\colon \mathbb{R}^d\to\mathbb{R}$ be an affine functional. Then the set $K=\{x\in C|\,L(x)=\max\{L(y)|y\in C\}\}$ is a face of $C$.
\end{St}
\begin{St}[see~\cite{Rockafellar} p. 163]\label{grgrgr}
If $F$ is a face of $U$, and $U$ is a face of $W$, then $F$ is a face of $W$.
\end{St}
\begin{St}[see~\cite{Rockafellar} Theorem 18.1]\label{grvipvnutgr}
Let $F$ be a face of $U$, and let $W$ be a convex subset of $U$ such that $\ri(W)\cap F\ne\varnothing$. Then $W\subseteq F$.
\end{St}
\begin{St}\label{bougr}
Let $U\subseteq\mathbb{R}^d$ be a convex set, and let $x\in U$ lie on the relative boundary of $U$. Then $x$ lies in some face of $U$ whose dimension is strictly smaller than the dimension of $U$.
\end{St}
\begin{proof}
This proposition is a direct corollary of Theorem 18.2 in~\cite{Rockafellar}, which says that any convex set is the disjoint union of the relative interiors of its faces, and of Corollary 18.1.3 there, which says that any face of a convex set not coinciding with the whole set has dimension smaller than that of the set itself.
\end{proof}
We will also need several lemmas concerning the local nature of faces of a convex set. Although we could not find them in the literature, they are probably not new.
\begin{Le}\label{locface}
Let $U,\,W\subseteq\mathbb{R}^d$ be convex sets, let $x\in U$ and $x\in\intX W$ \textup{(}in particular, $\intX W\ne\varnothing$ and $\dim \aff(W)=d$\textup{)}, and let $l$ be an affine subspace of $\mathbb{R}^d$ containing the point $x$. Then the following conditions are equivalent.
\begin{enumerate}
\item[$1.$]
$l\cap U$ is a face of $U$.
\item[$2.$]
$l\cap (U\cap W)$ is a face of $U\cap W$.
\end{enumerate}
\end{Le}
\begin{figure}[h]\centering
\def\svgwidth{10cm}\label{Figclegrvpergl}
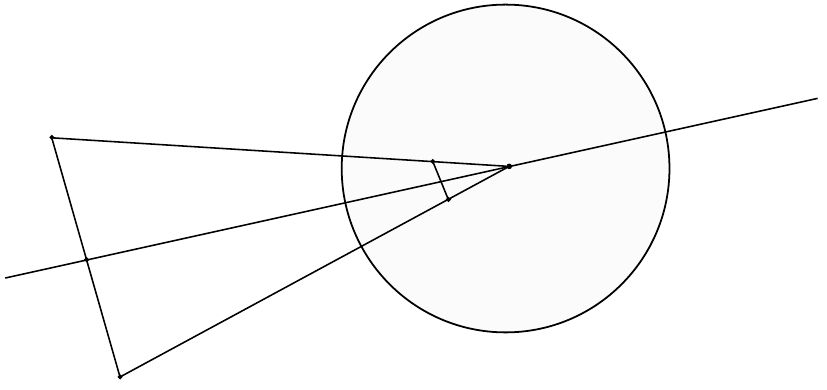
\caption{Illustration for Lemma~\ref{locface}}
\end{figure}

\begin{proof}
The proof consists of two steps.

{\bf 1) Derivation of $2.$ from $1.$}

The set $l\cap (U\cap W)$ is convex as an intersection of convex sets. Let $y,\,z\in U\cap W$ and $(l\cap (U\cap W))\cap (y,\,z)\ne\varnothing$. Then $y,\,z\in U$ and $(l\cap U)\cap (y,\,z)\ne\varnothing$. Thus, $y,\,z\in l\cap U$ by assumption $1$. Therefore, $y,\,z\in l\cap(U\cap W)$, and hence $[y,\,z]\subseteq l\cap(U\cap W)$.

{\bf 2) Derivation of $1.$ from $2.$}

Assume the contrary. Suppose that $l\cap U$ is not a face of $U$. Then there exist points $y, z\in U$ such that $(y,\,z)\cap (l\cap U)\ne\varnothing$, but $[y,\,z]\nsubseteq l\cap U$.
The segment $[y,\,z]$ has a point in the subspace $l$, but the whole segment does not lie in $l$. Thus, $[y,\,z]$ intersects $l$ at a single point $v$. In particular, neither $y$ nor $z$ lies in $l$. Therefore, since $x\in l$, the following formula:
\begin{equation}
\forall\varepsilon >0\colon\quad (1-\varepsilon)x+\varepsilon y,\,(1-\varepsilon)x+\varepsilon z\notin l;\quad\text{see Fig.~\ref{Figclegrvpergl}.}
\end{equation}
holds.
Choose the number $\varepsilon$ so small that $(1-\varepsilon)x+\varepsilon y,\,(1-\varepsilon)x+\varepsilon z\in W$. Then
\begin{equation}
(1-\varepsilon)x+\varepsilon v\in((1-\varepsilon)x+\varepsilon y,\,(1-\varepsilon)x+\varepsilon z)\cap (l\cap (U\cap W))\ne\varnothing,
\end{equation}
however
\begin{equation}
[(1-\varepsilon)x+\varepsilon y,\,(1-\varepsilon)x+\varepsilon z]\nsubseteq l\cap (U\cap W).
\end{equation}
Therefore, $l\cap(U\cap W)$ is not a face. Contradiction.
\end{proof}
\begin{Le}
Let $U,\,W\subseteq\mathbb{R}^d$ be convex sets, let $x\in U$, $x\in\intX W$, and let $l$ be an affine subspace of $\mathbb{R}^d$ such that $x\in l$. Then the equality $\aff(l\cap U)=\aff(l\cap (U\cap W))$ holds.
\end{Le}
\begin{proof}
On the one hand, we have the inclusion $\aff(l\cap(U\cap W))\subseteq\aff(l\cap U)$. On the other hand, let $\{x_1,\,\dots,\,x_k\}\subseteq l\cap U$ be such that $\aff(\{x_1,\,\dots,\,x_k\})=\aff(l\cap U)$. Since $x\in\intX W$, for sufficiently small number $\varepsilon>0$ the inclusion $\{x,\,(1-\varepsilon)x+\varepsilon x_1,\,\dots,\,(1-\varepsilon)x+\varepsilon x_k\}\subseteq W$ holds. Therefore,
\begin{equation*}
\aff(l\cap U)=\aff(\{x_1,\,\dots,\,x_k\})\subseteq\aff(\{x,\,(1-\varepsilon)x+\varepsilon x_1,\,\dots,\,(1-\varepsilon)x+\varepsilon x_k\})\subseteq\aff(l\cap(U\cap W)).
\end{equation*}
\end{proof}
\begin{Cor}\label{lemepdel}
Let $U,\,W\subseteq\mathbb{R}^d$ be convex sets, let $x\in U$ and $x\in\intX W$, and let $l$ be an affine subspace of $\mathbb{R}^d$ such that $x\in l$. The following conditions are equivalent.
\begin{enumerate}
\item
$l\cap U$ is a face of $U$ such that $l=\aff(l\cap U)$.
\item
$l\cap (U\cap W)$ is a face of $U\cap W$ such that $l=\aff(l\cap (U\cap W))$.
\end{enumerate}
\end{Cor}

\section{Extension of the space $\mathbb{R}^d$ by a set of directions and generalized simplices}\label{secgensim}

\subsection{Definitions and basic properties}

Extended spaces and simplices with vertices at infinity were considered by Rockafellar (see~\cite{Rockafellar}, \S 17).

The set of directions can be identified with the $(d-1)$-dimensional sphere $S^{d-1}$. Let $e\in S^{d-1}$; then~$\vec{e}$ is the point at infinity in this direction. Unlike in projective geometry, $\vec{e}$ and $-\vec{e}$ are different points at infinity. The set of points at infinity corresponds not to equivalence classes of parallel lines, as in projective geometry, but to equivalence classes of codirectional rays. We denote the set of points at infinity by $\vec{S}^{d-1}$. When speaking about convergence of directions, we assume that $\vec{S}^{d-1}$ inherits its topology from $S^{d-1}$. We will also need the notion of convergence of finite points to a given direction; for this purpose we give the following definition.
\begin{Def}
We say that a sequence of points $\{y_n\}_{n\in\mathbb{N}}\subseteq\mathbb{R}^d$ converges to a direction $\vec{e}\in\vec{S}^{d-1}$, and we write $y_n{\to} \vec{e}$ as $n\to+\infty$, if $|y_n|{\to}+\infty$ and ${y_n}/{|y_n|}{\to}e$ as $n\to+\infty$.
\end{Def}
\begin{Def}
A pair of sets $(V;\;E)$, where $V\subseteq\mathbb{R}^d$ and $E\subseteq\vec{S}^{d-1}$, is called closed if $V$ is closed in the topology of $\mathbb{R}^d$, $E$ is closed in the topology of the sphere, and, for any sequence of points $\{v_n\}_{n\in\mathbb{N}}$ from $V$ and any direction $\vec{e}\in\vec{S}^{d-1}$ such that $v_n{\to}\vec{e}$ as $n\to+\infty$, the inclusion $\vec{e}\in E$ holds.
\end{Def}

\begin{Def}
Let $x\in\mathbb{R}^d$ and $\vec{e}\in\vec{S}^{d-1}$. We denote by $[x,\,\vec{e})$ and $(x,\,\vec{e})$ the closed and open rays, respectively, starting from $x$ in the direction $\vec{e}$:
\begin{equation}
[x,\,\vec{e}\,)=\{x+\alpha e|\,\alpha\geqslant 0\},
\end{equation}
\begin{equation}
(x,\,\vec{e}\,)=\{x+\alpha e|\,\alpha> 0\}.
\end{equation}
The point $x$ will be called the origin, or base point, and the direction $\vec{e}$ will be called the direction of the ray.
\end{Def}

\begin{Def}
Let $E\subseteq \vec{S}^{d-1}$. We denote by $\ray E$ the set formed by the points of rays starting from the origin in the directions from $E$:
\begin{equation}
\ray E=\{0\}\cup\bigcup\limits_{\vec{e}\in E}(0,\,\vec{e}).
\end{equation}
In other words, $\ray E=\{0\}$ for $E=\varnothing$, and for $E\ne\varnothing$ we have
\begin{equation}
\ray E=\{\alpha e|\,\alpha\in[0,\,+\infty),\,\vec{e}\in E\}.
\end{equation}
\end{Def}
\begin{Def}
Let $V\subseteq\mathbb{R}^d$, $V\ne\varnothing$, and $E\subseteq\vec{S}^{d-1}$. The convex hull of the set of points $V$ and the set of directions $E$ is defined as the following set:
\begin{equation}
\conv(V;\;E)=\conv(V+\ray E).
\end{equation}
In other words,
\begin{multline}
\conv(V;\;E)
=
\Big\{x\in\mathbb{R}^d\Big|\,
\exists k\in\mathbb{N},\,y_1,\dots,\,y_k\in V,\,r\in\mathbb{N}_0,\,\vec{e}_1,\dots,\,\vec{e}_r\in E,\,\alpha_1,\dots,\,\alpha_k,\,\beta_1,\dots,\,\beta_r\geqslant 0\colon
\\
\sum\limits_{j=1}^{k}\alpha_j=1\text{ and }\sum\limits_{j=1}^{k}\alpha_j y_j+\sum\limits_{i=1}^{r}\beta_i e_i=x 
\Big\}.
\end{multline}
Sums of the form
\begin{equation}
\sum\limits_{j=1}^{k}\alpha_j y_j+\sum\limits_{i=1}^{r}\beta_i e_i,
\end{equation}
where $\alpha_1,\dots,\,\alpha_k,\,\beta_1,\dots,\,\beta_r\geqslant 0$ and $\sum_{j=1}^{k}\alpha_j=1$, will be called convex combinations of elements of the sets $V$ and $E$.
\end{Def}
\begin{Le}[see~\cite{Rockafellar} Theorem 18.3]\label{faceconv}
Let $V\subseteq\mathbb{R}^d$, $V\ne\varnothing$, and $E\subseteq\vec{S}^{d-1}$. Let $C=\conv(V;\;E)$, and let $C'\subseteq C$ be a nonempty face of $C$. Then there exist sets $V'\subseteq V$, $V'\ne\varnothing$, and $E'\subseteq E$ such that $C'=\conv(V';\;E')$.
\end{Le}
\begin{Def}
Let $V\subseteq\mathbb{R}^d$, $V\ne\varnothing$, and $E\subseteq\vec{S}^{d-1}$. The affine hull of the set of points $V$ and the set of directions $E$ is defined as the following set:
\begin{equation}
\aff(V;\;E)=\aff(V+\ray E).
\end{equation}
In other words,
\begin{multline}
\aff(V;\;E)
=
\Big\{
x\in\mathbb{R}^d\Big|\,
\exists k\in\mathbb{N},\,y_1,\dots,\,y_k\in V,\,r\in\mathbb{N}_0,\,\vec{e}_1,\dots,\,\vec{e}_r\in E,\,\alpha_1,\dots,\,\alpha_k,\,\beta_1,\dots,\,\beta_r\in\mathbb{R}\colon
\\
\sum\limits_{j=1}^{k}\alpha_j=1\text{ and }\sum\limits_{j=1}^{k}\alpha_j y_j+\sum\limits_{i=1}^{r}\beta_i e_i=x 
\Big\}.
\end{multline}
Sums of the form
\begin{equation}
\sum\limits_{j=1}^{k}\alpha_j y_j+\sum\limits_{i=1}^{r}\beta_i e_i,
\end{equation}
where $\alpha_1,\dots,\,\alpha_k,\,\beta_1,\dots,\,\beta_r\in\mathbb{R}$ and $\sum_{j=1}^{k}\alpha_j=1$, will be called affine combinations of elements of the sets $V$ and $E$.
\end{Def}
\begin{Def}
A pair of sets $(V;\;E)$, where $V\subseteq\mathbb{R}^d$, $V\ne\varnothing$, and $E\subseteq\vec{S}^{d-1}$, is called affinely independent if
\begin{equation}
\dim(\aff(V;\;E))=\#V+\#E-1.
\end{equation}
\end{Def}
\begin{Rem}
For $V=\{x_1,\dots,\,x_k\}\subseteq\mathbb{R}^d$, $V\ne\varnothing$, and $E=\{\vec{e}_1,\dots,\,\vec{e}_r\}\subseteq\vec{S}^{d-1}$, the condition that the pair of sets $(V;\;E)$ is affinely independent is equivalent to saying that any point $x\in\aff(V;\;E)$ is represented uniquely as an affine combination of elements of the sets $V$ and $E$.
\end{Rem}
\begin{Def}
The convex hull $C$ of an affinely independent pair of sets $(V;\;E)$, where $V\subseteq\mathbb{R}^d$, $V\ne\varnothing$, and $E\subseteq\vec{S}^{d-1}$, is called a generalized simplex with finite vertices $V_{\fin}(C)=V$ and vertices $V_{\dir}(C)=E$ at infinity.
\end{Def}
\begin{Th}[Carath\'eodory; see~\cite{Rockafellar} Theorem 17.1]\label{carat}
Let $V\subseteq\mathbb{R}^d$, $V\ne\varnothing$, and $E\subseteq\vec{S}^{d-1}$. Then $\conv(V;\; E)$ is the union of generalized simplices with vertices in $V\cup E$.
\end{Th}

\begin{Le}
Let $V\subseteq\mathbb{R}^d$, $V\ne\varnothing$, and $E\subseteq\vec{S}^{d-1}$. Let $C=\conv(V;\;E)$ be a generalized simplex. The set $C'\subseteq C$ is a face of $C$ if and only if there exist sets $V'\subseteq V$, $V'\ne\varnothing$, and $E'\subseteq E$ such that $C'=\conv(V';\;E')$.
\end{Le}
This lemma is probably not new. The author could not find it in the literature, so we provide a proof.
\begin{proof}
By Lemma~\ref{faceconv} any face of the simplex $C$ has the form $\conv(V';\;E')$, where $V'\subseteq V$ and $E'\subseteq E$. Therefore, it suffices to show that for any $V'\subseteq V$, $V'\ne\varnothing$ and $E'\subseteq E$ set $\conv(V';\;E')$ is a face of $C$.

Assume the contrary: $\conv(V';\;E')$ is not a face of $C$. Then there exist points $y,\,z\in C$ such that $[y,\,z]\nsubseteq\conv(V';\;E')$, but $(y,\,z)\cap\conv(V';\;E')\ne\varnothing$. Then at least one of the points $y,\,z$ does not lie in $\conv(V';\;E')$. Without loss of generality we assume that it is $y$. Let $x\in(y,\,z)\cap\conv(V';\;E')$. On the one hand, since $x\in\conv(V';\;E')$, the point $x$ is representable as a convex combination of elements of the sets $V'$ and $E'$. On the other hand, there exists a number $\gamma\in(0,\,1)$ such that $x=\gamma y+(1-\gamma)z$, and $y$ and $z$ are representable as convex combinations of elements of $V$ and $E$. Moreover, since $y$ does not lie in $\conv(V';\;E')$, in the representation of $y$ as a convex combination of elements of the sets $V$ and $E$ there is a nonzero coefficient at some element of the sets $(V\setminus V')$ and $(E\setminus E')$. Then the equality $x=\gamma y+(1-\gamma)z$ gives a representation of $x$ as a convex combination of elements of the sets $V$ and $E$ with at least one nonzero coefficient at an element of the sets $(V\setminus V')$ and $(E\setminus E')$. Thus, we have represented $x$ as two different convex combinations of elements of the sets $V$ and $E$, which contradicts the affine independence of the pair of sets $(V;\;E)$.
\end{proof}

\begin{Def}
The ray $[x,\,\vec{e})$ will be called an extreme ray of the convex set $U$ if it is a face of $U$.
\end{Def}

\begin{Cor}
Let $C\subseteq\mathbb{R}^d$ be a generalized simplex. Then $V_{\fin}(C)$ is the set of extreme points of $C$, and $V_{\dir}(C)$ is the set of directions of extreme rays of $C$. In particular, $C$ does not contain lines, since $V_{\fin}(C)\ne\varnothing$.
\end{Cor}

\begin{Le}\label{luvvipmn}
Let $U\subseteq\mathbb{R}^d$ be a closed convex set and let $\vec{e}\in\vec{S}^{d-1}$. Then the following conditions are equivalent.
\begin{enumerate}
\item[$1.$]
There exists a sequence of points $\{y_j\}_{j\in\mathbb{N}}\subseteq U$ such that $y_j\underset{j\to+\infty}{\longrightarrow}\vec{e}$.
\item[$2.$]
For any $x\in U$, the ray $[x,\,\vec{e}\,)$ lies in $U$.
\end{enumerate}
\end{Le}
\begin{proof}$ $

{\bf Implication from condition $2$ to condition $1$.}

If $[x,\,\vec{e})\subseteq U$, then the sequence $\{x+je\}_{j\in\mathbb{N}}$ lies in $U$ and tends to $\vec{e}$.

{\bf Implication from condition $1$ to condition $2$.}

Let $x\in U$ and let $\{y_j\}_{j\in\mathbb{N}}\subseteq U$, with $y_j\to\vec{e}$ as $j\to+\infty$. Since the set $U$ is convex, the inclusion $[x,\,y_j]\subseteq U$ holds. By the closedness of $U$, the inclusion $[x,\,\vec{e})\subseteq U$ also holds.

\end{proof}

\begin{Def}
Let $U\subseteq\mathbb{R}^d$ be a closed convex set. We say that a direction $\vec{e}\in\vec{S}^{d-1}$ lies in $U$ if, for any point $x\in U$, the ray $[x,\,\vec{e})$ lies in $U$.
\end{Def}

\begin{Le}
Let $C=\conv(\{x_1,\dots,\,x_k\};\;\{\vec{e}_1,\dots,\,\vec{e}_r\})$ be a generalized simplex, and let $\vec{e}\in\vec{S}^{d-1}$. Then the following conditions are equivalent.
\begin{enumerate}
\item
The direction $\vec{e}$ lies in $C$.
\item
There exist nonnegative numbers $\gamma_1,\dots,\,\gamma_r$ such that $\sum\limits_{i=1}^{r}\gamma_i e_i=e$.
\end{enumerate}
\end{Le}
\begin{proof}$ $

{\bf Implication from condition $2$ to condition $1$.}

If $x\in C$, then the element $x$ is representable as a convex combination of elements of the sets $\{x_1,\dots,\,x_k\}\cup\{\vec{e}_1,\dots,\,\vec{e}_r\}$. In this case, for $\alpha\geqslant 0$, the point $x+\alpha e=x+\alpha\sum_{i=1}^{r}\gamma_i e_i$ is also representable as a convex combination of elements of the sets $\{x_1,\dots,\,x_k\}\cup\{\vec{e}_1,\dots,\,\vec{e}_r\}$. Therefore, the whole ray $[x,\,\vec{e})$ lies in $C$. Thus, $\vec{e}$ lies in $C$.

{\bf Implication from condition $1$ to condition $2$.}

Let $[x,\,\vec{e})\subseteq C$. Then the points $x$ and $x+e$ are representable as convex combinations of elements of the sets $\{x_1,\dots,\,x_k\}\cup\{\vec{e}_1,\dots,\,\vec{e}_r\}$. In other words, there exist numbers
\begin{equation}
\alpha_1,\dots,\,\alpha_k,\,\alpha'_1,\dots,\,\alpha'_k,\,\beta_1,\dots,\,\beta_r,\,\beta'_1,\dots,\,\beta'_r\geqslant 0
\end{equation}
such that
\begin{equation}
\sum\limits_{j=1}^{k}\alpha_j
=
\sum\limits_{j=1}^{k}\alpha'_j=1,\quad
\sum\limits_{j=1}^{k}\alpha_j x_j+\sum\limits_{i=1}^{r}\beta_i e_i=x,
\quad\text{and}\quad
\sum\limits_{j=1}^{k}\alpha'_j x_j+\sum\limits_{i=1}^{r}\beta'_i e_i=x+e.
\end{equation}
Then
\begin{equation}\label{echere1ek}
e=\sum\limits_{j=1}^{k}(\alpha'_j-\alpha_j)x_j+\sum\limits_{i=1}^{r}(\beta'_i-\beta_i)e_i.
\end{equation}
Consequently, for any numbers $\zeta\in\mathbb{R}$ the formula 
\begin{equation}
x+\zeta e=\sum\limits_{j=1}^{k}(\alpha_j+\zeta(\alpha'_j-\alpha_j))x_j+\sum\limits_{i=1}^{r}(\beta_i+\zeta(\beta'_i-\beta_i))e_i
\end{equation}
holds. We have found a representation of the point $x+\zeta e$ as an affine combination of elements of the sets $\{x_1,\dots,\,x_k\}\cup\{\vec{e}_1,\dots,\,\vec{e}_r\}$. On the one hand, by the affine independence of the sets $\{x_1,\dots,\,x_k\}\cup\{\vec{e}_1,\dots,\,\vec{e}_r\}$, points of the form $x+\zeta e$ are uniquely representable as affine combinations of them. On the other hand, since the ray $[x,\,\vec{e})$ lies in $C$, for nonnegative $\zeta$ the point $x+\zeta e$ must be representable as a convex combination of the given elements. Thus, for any $\zeta\geqslant 0$ the affine combination written above is convex. In particular, for any $j=1,\dots,\,k$ and any $\zeta\geqslant 0$ the inequalities $0\leqslant \alpha_j+\zeta(\alpha'_j-\alpha_j)\leqslant 1$ hold, which is possible only if $\alpha'_j-\alpha_j=0$. Also, for any $i=1,\dots,\,r$ and any $\zeta\geqslant 0$, the inequality $\beta_i+\zeta(\beta'_i-\beta_i)\geqslant 0$ holds, which is possible only if $\beta'_i-\beta_i\geqslant 0$. Therefore, formula~\eqref{echere1ek} gives the desired representation.
\end{proof}

\subsection{On contracting a convex hull to a fixed point}
In this subsection we prove lemmas whose geometric meaning can be formulated as ``if we contract points of a sets to a given point, then the points of the convex hull also contract to the given point''.

\begin{Le}\label{stve}
Let $x,\,y,\,z,\,z'\in\mathbb{R}^d$, $V\subseteq\mathbb{R}^d$, and $E\subseteq\vec{S}^{d-1}$. Assume that $y\in\conv(V\cup\{z\};\;E)$ and $z'\in\conv(\{x,\,y,\,z\})$. Then there exists an element $y'\in[y,\,x]$ such that $y'\in\conv(V\cup\{z'\};\;E)$.
\end{Le}
\begin{figure}[h]\centering
\def\svgwidth{7cm}
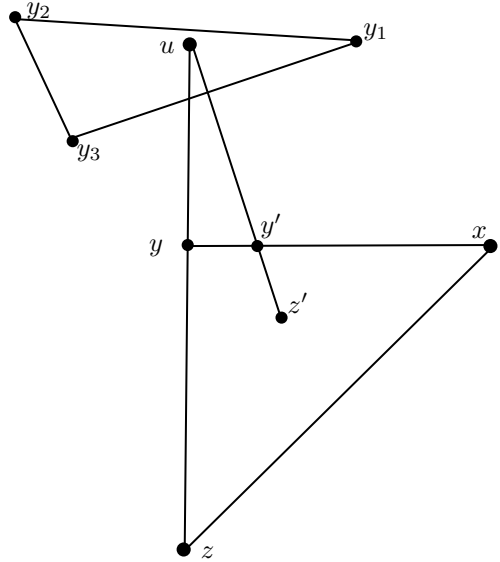
\caption{Illustration for the proof of Lemma~\ref{stve}}\label{ilstve}
\end{figure}

\begin{proof}
If $y\in\conv(V;\;E)$, then $y\in\conv(V\cup\{z'\};\;E)$, and one can take $y'=y$. If $y\notin\conv(V;\;E)$, then there exist numbers $\alpha,\,\alpha_1,\dots,\,\alpha_k,\,\beta_1,\dots,\,\beta_r>0$ such that
\begin{equation}
\alpha+\sum\limits_{j=1}^{k}\alpha_j=1
\quad\text{and}\quad
y=\alpha z+\sum\limits_{j=1}^{k}\alpha_j y_j+\sum\limits_{i=1}^{r}\beta_i e_i.
\end{equation}
Since $z'\in\conv(\{x,\,y,\,z\})$, there exist numbers $a,\,b,\,c\geqslant 0$ such that $z'=ax+by+cz$ and $a+b+c=1$. Consider two cases.

{\bf The first case: $k=0$.}

In this case $\alpha=1$ and $y=z+\sum_{i=1}^{r}\beta_i e_i$. It suffices to take $y'=z'+c\sum_{i=1}^{r}\beta_i e_i$: First, $y'$ is defined as a convex combination of elements of the sets $\{z'\}$ and $E$, and hence $y'$ automatically belongs to $\conv(V\cup\{z'\};\;E)$; Second,
\begin{equation}
y'=z'+c\sum\limits_{i=1}^{r}\beta_i e_i=ax+by+cz+c\sum\limits_{i=1}^{r}\beta_i e_i=ax+by+cy=ax+(1-a)y\in[x,\,y].
\end{equation}

{\bf Second case: $k>0$.}

In this case $V\ne\varnothing$ and $\alpha\ne 1$. Define the point $u$ by the formula
\begin{equation}\label{texfor}
u=\frac{1}{1-\alpha}y-\frac{\alpha}{1-\alpha}z=\frac{1}{1-\alpha}\sum\limits_{j=1}^{k}\alpha_j y_j+\frac{1}{1-\alpha}\sum\limits_{i=1}^{r}\beta_i e_i\in\conv(V;\;E).
\end{equation}
The point $u$ is the intersection of the ray with base $z$ passing through $y$ and the convex hull of the points $y_1,\dots,\,y_k$ and the directions $e_1,\dots,\,e_r$ (see Fig.~\ref{ilstve}).
Formula~\eqref{texfor} implies that $y=\alpha z+(1-\alpha)u$.

Let $\gamma=\frac{\alpha}{\alpha+(1-\alpha)c}$, let us check that it suffices to set $y'=\gamma z'+(1-\gamma)u$. By the definition, $y'\in[z',\,u]\subseteq\conv(V\cup\{z'\};E)$, and we only need to check that $y'\in[x,\,y]$: 
\begin{multline}
y'=\gamma z'+(1-\gamma)u=\frac{\alpha a}{\alpha+(1-\alpha)c}x+\frac{\alpha b}{\alpha+(1-\alpha)c}y+\frac{\alpha c}{\alpha+(1-\alpha)c}z+\frac{(1-\alpha) c}{\alpha+(1-\alpha)c}u=
\\
=\frac{\alpha a}{\alpha+(1-\alpha)c}x+\frac{\alpha (1-a-c)}{\alpha+(1-\alpha)c}y+\frac{c}{\alpha+(1-\alpha)c}y=\frac{\alpha a}{\alpha+(1-\alpha)c}x+\frac{\alpha +(1-\alpha)c-\alpha a}{\alpha+(1-\alpha)c}y\in[x,\,y].
\end{multline}
\end{proof}
\begin{Cor}\label{corstve}
Let $x,\,y,\,z,\,z'\in\mathbb{R}^d$, $V\subseteq\mathbb{R}^d$, and $E\subseteq\vec{S}^{d-1}$. Assume that $[x,\,y]\cap\conv(V\cup\{z\};\;E)\ne\varnothing$ and $z'\in\conv(\{x,\,y,\,z\})$. Then there exists an element $y'\in[x,\,y]$ such that $y'\in\conv(V\cup\{z'\};\;E)$.
\end{Cor}
\begin{proof}
Let $y_0\in[x,\,y]\cap\conv(V\cup\{z\};\;E)$. Then 
\begin{equation}
\conv(\{x,\,y,\,z\})=\conv(\{x,\,y_0,\,z\})\cup\conv(\{y_0,\,y,\,z\}).
\end{equation}
Consequently, $z'$ lies in at least one of the two sets $\conv(\{x,\,y_0,\,z\})$, $\conv(\{y_0,\,y,\,z\})$. If it lies in the first, then by Lemma~\ref{stve} there exists $y'\in[x,\,y_0]\cap\conv(V\cup\{z'\};\;E)$; if it lies in the second, then again by Lemma~\ref{stve} we find $y'\in[y,\,y_0]\cap\conv(V\cup\{z'\};\;E)$. Since $y_0\in[x,\,y]$, the inclusions $[x,\,y_0],\,[y,\,y_0]\subseteq[x,\,y]$ hold; that is, in both cases $y'\in[x,\,y]\cap\conv(V\cup\{z'\};\;E)$.
\end{proof}

The following lemma is an analog of Lemma~\ref{stve}, where $z$ is a point at infinity.
\begin{Le}\label{stna}
Let $x,\,y,\,z'\in\mathbb{R}^d$, $V\subseteq\mathbb{R}^d$, $V\ne\varnothing$, $\vec{e}\in\vec{S}^{d-1}$, and $E\subseteq\vec{S}^{d-1}$. Assume that $y\in\conv(V;\;E\cup\{\vec{e}\})$ and $z'\in\conv(\{x,\,y\};\;\{\vec{e}\})$. Then there exists an element $y'\in[y,\,x]$ such that $y'\in\conv(V\cup\{z'\};\;E)$.
\end{Le}
\begin{proof}
First we observe that there exists a number $\beta\geqslant 0$ such that $y-\beta e\in\conv(V;\;E)$. Since $y\in\conv(V;\;E\cup\{\vec{e}\})$, the element $y$ is represented as a convex combination of elements of $V\cup E\cup\{\vec{e}\}$.  Note that it suffices to take $\beta$ equal to the coefficient of $e$ in this convex combination, since then the point $y-\beta e$ will be representable as convex combinations of elements of $V\cup E$. Since $z'\in\conv(\{x,\,y\};\;\{\vec{e}\})$, there exist numbers $\alpha\in[0,\,1]$ and $\gamma\geqslant 0$ such that $z'=\alpha x+(1-\alpha)y+\gamma e$. If $\gamma=0$, then $z'\in[x,\,y]\cap\conv(V\cup\{z'\};\;E)$ and one may take $y'=z'$. In case $\gamma>0$, let us check that it suffices to set $y'=\frac{\gamma}{\gamma+\beta}(y-\beta e)+\frac{\beta}{\gamma+\beta}z'$. By the definition, $y'\in[z',\,y-\beta e]\subseteq\conv(V\cup\{z'\};\;E)$ and we only need to check that $y'\in[x,\,y]$:
\begin{multline}
y'
=
\frac{\gamma}{\gamma+\beta}(y-\beta e)+\frac{\beta}{\gamma+\beta}z'
=
\frac{\gamma}{\gamma+\beta}(y-\beta e)+\frac{\beta}{\gamma+\beta}(\alpha x+(1-\alpha)y+\gamma e)
=
\\
=
\frac{\alpha\beta}{\gamma+\beta}x+\frac{\gamma+(1-\alpha)\beta}{\gamma+\beta}y\in[x,\,y].
\end{multline}
\end{proof}
\begin{Cor}\label{corstna}
Let $x,\,y,\,z'\in\mathbb{R}^d$, $V\subseteq\mathbb{R}^d$, $V\ne\varnothing$, $\vec{e}\in\vec{S}^{d-1}$ and $E\subseteq\vec{S}^{d-1}$. Assume also that 
\begin{equation}
[x,\,y]\cap\conv(V;\;E\cup\{\vec{e}\})\ne\varnothing
\end{equation}
and $z'\in\conv(\{x,\,y\};\;\{\vec{e}\})$. Then there exists a point $y'\in[y,\,x]$ such that $y'\in\conv(V\cup\{z'\};\;E)$.
\end{Cor}
The deduction of Corollary~\ref{corstna} from Lemma~\ref{stna} is analogous to the derivation of Corollary~\ref{corstve} from Lemma~\ref{stve}.
\begin{Le}\label{stia}
Let $x,\,y\in\mathbb{R}^d$, $V,\,V'\subseteq\mathbb{R}^d$, $V,\,V'\ne\varnothing$, and $E,\,E'\subseteq\vec{S}^{d-1}$. Assume also that $y\in\conv(V;\;E)$ and that the following conditions hold.
\begin{enumerate}
\item
For any point $z\in V$ there exists $z'\in V'\cap\conv(\{x,\,y,\,z\})$.
\item
For any direction $\vec{e}\in E$, if $\vec{e}\notin E'$, then there exists $z'\in V'\cap\conv(\{x,\,y\};\;\{\vec{e}\})$.
\end{enumerate}
Then there exists a point $y'\in[x,\,y]$ such that $y'\in\conv(V';\; E')$.
\end{Le}

\begin{figure}[h]\centering
\def\svgwidth{7cm}
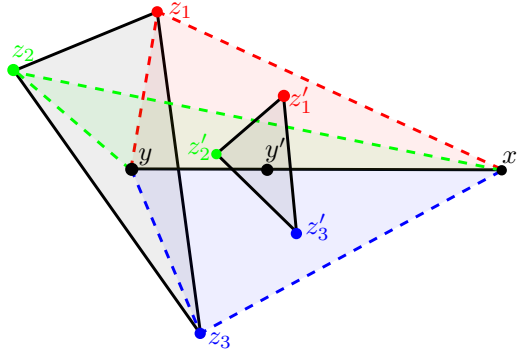
\caption{Illustration for the proof of Lemma~\ref{stia}}\label{ilstia}
\end{figure}
The proof of Lemma~\ref{stia} consists of the successive application of Corollaries~\ref{corstve} and~\ref{corstna} to the vertices of a generalized simplex with vertices from $V\cup E$ that contains $y$.
\begin{proof}
By the assumptions, there exists a function $\varphi\colon V\cup E\to V'\cup E'$ such that, for any point $v\in V$, the inclusion $\varphi(v)\in\conv(\{x,\,y,\,v\})$ holds, and, for any direction $\vec{e}\in E$, either $\varphi(\vec{e})=\vec{e}$ or $\varphi(\vec{e})\in\conv(\{x,\,y\};\;\{\vec{e}\})$. By Theorem~\ref{carat}, there exists a generalized simplex $C$ with vertices $y_1,\dots,\,y_k\in V$ and $\vec{e}_1,\dots,\,\vec{e}_r\in E$ such that $y\in C$. Let $V_0=\{y_1,\dots,\,y_k\}$, $E_0=\{\vec{e}_1,\dots,\,\vec{e}_r\}$. For $t=1,\dots,\,k$ set $V_t=\{\varphi(y_1),\dots,\,\varphi(y_t),\,y_{t+1},\dots,\,y_{k}\}$ and $E_t=E_0$. For $t=k+1,\dots,\,k+r$, if $\varphi(\vec{e}_{t-k})=\vec{e}_{t-k}$, then set $V_t=V_{t-1}$ and $E_t=E_{t-1}$; if $\varphi(\vec{e}_{t-k})\ne\vec{e}_{t-k}$, then set $V_t=V_{t-1}\cup\{\varphi(\vec{e}_{t-k})\}$ and $E_t=E_{t-1}\setminus\{\vec{e}_{t-k}\}$. Then $V_0\subseteq V$, $E_0\subseteq E$, $V_{k+r}\subseteq V'$, and $E_{k+r}\subseteq E'$. It remains to prove by induction on $t$ that there exists an element $u_t\in[x,\,y]$ such that $u_t\in\conv(V_t;\;E_t)$.

{\bf Base of the induction: $t=0$.}

In this case $u_0=y$.

{\bf Induction step with $t\leqslant k$.} 

In this case $E_t=E_{t-1}$, and $V_t$ differs from $V_{t-1}$ only in that the point $v_t$ is replaced by $\varphi(v_t)$, where $\varphi(v_t)\in\conv(\{x,\,y,\,y_t\})$. By the induction hypothesis, there exists a point $u_{t-1}\in[x,\,y]\cap\conv(V_{t-1},\,E_{t-1})$. Then, by Corollary~\ref{corstve} ($z=v_t$, $z'=\varphi(v_t)$), there exists $u_{t}\in[x,\,y]\cap\conv(V_{t},\,E_{t})$.

{\bf Induction step with $t>k$.} 

If $\varphi(\vec{e}_{t-k})=\vec{e}_{t-k}$, then $V_t=V_{t-1}$ and $E_{t}=E_{t-1}$, and hence one can take $u_t=u_{t-1}$.

If $\varphi(\vec{e}_{t-k})\ne\vec{e}_{t-k}$, then $\varphi(\vec{e}_{t-k})\in\conv(\{x,\,y\};\;\{\vec{e}_{t-k}\})$. In this case the sets $V_{t-1}$ and $E_{t-1}$ differ from the sets $V_t$ and $E_t$ only in that we removed the element $\vec{e}_{t-k}$ from $E_{t-1}$ and added the element $\varphi(\vec{e}_{t-k})$ to $V_{t-1}$. By the induction hypothesis, there exists a point $u_{t-1}\in[x,\,y]\cap\conv(V_{t-1},\,E_{t-1})$. Then, by Corollary~\ref{corstna} ($\vec{e}=\vec{e}_{t-k}$, $z'=\varphi(\vec{e}_{t-k})$), there exists a point $u_{t}\in[x,\,y]\cap\conv(V_{t},\,E_{t})$.

\end{proof}

\begin{Le}\label{stpr}
Let $x,\,y,\,z\in\mathbb{R}^d$ and $\vec{e}\in\vec{S}^{d-1}$. Assume also that $z\in\conv(\{x,\,y\};\;\{\vec{e}\})$. Then there exists a point $y'\in[x,\,y]$ such that $y'\in[z,\,-\vec{e})$.
\end{Le}

\begin{figure}[h]\centering
\def\svgwidth{7cm}
%% Creator: Inkscape 1.3.2 (091e20e, 2023-11-25, custom), www.inkscape.org
%% Accompanies image file 'stprris.pdf' (pdf, eps, ps)
\begingroup%
  \makeatletter%
  \providecommand\color[2][]{%
    \errmessage{(Inkscape) Color is used for the text in Inkscape, but the package 'color.sty' is not loaded}%
    \renewcommand\color[2][]{}%
  }%
  \providecommand\transparent[1]{%
    \errmessage{(Inkscape) Transparency is used (non-zero) for the text in Inkscape, but the package 'transparent.sty' is not loaded}%
    \renewcommand\transparent[1]{}%
  }%
  \providecommand\rotatebox[2]{#2}%
  \newcommand*\fsize{\dimexpr\f@size pt\relax}%
  \newcommand*\lineheight[1]{\fontsize{\fsize}{#1\fsize}\selectfont}%
  \ifx\svgwidth\undefined%
    \setlength{\unitlength}{255.11811024bp}%
    \ifx\svgscale\undefined%
      \relax%
    \else%
      \setlength{\unitlength}{\unitlength * \real{\svgscale}}%
    \fi%
  \else%
    \setlength{\unitlength}{\svgwidth}%
  \fi%
  \global\let\svgwidth\undefined%
  \global\let\svgscale\undefined%
  \makeatother%
  \begin{picture}(1,0.95555556)%
    \lineheight{1}%
    \setlength\tabcolsep{0pt}%
    \put(0,0){\includegraphics[width=\unitlength,page=1]{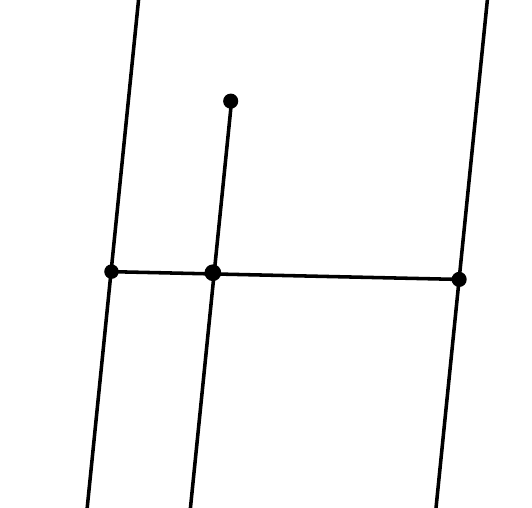}}%
    \put(0.89754418,0.43386112){\color[rgb]{0,0,0}\makebox(0,0)[lt]{\lineheight{1.25}\smash{\begin{tabular}[t]{l}$x$\end{tabular}}}}%
    \put(0.14842919,0.43845969){\color[rgb]{0,0,0}\makebox(0,0)[lt]{\lineheight{1.25}\smash{\begin{tabular}[t]{l}$y$\end{tabular}}}}%
    \put(0.45349341,0.78403904){\color[rgb]{0,0,0}\makebox(0,0)[lt]{\lineheight{1.25}\smash{\begin{tabular}[t]{l}$z$\end{tabular}}}}%
    \put(0.35411568,0.46974932){\color[rgb]{0,0,0}\makebox(0,0)[lt]{\lineheight{1.25}\smash{\begin{tabular}[t]{l}$y'$\end{tabular}}}}%
    \put(0.14252131,0.86957066){\color[rgb]{0,0,0}\makebox(0,0)[lt]{\lineheight{1.25}\smash{\begin{tabular}[t]{l}$\vec{e}$\end{tabular}}}}%
    \put(0.04754068,0.05160624){\color[rgb]{0,0,0}\makebox(0,0)[lt]{\lineheight{1.25}\smash{\begin{tabular}[t]{l}$-\vec{e}$\end{tabular}}}}%
    \put(0,0){\includegraphics[width=\unitlength,page=2]{stprris.pdf}}%
  \end{picture}%
\endgroup%

\caption{Illustration for Lemma~\ref{stpr}}\label{ilstpr}
\end{figure}

\begin{proof}
Since $z\in\conv(\{x,\,y\};\;\{\vec{e}\})$, there exist numbers $\alpha\in[0,\,1]$ and $\beta\geqslant 0$ such that 
\begin{equation}
z=\alpha x+(1-\alpha)y+\beta e.
\end{equation}
Then 
\begin{equation}
y'=\alpha x+(1-\alpha)y\in [x,\,y]\cap[z,\,-\vec{e}).
\end{equation}
\end{proof}
\begin{Cor}\label{stiag}
Let $x,\,y\in\mathbb{R}^d$, $V,\,V'\subseteq\mathbb{R}^d$, and $E,\,E'\subseteq\vec{S}^{d-1}$. Assume that the following conditions hold.
\begin{enumerate}
\item
For any point $z\in V$ there exists $z'\in V'\cap\conv(\{x,\,y,\,z\})$.
\item
For any direction $\vec{e}\in E$, if $\vec{e}\notin E'$, then there exists $z'\in V'\cap\conv(\{x,\,y\};\;\{\vec{e}\})$.
\end{enumerate}
Assume that at least one of the following two conditions holds.
\begin{enumerate}
\item[1']
$y\in\conv(V;\;E)$.
\item[2']
There exists a direction $\vec{e}\in\vec{S}^{d-1}$ such that $\vec{e},\,-\vec{e}\in E$.
\end{enumerate}
Then at least one of the following two statements holds.
\begin{enumerate}
\item[1'']
There exists an element $y'\in[x,\,y]$ such that $y'\in\conv(V';\; E')$.
\item[2'']
There exists a direction $\vec{e}\in\vec{S}^{d-1}$ such that $\vec{e},\,-\vec{e}\in E'$.
\end{enumerate}
\end{Cor}
\begin{proof}
If condition~1' holds, then we are in the setting of Lemma~\ref{stia}, and hence condition~1'' holds. If condition~2' holds but condition~2'' does not, then there exists a direction $\vec{e}\in\vec{S}^{d-1}$ such that $\{\vec{e},-\vec{e}\}\subseteq E$, whereas $\{\vec{e},-\vec{e}\}\nsubseteq E'$. Without loss of generality, let $\vec{e}\notin E'$. Then, by assumption~2, there exists a point $z_+\in V'\cap\conv(\{x,\,y\};\;\{\vec{e}\})$. By Lemma~\ref{stpr}, there exists a point
\begin{equation}\label{texforvslstiag}
y_+\in[x,\,y]\cap [z_+,\,-\vec{e})\subseteq [x,\,y]\cap\conv(V';\;E'\cup\{-\vec{e}\}).
\end{equation}
If $-\vec{e}\in E'$, then formula~\eqref{texforvslstiag} gives condition~1''. If $-\vec{e}\notin E'$, then, by assumption 2, there exists a point $z_-\in V'\cap\conv(\{x,\,y\};\;\{-\vec{e}\})$. Then, by Corollary~\ref{corstna}, the formula
\begin{equation}
\varnothing
\ne
[x,\,y]\cap\conv(V'\cup\{z_-\},\,E')
=
[x,\,y]\cap\conv(V',\,E').
\end{equation}
Again, condition~1'' holds.
\end{proof}

\subsection{On the intersection of a chain of nested generalized simplices.}
We will need to choose an inclusion-minimal generalized simplex from a certain family. Therefore, this subsection is devoted to proving that chains of simplices from such a family satisfy the hypothesis of Zorn's lemma. First we observe that it suffices to consider not arbitrary chains of nested simplices, but sequences indexed by natural numbers. The main result of the subsection is Corollary~\ref{tsorsim}.
\begin{Def}
Let $(\mathfrak{J},\,\leqslant)$ be a linearly ordered set. A collection of generalized simplices $\{C_j\}_{j\in\mathfrak{J}}$, indexed by the elements of $\mathfrak{J}$, will be called a decreasing chain of generalized simplices if, for any indices $i,\,j\in\mathfrak{J}$, the condition $i\leqslant j$ implies $C_j\subseteq C_i$.
\end{Def}
\begin{Le}\label{dosschcsim}
Let $\{C_j\}_{j\in\mathfrak{J}}$ be a decreasing chain of generalized simplices. Then there exists a sequence of elements $\{j_n\}_{n\in\mathbb{N}}\subseteq\mathfrak{J}$ such that $j_{n+1}\geqslant j_n$ and, for any $i\in\mathfrak{J}$, there exists a number $n\in\mathbb{N}$ such that $C_{j_n}\subseteq C_i$.
\end{Le}
\begin{proof}
If the chain stabilizes, that is, if there exists an index $j\in\mathfrak{J}$ such that $C_i=C_j$ for every $i\in\mathfrak{J}$ with $i\geqslant j$, then choose the constant sequence $j_n=j$. Further, without loss of generality, assume that the chain does not stabilize. Since the generalized simplices are nested, their dimensions do not increase and therefore stabilize. Thus, there exists an index $j\in\mathfrak{J}$ such that, for any index $i\in\mathfrak{J}$ with $i\geqslant j$, the equality $\dim(C_i)=\dim(C_j)$ holds. Since these simplices are nested, we have $\aff(C_i)=\aff(C_j)$. Denote this $j$ by $j_0$, and denote $\aff(C_{j_0})$ by $l$. Let a sequence $\{x_n\}_{n\in\mathbb{N}}\subseteq l$ be dense in $l$. We construct the sequence $\{j_n\}_{n\in\mathbb{N}}$ inductively. Suppose that the elements $j_0,\dots,\,j_n$ have already been chosen; we choose the element $j_{n+1}$ as follows.
\begin{enumerate}
\item
If $x_{n+1}\notin C_{j_n}$, then $j_{n+1}=j_{n}$.
\item
If $x_{n+1}\in{\bigcap}_{j\in\mathfrak{J}}C_j$, then $j_{n+1}=j_n$.
\item
If $x_{n+1}\in C_{j_n}$, but $x_{n+1}\notin{\bigcap}_{j\in\mathfrak{J}}C_j$, then choose $j_{n+1}$ so that $x_{n+1}\notin C_{j_{n+1}}$.
\end{enumerate}
Then the condition $j_{n+1}\geqslant j_n$ holds. Let us check that the constructed sequence has the required property. Let $i\in\mathfrak{J}$. If $\dim(C_i)>\dim(C_{j_1})$, then $C_i\nsubseteq C_{j_1}$, and hence $C_{j_1}\subseteq C_i$. If $\dim(C_i)=\dim(C_{j_1})$, choose an index $i'\in\mathfrak{J}$ such that $C_{i'}\subsetneq C_i$. Then $\ri(C_i)\setminus C_{i'}\ne\varnothing$ (because $C_i=\clos(\ri(C_i))$ and $C_{i'}=\clos(\ri(C_{i'}))$). Choose a number $k\in\mathbb{N}$ such that $x_k\in\ri(C_i)\setminus C_{i'}$. Such a number exists because the sequence $\{x_n\}_{n\in\mathbb{N}}$ is dense in $l$. Then $x_k\in C_i$, but $x_k\notin C_{i'}$. In particular, $x_{k}\notin\bigcap_{j\in\mathfrak{J}}C_j$; therefore, $x_k\notin C_{j_k}$. Thus, $C_i\nsubseteq C_{j_k}$ and $C_{j_k}\subseteq C_i$.
\end{proof}
Before passing to the search for a generalized simplex in the intersection of a sequence of nested generalized simplices, we show that generalized simplices are covered by certain cones.
\begin{Le}\label{kodno}
Let $C$ be generalized simplex, $V_{\fin}(C)=\{v_1,\dots,\,v_k\}$ and $V_{\dir}(C)=\{\vec{e}_1,\dots,\,\vec{e}_r\}$. Then 
\begin{equation}
C'=\conv\left(\{v_1\};\;\left\{\overrightarrow{\left(\frac{v_2-v_1}{|v_2-v_1|}\right)},\dots,\,\overrightarrow{\left(\frac{v_k-v_1}{|v_k-v_1|}\right)},\,\vec{e}_1,\dots,\,\vec{e}_r\right\}\right)
\end{equation}
is a generalized simplex containing $C$.
\end{Le}
\begin{proof}
First note that for $j=2,\dots,\,k$ the inclusion
\begin{equation}
v_j\in\left[v_1,\,\overrightarrow{\left(\frac{v_j-v_1}{|v_j-v_1|}\right)}\right)
\end{equation}
holds. Consequently,
\begin{equation}
\{v_1,\dots,\,v_k\}
\subseteq
\conv\left(\{v_1\};\;\left\{\overrightarrow{\left(\frac{v_2-v_1}{|v_2-v_1|}\right)},\dots,\,\overrightarrow{\left(\frac{v_k-v_1}{|v_k-v_1|}\right)},\,\vec{e}_1,\dots,\,\vec{e}_r\right\}\right).
\end{equation}
Hence $C\subseteq C'$. Let us check that $C'$ is a generalized simplex. Indeed, $C\subseteq C'$, and therefore $\aff(C)\subseteq\aff(C')$. In particular, $\dim(\aff(C'))\geqslant\dim(\aff(C))=k+r-1$. The set $C'$ is also the convex hull of $k+r$ points and directions; hence $\dim(\aff(C'))=k+r-1$ and $C'$ is a generalized simplex.
\end{proof}
\begin{Le}\label{pocvipcon}
Let $U\subseteq\mathbb{R}^d$ be a closed convex set that contains no lines. Let $x\in \mathbb{R}^d\setminus U$. Then the cone
\begin{equation}
C=\clos(\{x+\lambda(y-x)|\,y\in U,\,\lambda>0\}\cup\{x\})
\end{equation}
does not contain lines.
\end{Le}
\begin{proof}
Assume the contrary. Then there exists a point $v\in C\setminus\{x\}$ such that $2x-v\in C$. First we wish to prove that the ray
\begin{equation}
\left(x,\,\overrightarrow{\frac{v-x}{|v-x|}}\right)
\end{equation}
does not intersect $U$. Assume the contrary; then there exists a number $\lambda>0$ such that $y=x+\lambda(v-x)\in U$. Since $2x-v\in C$, there exists a sequence of points $\{v_n\}_n$ and positive numbers $\{\lambda_n\}_n$ such that $v_n\to 2x-v$ and $z_n=x+\lambda_n(v_n-x)\in U$. Then, by the convexity of $U$,
\begin{multline}
U\ni w_n=\frac{\lambda_n}{\lambda+\lambda_n}y+\frac{\lambda}{\lambda+\lambda_n}z_n
=
\frac{\lambda_n}{\lambda+\lambda_n}(x+\lambda(v-x))+\frac{\lambda}{\lambda+\lambda_n}(x+\lambda_n(v_n-x))
=
\\
=
x+\frac{\lambda\lambda_n}{\lambda+\lambda_n}(v_n-(2x-v)).
\end{multline}
Moreover
\begin{equation}
|w_n-x|=\frac{\lambda\lambda_n}{\lambda+\lambda_n}|v_n-(2x-v)|<\frac{\lambda\lambda_n+\lambda^2}{\lambda+\lambda_n}|v_n-(2x-v)|=\lambda|v_n-(2x-v)|\to 0.
\end{equation}
Since $U$ is closed, this means that $x\in U$, a contradiction.

Thus, the ray 
\begin{equation}
\left(x,\,\overrightarrow{\frac{v-x}{|v-x|}}\right)
\end{equation}
does not intersect $U$. Similarly, the ray
\begin{equation}
\left(x,\,-\overrightarrow{\frac{v-x}{|v-x|}}\right)
\end{equation}
does not intersect $U$. Let us show that it suffices to prove that if the ray $[x,\,\vec{e})$ lies in $C$ but does not intersect $U$, then the direction $\vec{e}$ lies in $U$. Indeed, in such a case both the direction $\overrightarrow{\frac{v-x}{|v-x|}}$ and the direction $-\overrightarrow{\frac{v-x}{|v-x|}}$ lie in $U$, and hence $U$ contains a line in the direction $\overrightarrow{\frac{v-x}{|v-x|}}$, which contradicts the assumption of the lemma.

Let $\vec{e}\in\vec{S}^{d-1}$ be a direction such that the ray $[x,\,\vec{e})$ lies in $C$ but does not intersect $U$. Then there exists a sequence of points $\{v_n\}_n$ and positive numbers $\{\lambda_n\}_n$ such that $v_n\to x+e$ and $x+\lambda_n(v_n-x)\in U$. Then $\lambda_n\to+\infty$; otherwise, the point $x+(\liminf\lambda_n)e$ would be a limit point of the sequence $\{x+\lambda_n(v_n-x)\}_n$ and hence would lie in the intersection of $U$ and the ray $[x,\,\vec{e})$. Therefore, $\lambda_n\to+\infty$ and $x+\lambda_n(v_n-x)\to \vec{e}$. Thus, by Lemma~\ref{luvvipmn}, the direction $\vec{e}$ lies in $U$.
\end{proof}

\begin{Le}\label{pocconconcon}
Let $C\subseteq\mathbb{R}^d$ be a closed cone that contains no lines and has a vertex at the point $x$. Then there exist $e\in S^{d-1}$ and $\varepsilon>0$ such that $C$ lies in the cone $\left\{y\in\mathbb{R}^d\big|\,\langle e,\,y-x\rangle\geqslant\varepsilon|y-x|\right\}$.
\end{Le}
\begin{proof}
Since $x$ is the unique extreme point of $C$, by Theorem 18.6 in~\cite{Rockafellar}, $x$ is an exposed point of the cone $C$. That is, there exists an affine functional $L\colon\mathbb{R}^d\to\mathbb{R}$ such that $L(x)=0$ and $L$ is positive on $C\setminus\{x\}$. Let $L^0$ be the kernel of $L$, let $L^+$ be the open half-space where $L$ is positive, and let $L^-$ be the open half-space where $L$ is negative. Let a vector $e\in S^{d-1}$ be orthogonal to $L^0$ and chosen in such a way that $x+e\in L^+$. Then there exists a number $\alpha>0$ such that $L(y)=\alpha \langle e,\,y-x\rangle$. Let us show that, for sufficiently small $\varepsilon>0$, the inclusion
\begin{equation}
C\subseteq\{y\in\mathbb{R}^d|\langle e,\,y-x\rangle\geqslant\varepsilon|y-x|\}
\end{equation}
holds. Assume the contrary, then for any $\varepsilon>0$ there exists a point $y\in C$ such that $\langle e,\,y-x\rangle<\varepsilon |y-x|$. Let $y_n'\in C$ be a point such that $\langle e,\,y_n'-x\rangle<\frac{1}{n} |y_n'-x|$. Then the point $y_n=x+\frac{y_n'-x}{|y_n'-x|}$ has the same properties, with $|y_n-x|=1$. One can choose a subsequence $y_{n_k}\to y$. However $|y-x|=1$ and $y\in C$, since the set $C$ is closed. This
\begin{equation}
0<\langle e,\,y-x\rangle=\lim\limits_{k\to+\infty}\langle e,\,y_{n_k}-x\rangle\leqslant\lim\limits_{k\to+\infty}\frac{1}{n_k}|y_{n_k}-x|=\lim\limits_{k\to+\infty}\frac{1}{n_k}=0.
\end{equation}
Contradiction.
\end{proof}

\begin{Cor}
Let $C$ be a generalized simplex, and let $x\in V_{\fin}(C)$. Then there exist $e\in S^{d-1}$ and $\varepsilon>0$ such that $C$ lies in the cone $\left\{y\in\mathbb{R}^d\big|\,\langle e,\,y-x\rangle\geqslant\varepsilon|y-x|\right\}$.
\end{Cor}
\begin{proof}
By Lemma~\ref{kodno}, the simplex $C$ is covered by a simplex $C'$ with the unique finite vertex $x$. Then $C'$ is a cone that contains no lines and has a vertex at $x$ . Therefore, by Lemma~\ref{pocconconcon}, the simplex $C'$, and hence the simplex $C$, is covered by the required cone.
\end{proof}
\begin{Cor}\label{konsi}
Let $U\subseteq\mathbb{R}^d$ be a closed convex set that contains no lines. Let $x\in \mathbb{R}^d\setminus U$. Then there exist $e\in S^{d-1}$ and $\varepsilon>0$ such that $U$ lies in the cone $\left\{y\in\mathbb{R}^d\big|\,\langle e,\,y-x\rangle\geqslant\varepsilon|y-x|\right\}$.
\end{Cor}
\begin{proof}
By Lemma~\ref{pocvipcon}, the set $U$ is covered by a cone $C$ with vertex $x$ that contains no lines. Therefore, by Lemma~\ref{pocconconcon}, the cone $C$, and hence the set $U$, is covered by the required cone.
\end{proof}

\begin{Cor}\label{abots}
Let $U\subseteq\mathbb{R}^d$ be a closed convex set that contains no lines. Then there exist numbers $a>0,\,b\geqslant 0$ such that, for any collection of points $\{x_1,\dots,\,x_k\}$ and directions $\{\vec{e}_1,\dots,\,\vec{e}_r\}$ from $U$ and any convex combination of elements of the sets $\{x_1,\dots,\,x_k\}\cup\{\vec{e}_1,\dots,\,\vec{e}_r\}$, the inequality
\begin{equation}
a\left|\sum\limits_{j=1}^{k}\alpha_j x_j+\sum\limits_{i=1}^{r}\beta_i e_i\right|+b\geqslant \sum\limits_{j=1}^{k}\alpha_j |x_j|+\sum\limits_{i=1}^{r}\beta_i
\end{equation}
holds true.
\end{Cor}
\begin{proof}
Let $x\in \mathbb{R}^d\setminus U$. By Corollary~\ref{konsi}, there exist $e\in S^{d-1}$ and $\varepsilon>0$ such that $U$ is contained in the cone $C=\left\{y\in\mathbb{R}^d\big|\,\langle e,\,y-x\rangle\geqslant\varepsilon|y-x|\right\}$. Let $\vec{f}\in\vec{S}^{d-1}$ be a direction from $U$. Then $x+f\in C$ and $\langle e,\,f\rangle\geqslant\varepsilon$. We write the estimate for the absolute value of a convex combination:
\begin{multline}
\left|\sum\limits_{j=1}^{k}\alpha_j x_j+\sum\limits_{i=1}^{r}\beta_i e_i\right|
\geqslant
\left|\sum\limits_{j=1}^{k}\alpha_j (x_j-x)+\sum\limits_{i=1}^{r}\beta_i e_i\right|-|x|
\geqslant
\\
\geqslant
\left\langle e,\,\left(\sum\limits_{j=1}^{k}\alpha_j (x_j-x)+\sum\limits_{i=1}^{r}\beta_i e_i\right)\right\rangle-|x|
=
\sum\limits_{j=1}^{k}\alpha_j \langle e,\,x_j-x\rangle+\sum\limits_{i=1}^{r}\beta_i \langle e,\, e_i\rangle-|x|
\geqslant
\\
\geqslant
\sum\limits_{j=1}^{k}\alpha_j \varepsilon|x_j-x|+\sum\limits_{i=1}^{r}\beta_i \varepsilon-|x|
\geqslant
\sum\limits_{j=1}^{k}\alpha_j \varepsilon|x_j|+\sum\limits_{i=1}^{r}\beta_i \varepsilon-(1+\varepsilon)|x|.
\end{multline}
Thus,
\begin{equation}
\frac{1}{\varepsilon}\left|\sum\limits_{j=1}^{k}\alpha_j x_j+\sum\limits_{i=1}^{r}\beta_i e_i\right|+\frac{(1+\varepsilon)|x|}{\varepsilon}\geqslant \sum\limits_{j=1}^{k}\alpha_j |x_j|+\sum\limits_{i=1}^{r}\beta_i.
\end{equation}
\end{proof}
\begin{Cor}\label{abotsx}
Let $U\subseteq\mathbb{R}^d$ be a closed convex set that contains no lines. Then there exist numbers $a>0,\,b\geqslant 0$ such that, for any point $x\in U$ and any collection of points $\{x_1,\dots,\,x_k\}$ and directions $\{\vec{e}_1,\dots,\,\vec{e}_r\}$ from $U$ such that
\begin{equation}
x\in\conv(\{x_1,\dots,\,x_k\};\;\{\vec{e}_1,\dots,\,\vec{e}_r\}),
\end{equation}
the inequality
\begin{equation}
\min\left\{|x_j|\big|\,j=1,\dots,\,k\right\}
\leqslant
a|x|+b
\end{equation}
holds.
\end{Cor}
\begin{proof} Let us show that the numbers $a$ and $b$ from Corollary~\ref{abots} are suitable. Since
\begin{equation}
x\in\conv(\{x_1,\dots,\,x_k\};\;\{\vec{e}_1,\dots,\,\vec{e}_r\}),
\end{equation}
the point $x$ is representable as a convex combination of the points $\{x_1,\dots,\,x_k\}$ and directions $\{\vec{e}_1,\dots,\,\vec{e}_r\}$. Then by Corollary~\ref{abots} the following chain of inequalities:
\begin{equation}
a|x|+b
=
a\left|\sum\limits_{j=1}^{k}\alpha_j x_j+\sum\limits_{i=1}^{r}\beta_i e_i\right|+b
\geqslant
\sum\limits_{j=1}^{k}\alpha_j |x_j|+\sum\limits_{i=1}^{r}\beta_i
\geqslant
\sum\limits_{j=1}^{k}\alpha_j |x_j|
\geqslant
\min\left\{|x_j|\big|\,j=1,\dots,\,k\right\}
\end{equation}
holds.
\end{proof}

\begin{Le}\label{perescep}
Let $V\subseteq\mathbb{R}^d$, $E\subseteq\vec{S}^{d-1}$, and let $(V;\;E)$ be a closed pair of sets. Let $x\in\conv(V;\;E)$ and let $\{C_j\}_{j\in\mathbb{N}}$ be a decreasing sequence of generalized simplices such that, for any $j\in\mathbb{N}$, the inclusions $x\in C_{j}$, $V_{\fin}(C_{j})\subseteq V$, and $V_{\dir}(C_j)\subseteq E$ hold. Then there exists a generalized simplex $C$ such that $x\in C$, $V_{\fin}(C)\subseteq V$, $V_{\dir}(C)\subseteq E$, and $C\subseteq\bigcap_{j\in\mathbb{N}}C_j$.
\end{Le}
\begin{proof}
Passing to a subsequence of $\{C_j\}_{j\in\mathbb{N}}$, we may assume without loss of generality that all generalized simplices $C_j$ have the same dimension, the same number $k$ of finite vertices, and the same number $r$ of vertices at infinity. Let the finite vertices of the generalized simplex $C_j$ be the points $v_{1,j},\dots,\,v_{k,j}$, and let the vertices at infinity be the directions $\vec{e}_{1,j},\dots,\,\vec{e}_{r,j}$.

Since $(V; E)$ is a closed pair of sets, after passing to a subsequence of $\{C_j\}_{j\in\mathbb{N}}$ we may assume that, for any $i=1,\dots,\,k$, the sequence of points $\{v_{i,\,j}\}_{j\in\mathbb{N}}$ converges either to some point $v_i\in V$ or to some direction $\vec{f}_i\in E$, and, for any $t=1,\dots,\,r$, the sequence of directions $\{\vec{e}_{t,\,j}\}_{j\in\mathbb{N}}$ converges to some direction $\vec{e}_{t}\in E$. Corollary~\ref{abotsx} says that, for any $j$, the generalized simplex $C_j$ has a vertex $v_{i,\,j}$ such that $|v_{i,\,j}|\leqslant a|x|+b$, where $a,\,b$ are constructed from the simplex $C_1$. Therefore, among the limiting vertices there is at least one finite vertex $v_i$, and it also satisfies the estimate $|v_i|\leqslant a|x|+b$. Let there be $n$ finite limiting points, and, without loss of generality, let them be $v_1,\dots,\,v_n$. That is, for $i\leqslant n$ we have convergence $v_{i,\,j}\longrightarrow v_i$, and for $i\geqslant n+1$ we have convergence $v_{i,\,j}\longrightarrow \vec{f}_i$. Since for $j'\leqslant j$ the inclusion $v_{i,\,j}\in C_{j'}$ holds, we have $v_i\in C_j$, and the direction $\vec{f}_i$ is a direction from $C_j$ for any $j$. In particular, the ray $[v_1,\,\vec{f}_i)$ lies in the intersection of all simplices of our sequence. Similarly, the direction $\vec{e}_i$ lies in the simplex $C_j$ for any $j$, and the ray $[v_1,\,\vec{e}_i)$ lies in the intersection of the simplices. Therefore, by the convexity of the intersection of convex sets, we obtain the inclusion
\begin{equation}
\conv(\{v_1,\dots,\,v_n\};\;\{\vec{f}_{n+1},\dots,\,\vec{f}_k,\,\vec{e}_1,\dots,\,\vec{e}_r\})\subseteq\bigcap\limits_{j\in\mathbb{N}}C_j.
\end{equation}
It remains to show that 
\begin{equation}\label{vklvlemper}
x\in\conv(\{v_1,\dots,\,v_n\};\;\{\vec{f}_{n+1},\dots,\,\vec{f}_k,\,\vec{e}_1,\dots,\,\vec{e}_r\}),
\end{equation}
then by Carath\'eodory's theorem, Theorem~\ref{carat}, we find the generalized simplex we need.

Since $x\in C_j$, the point $x$ is representable as a convex combination of elements of the sets 
\begin{equation}
\{v_{1,\,j},\dots,\,v_{k,\,j}\}\cap\{\vec{e}_{1,\,j},\dots,\,\vec{e}_{r,\,j}\}
\end{equation}
and there exist numbers $\alpha_{1,\,j},\dots,\,\alpha_{k,\,j},\,\beta_{1,\,j},\dots,\,\beta_{r,\,j}\geqslant 0$ such that $\sum\limits_{i=1}^{k}\alpha_{i,\,j}=1$ and
\begin{equation}
\sum\limits_{i=1}^{k}\alpha_{i,\,j}v_{i,\,j}+\sum\limits_{t=1}^{r}\beta_{t,\,j}e_{t,\,j}=x.
\end{equation}
Since $C_j\subseteq C_1$, Corollary~\ref{abots} gives the estimate
\begin{equation}\label{abvpr}
\sum\limits_{i=1}^{k}\alpha_{i,\,j}|v_{i,\,j}|+\sum\limits_{t=1}^{r}\beta_{t,\,j}\leqslant a|x|+b.
\end{equation}
It follows from this inequality that $\beta_{t,\,j}\leqslant a|x|+b$ and $\alpha_{i,\,j}|v_{i,\,j}|\leqslant a|x|+b$. Consequently, we can pass to a subsequence $\{C_j\}_{j\in\mathbb{N}}$ such that, for any $i=1,\dots,\,k$, the sequence of numbers $\{\alpha_{i,\,j}\}_{j\in\mathbb{N}}$ converges to some number $\alpha_i$, the sequence of numbers $\{\alpha_{i,\,j}|v_{i,\,j}|\}_{j\in\mathbb{N}}$ converges to some number $\gamma_i$, and, for any $t=1,\dots,\,r$, the sequence of numbers $\{\beta_{t,\,j}\}_{j\in\mathbb{N}}$ converges to some number $\beta_t$.
Then one can see that
\begin{equation}
x=\sum\limits_{i=1}^{k}\alpha_{i,\,j}v_{i,\,j}+\sum\limits_{t=1}^{r}\beta_{t,\,j}e_{t,\,j}\underset{j\to+\infty}{\longrightarrow}
\sum\limits_{i=1}^{n}\alpha_{i}v_{i}+\sum\limits_{i=n+1}^{k}\gamma_{i}f_{i}+\sum\limits_{t=1}^{r}\beta_{t}e_{t}.
\end{equation}
This proves the inclusion~\eqref{vklvlemper}.
\end{proof}
\begin{Cor}\label{tsorsim}
Let $(V;\;E)$ be a closed pair of sets and let $x\in \conv(V;\;E)$. Then, among the generalized simplices $C$ such that $x\in C$, $V_{\fin}(C)\subseteq V$, and $V_{\dir}(C)\subseteq E$, there is an inclusion-minimal one. 
\end{Cor}
\begin{proof}
It suffices to show that the set of generalized simplices $C$ such that $x\in C$, $V_{\fin}(C)\subseteq V$, and $V_{\dir}(C)\subseteq E$ is nonempty and, with the order $C\leqslant C'$ whenever $C'\subseteq C$, satisfies Zorn's lemma. By Lemma~\ref{dosschcsim}, it suffices to consider chains indexed by natural numbers. Such chains have an upper bound by Lemma~\ref{perescep}.
\end{proof}
\begin{Rem}
For Lemma~\ref{perescep} and Corollary~\ref{tsorsim}, it suffices to require that the sets $V$ and $E$ be closed separately, rather than as a pair of sets.
\end{Rem}
For this, it suffices to note that the directions $\vec{f}_{n+1},\dots,\,\vec{f}_k$ from the proof of Lemma~\ref{perescep} lie in the simplices $C_j$; hence, the vectors $f_{n+1},\dots,\,f_k$ are representable as sums of the vectors $e_{1,\,j},\dots,\,e_{r,\,j}$ with nonnegative coefficients; moreover, by Corollary~\ref{abots}, these coefficients are uniformly bounded. Passing to the limit, we obtain that $f_{n+1},\dots,\,f_k$ are represented as sum vectors $e_{1},\dots,\,e_{r}$ with nonnegative coefficients. Hence,
\begin{equation}
\conv(\{v_1,\dots,\,v_n\};\;\{\vec{f}_{n+1},\dots,\,\vec{f}_k,\,\vec{e}_1,\dots,\,\vec{e}_r\})
=
\conv(\{v_1,\dots,\,v_n\};\;\{\vec{e}_1,\dots,\,\vec{e}_r\}).
\end{equation}

\part{The hereditary property of minimal locally concave functions.}\label{prinnasl}
We allow concave functions to attain the values $-\infty$ and $+\infty$, following Rockafellar~\cite{Rockafellar}. A concave function is a function whose subgraph is convex (equivalently, $(-1)$ times this function is convex in the sense of Rockafellar). 
\section{Piecewise definition of concave functions on a segment.}\label{seckusnaotr}
\begin{Le}\label{lepo}
Let $[a,\,b]\subseteq\mathbb{R}^d$ be a segment, $c\in [a,\,b]$. Let $B\colon [a,\,b]\to\re$ be a concave function. Let $B^{\circ}\colon [a,\,c]\to\re$ be a concave function such that $B(c)=B^{\circ}(c)$, and for any point $x\in[a,\,c)$ the inequality $B^{\circ}(x)\leqslant B(x)$ holds. Then the function $B'\colon[a,\,b]\to\re$ defined by the formula
\begin{equation}
B'(x)=
\begin{cases}
B(x),&x\in [c,\,b],\\
B^{\circ}(x),&x\in [a,\,c],
\end{cases}
\end{equation}
is concave.
\end{Le}
\begin{figure}[h]\centering
\def\svgwidth{17cm}
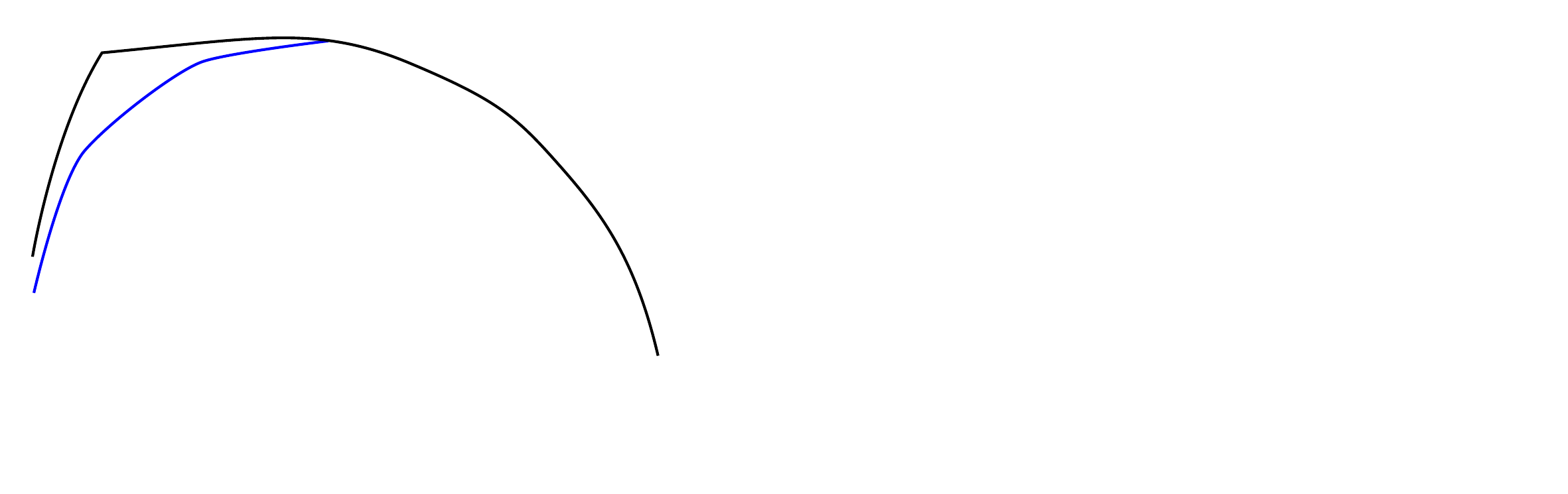
\caption{Illustration for Lemma~\ref{lepo}}
\end{figure}
\begin{proof}
We need to check that, for any $x,\,y\in[a,\,b]$, $x<y$, and $\alpha\in[0,\,1]$, the inequality $B'(\alpha x+(1-\alpha)y)\geqslant\alpha B'(x)+(1-\alpha)B'(y)$ holds. If $\alpha x+(1-\alpha)y\geqslant c$, then
\begin{equation}\label{bc}
B'(\alpha x+(1-\alpha)y)=B(\alpha x+(1-\alpha)y)\geqslant\alpha B(x)+(1-\alpha)B(y)\geqslant \alpha B'(x)+(1-\alpha)B'(y).
\end{equation}
If $y\leqslant c$, then
\begin{equation}\label{mc}
B'(\alpha x+(1-\alpha)y)=B^{\circ}(\alpha x+(1-\alpha)y)\geqslant\alpha B^{\circ}(x)+(1-\alpha)B^{\circ}(y)= \alpha B'(x)+(1-\alpha)B'(y).
\end{equation}
It remains to consider the case where $\alpha x+(1-\alpha)y<c<y$. In this case, let $\beta\in(0,\,1)$ be such that $c=\beta x+(1-\beta)y$. Then
\begin{equation}
\alpha x+(1-\alpha)y=\frac{\alpha-\beta}{1-\beta}x+\frac{1-\alpha}{1-\beta}c
\end{equation}
and
\begin{multline}
B'(\alpha x+(1-\alpha)y)\overset{\eqref{mc}}{\geqslant}\frac{\alpha-\beta}{1-\beta}B'(x)+\frac{1-\alpha}{1-\beta}B'(c)\overset{\eqref{bc}}{\geqslant}\frac{\alpha-\beta}{1-\beta}B'(x)+\frac{1-\alpha}{1-\beta}(\beta B'(x)+(1-\beta)B'(y))
=
\\
=
\alpha B'(x)+(1-\alpha)B'(y).
\end{multline}
\end{proof}
\begin{Rem}\label{prpo}
An analogous lemma, with a similar proof, holds for a function $B^{\circ}$ defined on the right half-segment instead of the left one, that is, on $[c,\,b]$ rather than on $[a,\,c]$.
\end{Rem}
\begin{figure}[h]\centering
\def\svgwidth{17cm}
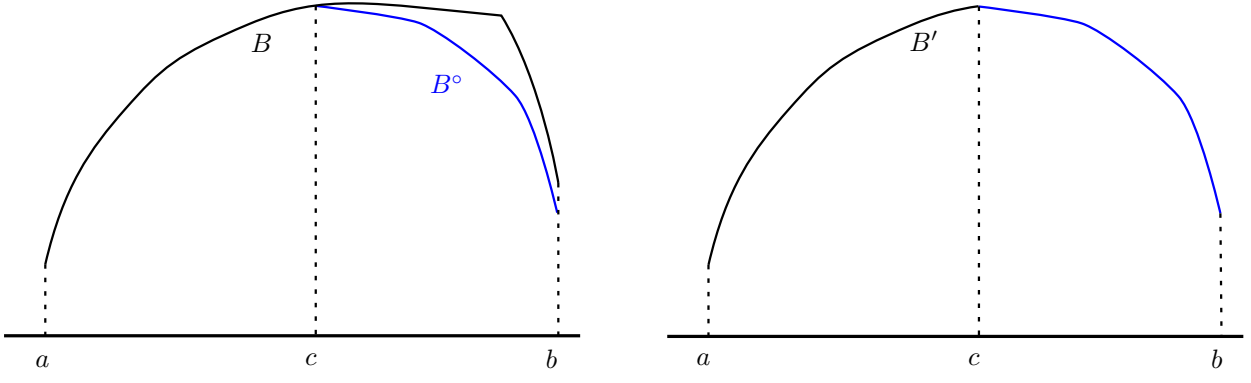
\caption{Illustration for Remark~\ref{prpo}}
\end{figure}
\begin{Le}\label{cer}
Let $[a,\,b]\subseteq\mathbb{R}$ be a segment, $c,\,d\in [a,\,b]$ and $c<d$. Let $B\colon [a,\,b]\to\re$ be a concave function. Let $B^{\circ}\colon [c,\,d]\to\re$ be a concave function such that $B(c)=B^{\circ}(c)$ and $B(d)=B^{\circ}(d)$, and for any $x\in(c,\,d)$ the inequality $B^{\circ}(x)\leqslant B(x)$ holds. Then the function $B'\colon[a,\,b]\to\re$ defined by the formula
\begin{equation}
B'(x)=
\begin{cases}
B(x),&x\notin [c,\,d],\\
B^{\circ}(x),&x\in [c,\,d],
\end{cases}
\end{equation}
is concave.
\end{Le}
\begin{figure}[h]\centering
\def\svgwidth{17cm}
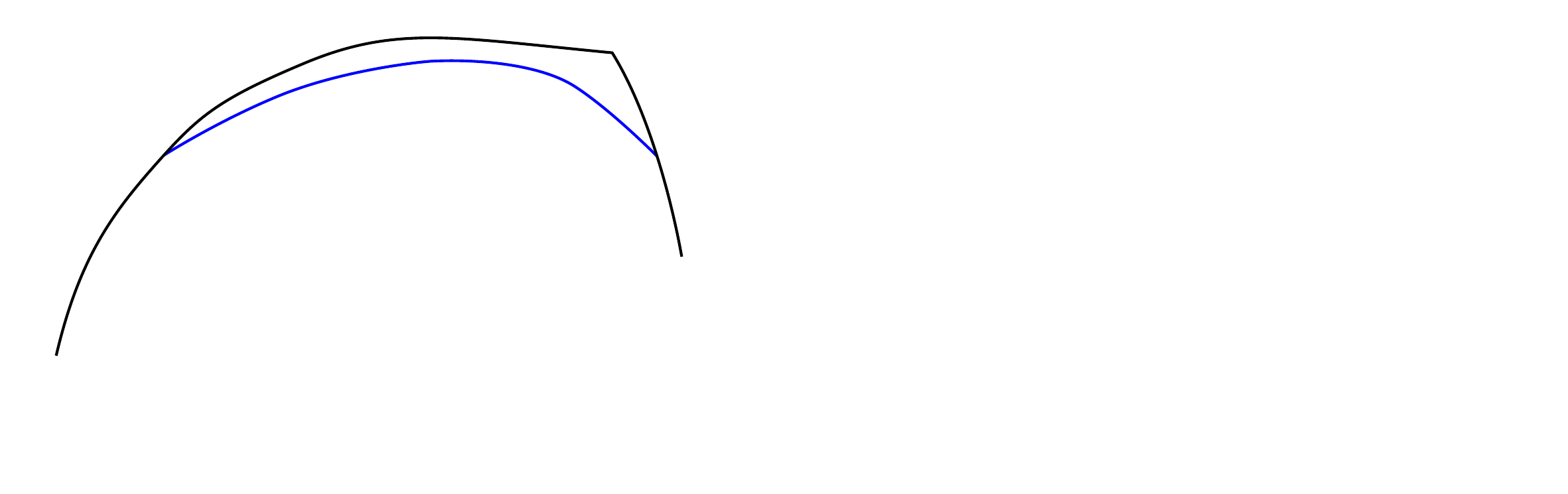
\caption{Illustration for Lemma~\ref{cer}}
\end{figure}
\begin{proof}
By Lemma~\ref{lepo}, the function $B'|_{[c,\,b]}$ is concave; hence, by Remark~\ref{prpo}, the function $B'$ is concave as well.
\end{proof}

\begin{Le}\label{kusotrpolob}
Let $[a,\,b]\subseteq\mathbb{R}$ be a segment, $V\subseteq [a,\,b]$ be a subset of the segment such that for any $x,\,y\in[a,\,b]$, if $(x,\,y)\subseteq V$, then $[x,\,y]\subseteq V$. Let $B\colon [a,\,b]\to\re$ be a concave function. Let $B^{\circ}\colon V\to\re$ be a locally concave function such that $B|_{(\partial V\setminus \{a,\,b\})\cap V}=B^{\circ}|_{(\partial V\setminus \{a,\,b\})\cap V}$, and for any element $x\in V$ the inequality $B^{\circ}(x)\leqslant B(x)$ holds. Then the function $B'\colon[a,\,b]\to\re$ defined by the formula
\begin{equation}\label{kusfor}
B'(x)=
\begin{cases}
B(x),&x\notin V,\\
B^{\circ}(x),&x\in V,
\end{cases}
\end{equation}
is concave.
\end{Le}
\begin{Rem}\label{zampootequsle}
The condition on the set $V$ saying that for any $x,\,y\in[a,\,b]$, if $(x,\,y)\subseteq V$, then $[x,\,y]\subseteq V$, means that $V$ is closed with respect to segments in $[a,\,b]$ \textup{(}see Definition~\ref{closeotr}\textup{)}. The equality $B|_{(\partial V\setminus \{a,\,b\})\cap V}=B^{\circ}|_{(\partial V\setminus \{a,\,b\})\cap V}$ means that the functions coincide on the inner boundary of $V$ with respect to segments in $[a,\,b]$ \textup{(}see Definition~\ref{bounotr}\textup{)}. 
\end{Rem}

\begin{figure}[h]\centering
\def\svgwidth{17cm}
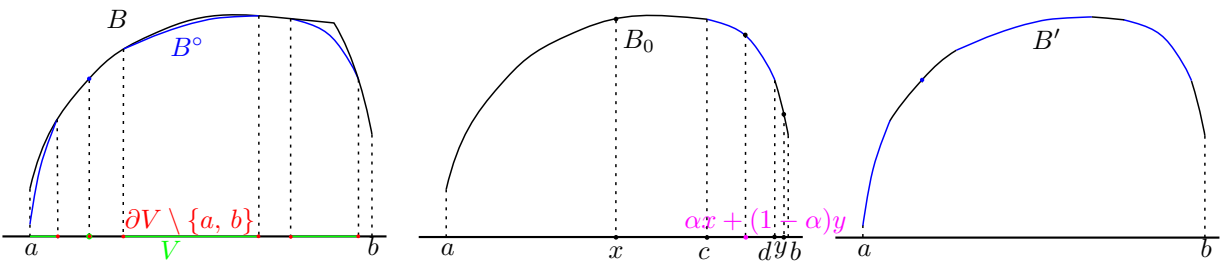
\caption{Illustration for Lemma~\ref{kusotrpolob}}
\end{figure}
\begin{proof}
Let $x,\,y\in [a,\,b]$, $x<y$, and $\alpha\in[0,\,1]$. We need to check the inequality
\begin{equation}\label{nervogkuszad}
B'(\alpha x+(1-\alpha)y)\geqslant\alpha B'(x)+(1-\alpha)B'(y).
\end{equation}
We know that 
\begin{equation}
B(\alpha x+(1-\alpha)y)\geqslant\alpha B(x)+(1-\alpha)B(y)\geqslant\alpha B'(x)+(1-\alpha)B'(y).
\end{equation}
Thus, if $B'(\alpha x+(1-\alpha)y)=B(\alpha x+(1-\alpha)y)$, then the desired inequality~\eqref{nervogkuszad} holds. Therefore, the estimate certainly holds if $\alpha x+(1-\alpha)y\in [a,\,b]\setminus V$ or $\alpha x+(1-\alpha)y\in(\partial V\setminus\{a,\,b\})\cap V$. If $\alpha x+(1-\alpha)y\in\{a,\,b\}$, then the inequality holds automatically, since this means either $x=y$ or $\alpha\in\{0,\,1\}$.

Thus, it remains to consider the case $\alpha x+(1-\alpha)y\in\intX V$. Let $c$ be the point of $\partial V\cap [a,\,\alpha x+(1-\alpha)y]$ xlosest to $\alpha x+(1-\alpha)y$. Similarly, let $d$ be the point of $\partial V\cap [\alpha x+(1-\alpha)y,\,b]$ xlosest to $\alpha x+(1-\alpha)y$. Then the intervals $(c,\,\alpha x+(1-\alpha)y)$ and $(\alpha x+(1-\alpha)y,\,d)$ lie in $V$, and hence the points $c,\,d$ lie in $V$. In other words, $c,\,d\in V\cap\partial V$. Thus, either $c=a$ or $B(c)=B^{\circ}(c)$, and either $d=b$ or $B(d)=B^{\circ}(d)$. Let the function $B_0\colon [x,\,y]\to\re$ be defined by the formula
\begin{equation}
B_0(z)=
\begin{cases}
B(z),&z\notin [c,\,d],\\
B^{\circ}(z),&z\in [c,\,d].
\end{cases}
\end{equation}
Then, by Lemmas~\ref{lepo} and~\ref{cer} and Remark~\ref{prpo}, the function $B_0$ is concave; hence, 
\begin{equation}
B'(\alpha x+(1-\alpha)y)=B_0(\alpha x+(1-\alpha)y)\geqslant \alpha B_0(x)+(1-\alpha)B_0(y)\geqslant \alpha B'(x)+(1-\alpha)B'(y).
\end{equation}
\end{proof}

\begin{Cor}\label{kusotr}
Let $[a,\,b]\subseteq\mathbb{R}^d$ be a segment, and let $V\subseteq [a,\,b]$ be a closed set. Let $B\colon [a,\,b]\to\re$ be a concave function. Let $B^{\circ}\colon V\to\re$ be a locally concave function such that $B|_{\partial V\setminus \{a,\,b\}}=B^{\circ}|_{\partial V\setminus \{a,\,b\}}$, and for any element $x\in V$ the inequality $B^{\circ}(x)\leqslant B(x)$ holds. Then the function $B'\colon[a,\,b]\to\re$ defined by the formula
\begin{equation}\label{kusfor}
B'(x)=
\begin{cases}
B(x),&x\notin V,\\
B^{\circ}(x),&x\in V,
\end{cases}
\end{equation}
is concave.
\end{Cor}

\section{On closedness with respect to segments in a linear space.}\label{secclpootr}
Let $X$ be a vector space over $\mathbb{R}$. Let $U\subseteq X$ and $V\subseteq U$.
\begin{Def}\label{closeotr}
We say that the set $V$ is closed with respect to segments in $U$ if, for all $x,\,y\in U$ such that $(x,\,y)\subseteq V$, the inclusion $x,\,y\in V$ holds.
\end{Def}
\begin{Def}\label{grintdef}
Let $a,\,b\in X$. The boundary of the set $V$ along the interval $(a,\,b)$ is the set of points of $(a,\,b)$ that lie in $V$ and in the boundary of $(a,\,b)\cap V$ in the topology of the line $\aff (\{a,\,b\})$:
\begin{equation}
\partial_a^b V=(V\cap (a,\,b))\cap (\partial_{\aff(\{a,\,b\})}(V\cap (a,\,b))).
\end{equation}
\end{Def}
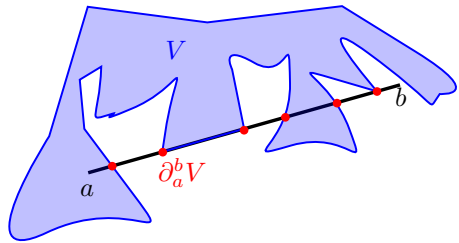
\begin{figure}[h]\centering
\def\svgwidth{6cm}
%% Creator: Inkscape 1.3.2 (091e20e, 2023-11-25, custom), www.inkscape.org
%% Accompanies image file 'grvint.pdf' (pdf, eps, ps)
\begingroup%
  \makeatletter%
  \providecommand\color[2][]{%
    \errmessage{(Inkscape) Color is used for the text in Inkscape, but the package 'color.sty' is not loaded}%
    \renewcommand\color[2][]{}%
  }%
  \providecommand\transparent[1]{%
    \errmessage{(Inkscape) Transparency is used (non-zero) for the text in Inkscape, but the package 'transparent.sty' is not loaded}%
    \renewcommand\transparent[1]{}%
  }%
  \providecommand\rotatebox[2]{#2}%
  \newcommand*\fsize{\dimexpr\f@size pt\relax}%
  \newcommand*\lineheight[1]{\fontsize{\fsize}{#1\fsize}\selectfont}%
  \ifx\svgwidth\undefined%
    \setlength{\unitlength}{263.62204724bp}%
    \ifx\svgscale\undefined%
      \relax%
    \else%
      \setlength{\unitlength}{\unitlength * \real{\svgscale}}%
    \fi%
  \else%
    \setlength{\unitlength}{\svgwidth}%
  \fi%
  \global\let\svgwidth\undefined%
  \global\let\svgscale\undefined%
  \makeatother%
  \begin{picture}(1,0.53763441)%
    \lineheight{1}%
    \setlength\tabcolsep{0pt}%
    \put(0,0){\includegraphics[width=\unitlength,page=1]{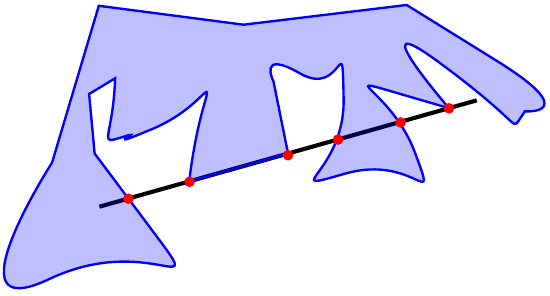}}%
    \put(0.1635056,0.11063241){\color[rgb]{0,0,0}\makebox(0,0)[lt]{\lineheight{1.25}\smash{\begin{tabular}[t]{l}$a$\end{tabular}}}}%
    \put(0.85740235,0.30743937){\color[rgb]{0,0,0}\makebox(0,0)[lt]{\lineheight{1.25}\smash{\begin{tabular}[t]{l}$b$\end{tabular}}}}%
    \put(0.35118234,0.41294409){\color[rgb]{0,0,0}\makebox(0,0)[lt]{\lineheight{1.25}\smash{\begin{tabular}[t]{l}\textcolor[rgb]{0,0,1}{$V$}\end{tabular}}}}%
    \put(0.33495085,0.14410987){\color[rgb]{0,0,0}\makebox(0,0)[lt]{\lineheight{1.25}\smash{\begin{tabular}[t]{l}\textcolor[rgb]{1,0,0}{$\partial_a^b V$}\end{tabular}}}}%
  \end{picture}%
\endgroup%

\caption{Illustration for Definition~\ref{grintdef}}
\end{figure}
\begin{Def}\label{bounotr}
The inner boundary of $V$ with respect to segments in $U$ is the union of the boundaries of $V$ along the intervals $(a,\,b)$ such that $[a,\,b]\subseteq U$:
\begin{equation}
\partial_U^s V=\underset{[a,\,b]\subseteq U}{\bigcup}\partial_a^b V.
\end{equation}
\end{Def}
In other words, the set $\partial_U^s V$ consists of the points of $V$ for which there exists an interval $(a,\,b)$ in $U$ such that the given point is a boundary point of $V$ inside this interval.
\begin{Rem}
In the case where the ambient set $U=[a,\,b]$ is a segment, the condition that $V$ is closed with respect to segments means that $V$ is the disjoint union of at most countably many closed segments $([c_j,\,d_j])_j$ and a set $Z$ with empty interior. The boundary of $V$ with respect to segments in $U$ in this case is the following set:
\begin{equation}
\partial_{[a,\,b]}^s V=\left(Z\cup\textstyle{\bigcup_j}\{c_j,\,d_j\}\right)\setminus\{a,\,b\}.
\end{equation} 
\end{Rem}

\begin{Ex}\label{vngr}
Let $U=\clos\Omega_0\setminus\Omega_1$, where the sets $\Omega_0,\,\Omega_1\subseteq\mathbb{R}^d$ are open and convex, $\Omega_0$ is nonempty, and $\clos\Omega_1\subseteq\Omega_0$. Let $V$ be a face of $\clos\Omega_0$ that does not coincide with $\clos\Omega_0$. Then $\partial_U^s V=\varnothing$.
\end{Ex}
\begin{figure}[h]\centering
\def\svgwidth{8cm}
%% Creator: Inkscape 1.3.2 (091e20e, 2023-11-25, custom), www.inkscape.org
%% Accompanies image file 'exvngr.pdf' (pdf, eps, ps)
\begingroup%
  \makeatletter%
  \providecommand\color[2][]{%
    \errmessage{(Inkscape) Color is used for the text in Inkscape, but the package 'color.sty' is not loaded}%
    \renewcommand\color[2][]{}%
  }%
  \providecommand\transparent[1]{%
    \errmessage{(Inkscape) Transparency is used (non-zero) for the text in Inkscape, but the package 'transparent.sty' is not loaded}%
    \renewcommand\transparent[1]{}%
  }%
  \providecommand\rotatebox[2]{#2}%
  \newcommand*\fsize{\dimexpr\f@size pt\relax}%
  \newcommand*\lineheight[1]{\fontsize{\fsize}{#1\fsize}\selectfont}%
  \ifx\svgwidth\undefined%
    \setlength{\unitlength}{575.43307087bp}%
    \ifx\svgscale\undefined%
      \relax%
    \else%
      \setlength{\unitlength}{\unitlength * \real{\svgscale}}%
    \fi%
  \else%
    \setlength{\unitlength}{\svgwidth}%
  \fi%
  \global\let\svgwidth\undefined%
  \global\let\svgscale\undefined%
  \makeatother%
  \begin{picture}(1,0.62068966)%
    \lineheight{1}%
    \setlength\tabcolsep{0pt}%
    \put(0,0){\includegraphics[width=\unitlength,page=1]{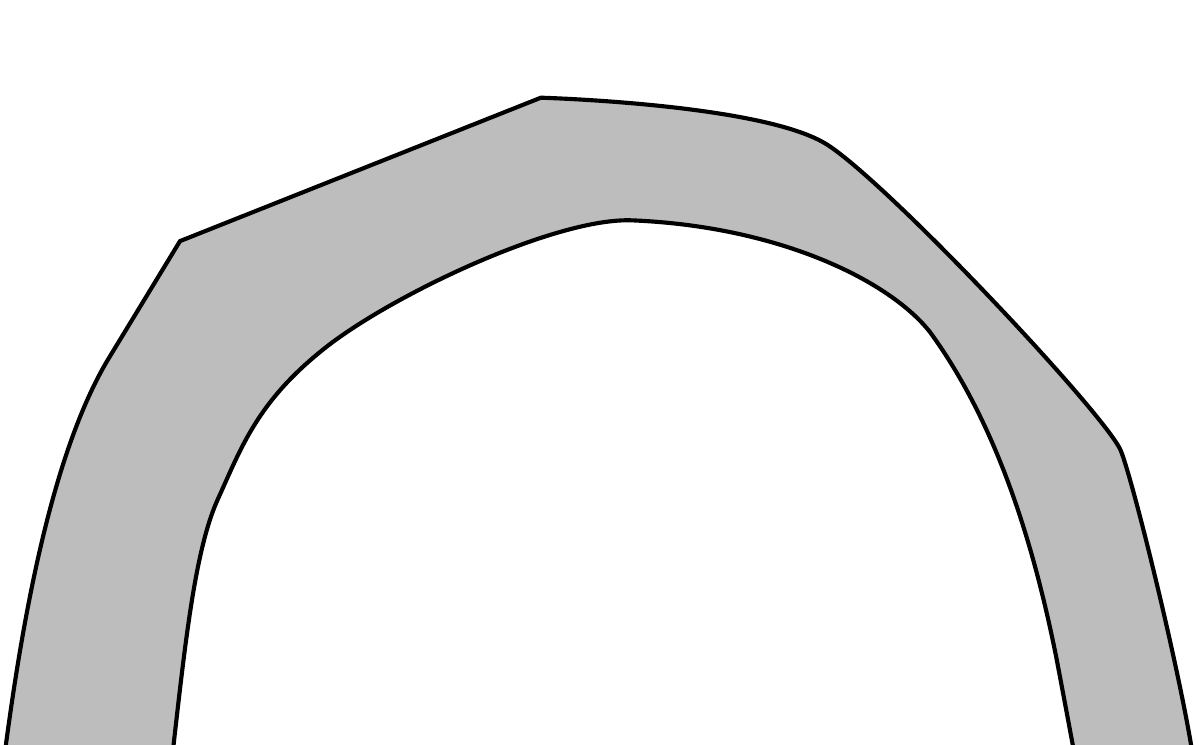}}%
    \put(0.10150909,0.23290313){\color[rgb]{0,0,0}\makebox(0,0)[lt]{\lineheight{1.25}\smash{\begin{tabular}[t]{l}$U$\end{tabular}}}}%
    \put(0,0){\includegraphics[width=\unitlength,page=2]{exvngr.pdf}}%
    \put(0.20535699,0.47740584){\color[rgb]{1,0,0}\makebox(0,0)[lt]{\lineheight{1.25}\smash{\begin{tabular}[t]{l}\textcolor[rgb]{1,0,0}{$V$}\end{tabular}}}}%
    \put(0.0436698,0.55890672){\color[rgb]{0,0,1}\makebox(0,0)[lt]{\lineheight{1.25}\smash{\begin{tabular}[t]{l}\textcolor[rgb]{0,0,1}{$\partial_U^s V=\varnothing$}\end{tabular}}}}%
  \end{picture}%
\endgroup%

\caption{Illustration for Example~\ref{vngr}}
\end{figure}

\begin{Ex}\label{vnegr}
Let $U=\clos\Omega_0\setminus\Omega_1$, where the sets $\Omega_0,\,\Omega_1\subseteq\mathbb{R}^d$ are open and convex, $\Omega_0$ is nonempty, and $\clos\Omega_1\subseteq\Omega_0$. Let $V$ be a face of $\clos\Omega_1$ of dimension $d-1$. Then $\partial_U^s V=\rb V$.
\end{Ex}
\begin{figure}[h]\centering
\def\svgwidth{8cm}
%% Creator: Inkscape 1.3.2 (091e20e, 2023-11-25, custom), www.inkscape.org
%% Accompanies image file 'exvnegr.pdf' (pdf, eps, ps)
\begingroup%
  \makeatletter%
  \providecommand\color[2][]{%
    \errmessage{(Inkscape) Color is used for the text in Inkscape, but the package 'color.sty' is not loaded}%
    \renewcommand\color[2][]{}%
  }%
  \providecommand\transparent[1]{%
    \errmessage{(Inkscape) Transparency is used (non-zero) for the text in Inkscape, but the package 'transparent.sty' is not loaded}%
    \renewcommand\transparent[1]{}%
  }%
  \providecommand\rotatebox[2]{#2}%
  \newcommand*\fsize{\dimexpr\f@size pt\relax}%
  \newcommand*\lineheight[1]{\fontsize{\fsize}{#1\fsize}\selectfont}%
  \ifx\svgwidth\undefined%
    \setlength{\unitlength}{575.43307087bp}%
    \ifx\svgscale\undefined%
      \relax%
    \else%
      \setlength{\unitlength}{\unitlength * \real{\svgscale}}%
    \fi%
  \else%
    \setlength{\unitlength}{\svgwidth}%
  \fi%
  \global\let\svgwidth\undefined%
  \global\let\svgscale\undefined%
  \makeatother%
  \begin{picture}(1,0.39408867)%
    \lineheight{1}%
    \setlength\tabcolsep{0pt}%
    \put(0,0){\includegraphics[width=\unitlength,page=1]{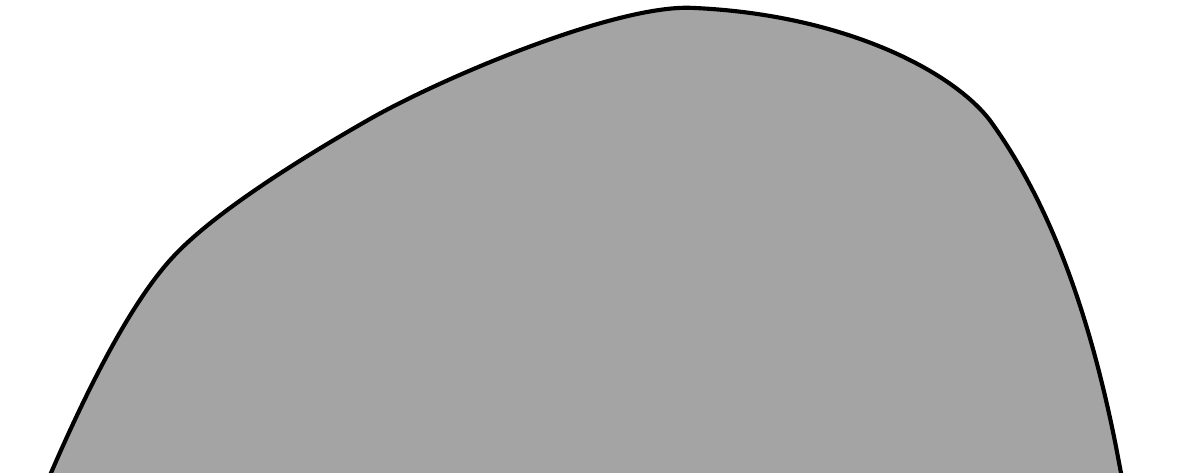}}%
    \put(0.26713994,0.22970265){\color[rgb]{0,0,0}\makebox(0,0)[lt]{\lineheight{1.25}\smash{\begin{tabular}[t]{l}$U$\end{tabular}}}}%
    \put(0,0){\includegraphics[width=\unitlength,page=2]{exvnegr.pdf}}%
    \put(0.24347839,0.11139491){\color[rgb]{1,0,0}\makebox(0,0)[lt]{\lineheight{1.25}\smash{\begin{tabular}[t]{l}\textcolor[rgb]{1,0,0}{$V$}\end{tabular}}}}%
    \put(0.41699647,0.181065){\color[rgb]{0,0,1}\makebox(0,0)[lt]{\lineheight{1.25}\smash{\begin{tabular}[t]{l}\textcolor[rgb]{0,0,1}{$\partial_U^s V$}\end{tabular}}}}%
    \put(0,0){\includegraphics[width=\unitlength,page=3]{exvnegr.pdf}}%
  \end{picture}%
\endgroup%

\caption{Illustration for Example~\ref{vnegr}}
\end{figure}

\begin{Ex}\label{extfor}
Let $U=\clos\Omega_0\setminus\Omega_1$, where the sets $\Omega_0,\,\Omega_1\subseteq\mathbb{R}^2$ are open and convex, $\Omega_0$ is nonempty and strictly convex, and $\clos\Omega_1\subseteq\Omega_0$. Let the set $V$ be the one shown in Fig.~\ref{extpic}. Then $\partial_U^s V$ is the union of the two half-intervals $a,\,b$ shown in Fig.~\ref{extpic}.
\end{Ex}
\begin{figure}[h]\centering
\def\svgwidth{8cm}
%% Creator: Inkscape 1.3.2 (091e20e, 2023-11-25, custom), www.inkscape.org
%% Accompanies image file 'ext.pdf' (pdf, eps, ps)
\begingroup%
  \makeatletter%
  \providecommand\color[2][]{%
    \errmessage{(Inkscape) Color is used for the text in Inkscape, but the package 'color.sty' is not loaded}%
    \renewcommand\color[2][]{}%
  }%
  \providecommand\transparent[1]{%
    \errmessage{(Inkscape) Transparency is used (non-zero) for the text in Inkscape, but the package 'transparent.sty' is not loaded}%
    \renewcommand\transparent[1]{}%
  }%
  \providecommand\rotatebox[2]{#2}%
  \newcommand*\fsize{\dimexpr\f@size pt\relax}%
  \newcommand*\lineheight[1]{\fontsize{\fsize}{#1\fsize}\selectfont}%
  \ifx\svgwidth\undefined%
    \setlength{\unitlength}{575.43307087bp}%
    \ifx\svgscale\undefined%
      \relax%
    \else%
      \setlength{\unitlength}{\unitlength * \real{\svgscale}}%
    \fi%
  \else%
    \setlength{\unitlength}{\svgwidth}%
  \fi%
  \global\let\svgwidth\undefined%
  \global\let\svgscale\undefined%
  \makeatother%
  \begin{picture}(1,0.39408867)%
    \lineheight{1}%
    \setlength\tabcolsep{0pt}%
    \put(0,0){\includegraphics[width=\unitlength,page=1]{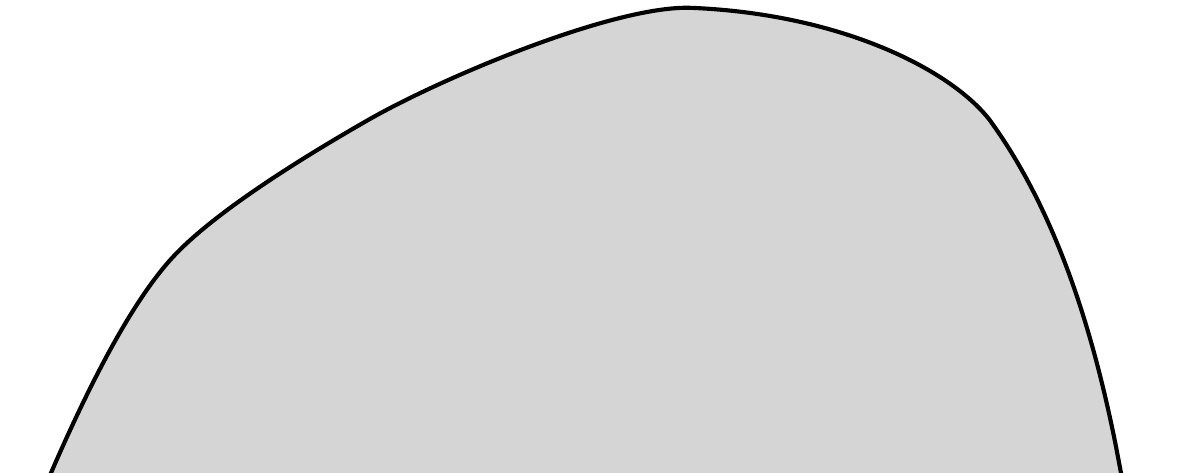}}%
    \put(0.71933848,0.26256591){\color[rgb]{0,0,0}\makebox(0,0)[lt]{\lineheight{1.25}\smash{\begin{tabular}[t]{l}$U$\end{tabular}}}}%
    \put(0,0){\includegraphics[width=\unitlength,page=2]{ext.pdf}}%
    \put(0.29211602,0.2302422){\color[rgb]{1,0,0}\makebox(0,0)[lt]{\lineheight{1.25}\smash{\begin{tabular}[t]{l}\textcolor[rgb]{1,0,0}{$V$}\end{tabular}}}}%
    \put(0,0){\includegraphics[width=\unitlength,page=3]{ext.pdf}}%
    \put(0.198469,0.18040068){\color[rgb]{0,0,1}\makebox(0,0)[lt]{\lineheight{1.25}\smash{\begin{tabular}[t]{l}\textcolor[rgb]{0,0,1}{$a$}\end{tabular}}}}%
    \put(0.57730215,0.31258778){\color[rgb]{0,0,1}\makebox(0,0)[lt]{\lineheight{1.25}\smash{\begin{tabular}[t]{l}\textcolor[rgb]{0,0,1}{$b$}\end{tabular}}}}%
    \put(0.27108359,0.05355553){\color[rgb]{0,0,1}\makebox(0,0)[lt]{\lineheight{1.25}\smash{\begin{tabular}[t]{l}\textcolor[rgb]{0,0,1}{$\partial_U^s V=a \cup b$}\end{tabular}}}}%
  \end{picture}%
\endgroup%

\caption{Illustration for Example~\ref{extfor}}\label{extpic}
\end{figure}

\begin{Ex}\label{exhal}
Let $X$ be a linear space over the field $\mathbb{R}$, and let $Y$ be its affine subspace of codimension $1$. Then $Y$ divides $X$ into two closed half-spaces $Y^+$ and $Y^-$. Let $U\subseteq X$, and let $V$ be one of the halves into which $Y$ divides $U$; in other words, $V=U\cap(Y\cup Y^+)$. Then $\partial_U^s V\subseteq U\cap Y$.
\end{Ex}
\begin{figure}[h]\centering
\def\svgwidth{6cm}
%% Creator: Inkscape 1.3.2 (091e20e, 2023-11-25, custom), www.inkscape.org
%% Accompanies image file 'hal.pdf' (pdf, eps, ps)
\begingroup%
  \makeatletter%
  \providecommand\color[2][]{%
    \errmessage{(Inkscape) Color is used for the text in Inkscape, but the package 'color.sty' is not loaded}%
    \renewcommand\color[2][]{}%
  }%
  \providecommand\transparent[1]{%
    \errmessage{(Inkscape) Transparency is used (non-zero) for the text in Inkscape, but the package 'transparent.sty' is not loaded}%
    \renewcommand\transparent[1]{}%
  }%
  \providecommand\rotatebox[2]{#2}%
  \newcommand*\fsize{\dimexpr\f@size pt\relax}%
  \newcommand*\lineheight[1]{\fontsize{\fsize}{#1\fsize}\selectfont}%
  \ifx\svgwidth\undefined%
    \setlength{\unitlength}{277.79527559bp}%
    \ifx\svgscale\undefined%
      \relax%
    \else%
      \setlength{\unitlength}{\unitlength * \real{\svgscale}}%
    \fi%
  \else%
    \setlength{\unitlength}{\svgwidth}%
  \fi%
  \global\let\svgwidth\undefined%
  \global\let\svgscale\undefined%
  \makeatother%
  \begin{picture}(1,0.78571429)%
    \lineheight{1}%
    \setlength\tabcolsep{0pt}%
    \put(0,0){\includegraphics[width=\unitlength,page=1]{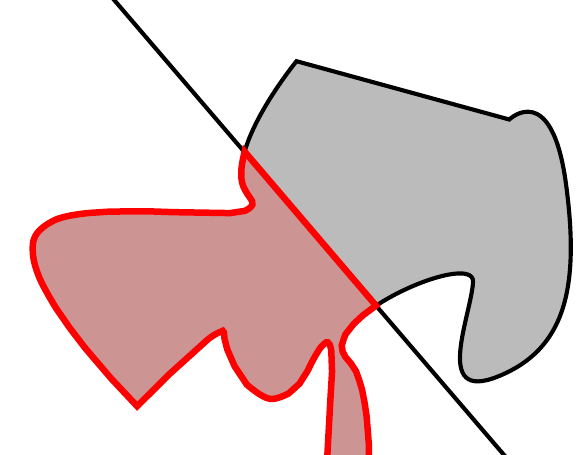}}%
    \put(0.62671646,0.35522095){\color[rgb]{0,0,0}\makebox(0,0)[lt]{\lineheight{1.25}\smash{\begin{tabular}[t]{l}$U$\end{tabular}}}}%
    \put(0.2647373,0.69602034){\color[rgb]{0,0,0}\makebox(0,0)[lt]{\lineheight{1.25}\smash{\begin{tabular}[t]{l}$Y$\end{tabular}}}}%
    \put(0.6710012,0.67484083){\color[rgb]{0,0,0}\makebox(0,0)[lt]{\lineheight{1.25}\smash{\begin{tabular}[t]{l}$Y^+$\end{tabular}}}}%
    \put(-0.00289588,0.47459707){\color[rgb]{0,0,0}\makebox(0,0)[lt]{\lineheight{1.25}\smash{\begin{tabular}[t]{l}$Y^-$\end{tabular}}}}%
    \put(0.18772051,0.28590588){\color[rgb]{1,0,0}\makebox(0,0)[lt]{\lineheight{1.25}\smash{\begin{tabular}[t]{l}\textcolor[rgb]{1,0,0}{$V$}\end{tabular}}}}%
    \put(0,0){\includegraphics[width=\unitlength,page=2]{hal.pdf}}%
    \put(0.51504205,0.44764118){\color[rgb]{0,0,1}\makebox(0,0)[lt]{\lineheight{1.25}\smash{\begin{tabular}[t]{l}\textcolor[rgb]{0,0,1}{$\partial_U^s V$}\end{tabular}}}}%
  \end{picture}%
\endgroup%

\caption{Illustration for Example~\ref{exhal}}\label{halpic}
\end{figure}
\begin{Rem}\label{otgkc}
If $X$ is a topological vector space over $\mathbb{R}$, then the following chain of inclusions:
\begin{equation}
\partial_U^s V\subseteq\partial_U V\subseteq \partial V
\end{equation}
holds. If $V$ is closed in $U$, then $V$ is closed with respect to segments in $U$.
\end{Rem}

\section{Piecewise definition of locally concave functions.}\label{seccuszalo}
Let $X$ be a linear space over $\mathbb{R}$, and let $\Omega\subseteq X$ be a subset. A function $B\colon\Omega\to\re$ is called locally concave if its restriction to any segment lying in its domain is concave. That is, for any points $x,\,y\in X$ such that $[x,\,y]\subseteq\Omega$, the function $B|_{[x,\,y]}$ is concave.

\begin{Le}\label{kuslocvec}
Let $X$ be a vector space over $\mathbb{R}$. Let $U\subseteq X$, and let $V\subseteq U$ be closed with respect to segments in $U$. Let $B\colon U\to \re$ be a locally concave function. Let $B^{\circ}\colon V\to\re$ be a locally concave function such that $B^{\circ}|_{\partial_U^s V}=B|_{\partial_U^s V}$ and, for any $x\in V$, the inequality $B^{\circ}(x)\leqslant B(x)$ holds. Then the function $B'\colon U\to\re$ defined by the formula
\begin{equation}
B'(x)=
\begin{cases}
B(x),&x\in U\setminus V,\\
B^{\circ}(x),&x\in V,
\end{cases}
\end{equation}
is locally concave.
\end{Le}
\begin{proof}
Let $a,\,b\in X$ be such that $[a,\,b]\subseteq U$. Then the function $B|_{[a,\,b]}$ is concave, and the function $B^{\circ}|_{[a,\,b]\cap V}$ is locally concave. The latter function does not exceed the former, and they coincide on the subset
\begin{equation}
[a,\,b]\cap\partial_U^s V\supseteq\partial_a^b V=\partial_a^b (V\cap[a,\,b]).
\end{equation}
The set $V$, in its turn, is closed with respect to segments in $U$, and hence $V\cap[a,\,b]$ is closed with respect to segments in $[a,\,b]$. Therefore, by Lemma~\ref{kusotrpolob} and Remark~\ref{zampootequsle}, the function $B'|_{[a,\,b]}$ is concave, which proves the local concavity of $B'$.
\end{proof}
\begin{Cor}\label{kusloc}
Let $U\subseteq\mathbb{R}^d$, and let the set $V\subseteq U$ be closed in $U$. Let $B\colon U\to \re$ be a locally concave function. Let $B^{\circ}\colon V\to \re$ be a locally concave function such that $B^{\circ}|_{\partial_U V}=B|_{\partial_U V}$ and, for any element $x\in V$, the inequality $B^{\circ}(x)\leqslant B(x)$ holds. Then the function $B'\colon U\to\re$ defined by the formula
\begin{equation}
B'(x)=
\begin{cases}
B(x),&x\in U\setminus V,\\
B^{\circ}(x),&x\in V,
\end{cases}
\end{equation}
is locally concave.
\end{Cor}
\begin{proof}
By Remark~\ref{otgkc} the set $V$ is closed with respect to the segments in $U$ and $B^{\circ}|_{\partial_U^s V}=B|_{\partial_U^s V}$, since $\partial_U^s V\subseteq\partial_U V$. Thus, the function $B'$ is locally concave by Lemma~\ref{kuslocvec}.
\end{proof}

\section{The hereditary property of minimal locally concave functions.}\label{secprnasl}
Let $X$ be a linear space over the field $\mathbb{R}$, and let $\Omega\subseteq X$ be a subset. The set $\Omega$ will be the domain of a minimal locally concave function. Let $f\colon\Omega\to\re$ be a function.
\begin{Def}
The minimal locally concave function $\BG=\BG_{f,\,\Omega}$ is defined by the formula
\begin{equation}
\BG(x)=\inf\left\{B(x)\big|\,B\colon\Omega\to\re\text{ is locally concave and }B\geqslant f\right\},\quad x\in\Omega.
\end{equation}
In this case, the function $f$ will be called the obstacle of the minimal locally concave function $\BG_{f,\,\Omega}$. 
\end{Def}
Note that $\BG$ may take infinite values.
\begin{Rem}
The function $\BG_{f,\,\Omega}$ is locally concave on $\Omega$.
\end{Rem}
In this section, we prove generalizations of Lemma 2.6 from~\cite{StZa}. 
\begin{Th}\label{kuseeinvec}
Let $X$ be a vector space over $\mathbb{R}$. Let $\Omega\subseteq X$, $f\colon \Omega\to\re$. Let the set $\Omega^{\circ}\subseteq\Omega$ be closed with respect to segments in $\Omega$, and let the function $f^{\circ}\colon\Omega^{\circ}\to\re$ be defined by the formula
\begin{equation}
f^{\circ}(x)=
\begin{cases}
f(x),&x\in \Omega^{\circ}\setminus \partial_{\Omega}^s\Omega^{\circ},\\
\BG_{f,\,\Omega}(x),&x\in \partial_{\Omega}^s\Omega^{\circ}.
\end{cases}
\end{equation}
Then $\BG_{f,\,\Omega}|_{\Omega^{\circ}}=\BG_{f^{\circ},\,\Omega^{\circ}}$.
\end{Th}
\begin{proof}
First observe that $\BG_{f,\,\Omega}|_{\Omega^{\circ}}\geqslant\BG_{f^{\circ},\,\Omega^{\circ}}$. This follows from the fact that the function $\BG_{f,\,\Omega}|_{\Omega^{\circ}}$ is locally concave and $\BG_{f,\,\Omega}|_{\Omega^{\circ}}\geqslant f^{\circ}$. In particular, this implies that $\BG_{f,\,\Omega}|_{\partial_{\Omega}^s \Omega^{\circ}}=\BG_{f^{\circ},\,\Omega^{\circ}}|_{\partial_{\Omega}^s\Omega^{\circ}}$. Then, by Lemma~\ref{kuslocvec}, the function $\BG'\colon\Omega\to\re$, defined by the formula
\begin{equation}
\BG'(x)=
\begin{cases}
\BG_{f,\,\Omega}(x),&x\in \Omega\setminus \Omega^{\circ},\\
\BG_{f^{\circ},\,\Omega^{\circ}}(x),&x\in \Omega^{\circ},
\end{cases}
\end{equation}
is locally concave. On the other hand, it is easy to see that $\BG'\leqslant\BG_{f,\,\Omega}$ and $\BG'\geqslant f$, and hence $\BG'=\BG_{f,\,\Omega}$. Therefore, $\BG_{f,\,\Omega}|_{\Omega^{\circ}}=\BG_{f^{\circ},\,\Omega^{\circ}}$.
\end{proof}

\begin{Th}\label{prinnasotgr}
Let $\Omega\subseteq\mathbb{R}^d$, $f\colon \Omega\to\re$. Let the set $\Omega^{\circ}\subseteq\Omega$ be closed in $\Omega$, and let $f^{\circ}\colon\Omega^{\circ}\to\re$ be defined by
\begin{equation}
f^{\circ}(x)=
\begin{cases}
f(x),&x\in \Omega^{\circ}\setminus \partial_{\Omega}\Omega^{\circ},\\
\BG_{f,\,\Omega}(x),&x\in \partial_{\Omega}\Omega^{\circ}.
\end{cases}
\end{equation}
Then $\BG_{f,\,\Omega}|_{\Omega^{\circ}}=\BG_{f^{\circ},\,\Omega^{\circ}}$.
\end{Th}
\begin{proof}
The proof is similar to the previous one, except that the theorem reduces to Corollary~\ref{kusloc} rather than to Lemma~\ref{kuslocvec}.
\end{proof}

\part{On the structure of minimal locally concave functions.}\label{prstrmlcf}
\section{Relative convexity.}\label{secotnvip}
In what follows, we restrict ourselves to the finite-dimensional case $\Omega\subseteq\mathbb{R}^d$, and we will always assume that the function $\BG|_{\intX\Omega}$ is pointwise finite. This condition implies two restrictions:
\begin{equation}
\BG|_{\intX\Omega}\ne-\infty,
\end{equation}
\begin{equation}
\BG|_{\intX\Omega}\ne+\infty.
\end{equation}
The first condition imposes certain restrictions on the geometry of the domain $\Omega$ and on the set $C=\{x\in\Omega|\,f(x)\ne-\infty\}$. They can be formulated in terms of the relative convex hull as follows.
\begin{Def}\label{otnositvip}
Let $X\subseteq\Omega\subseteq\mathbb{R}^d$. We say that the set $X$ is relatively convex in the domain $\Omega$ if, for any points $x,\,y\in X$ such that the segment $[x,\,y]$ lies in $\Omega$, the segment $[x,\,y]$ as well lies in~$X$.
\end{Def}
\begin{Def}\label{relconv}
Let $X\subseteq\Omega\subseteq\mathbb{R}^d$. The relative convex hull of the set $X$ in the domain $\Omega$ is the inclusion-minimal superset of $X$ that is relatively convex in $\Omega$; it is denoted by $\conv_{\Omega}(X)$.
\end{Def}
\begin{Ex}\label{inconv}
Let $\Omega\subseteq\mathbb{R}^2$ be an annulus \textup{(}see Fig.~\ref{pikkinconv}\textup{)}, and let $X\subseteq\Omega$ be the lower semicircle of the outer boundary of $\Omega$. Then $\conv X$ is the lower semidisk, whereas $\conv_{\Omega}X$ is the domain bounded by the tangents to the inner circle of $\Omega$.
\end{Ex}
\begin{figure}[h]\centering
\def\svgwidth{8cm}
%% Creator: Inkscape 1.3.2 (091e20e, 2023-11-25, custom), www.inkscape.org
%% Accompanies image file 'relconv.pdf' (pdf, eps, ps)
\begingroup%
  \makeatletter%
  \providecommand\color[2][]{%
    \errmessage{(Inkscape) Color is used for the text in Inkscape, but the package 'color.sty' is not loaded}%
    \renewcommand\color[2][]{}%
  }%
  \providecommand\transparent[1]{%
    \errmessage{(Inkscape) Transparency is used (non-zero) for the text in Inkscape, but the package 'transparent.sty' is not loaded}%
    \renewcommand\transparent[1]{}%
  }%
  \providecommand\rotatebox[2]{#2}%
  \newcommand*\fsize{\dimexpr\f@size pt\relax}%
  \newcommand*\lineheight[1]{\fontsize{\fsize}{#1\fsize}\selectfont}%
  \ifx\svgwidth\undefined%
    \setlength{\unitlength}{374.25bp}%
    \ifx\svgscale\undefined%
      \relax%
    \else%
      \setlength{\unitlength}{\unitlength * \real{\svgscale}}%
    \fi%
  \else%
    \setlength{\unitlength}{\svgwidth}%
  \fi%
  \global\let\svgwidth\undefined%
  \global\let\svgscale\undefined%
  \makeatother%
  \begin{picture}(1,1.00200401)%
    \lineheight{1}%
    \setlength\tabcolsep{0pt}%
    \put(0,0){\includegraphics[width=\unitlength,page=1]{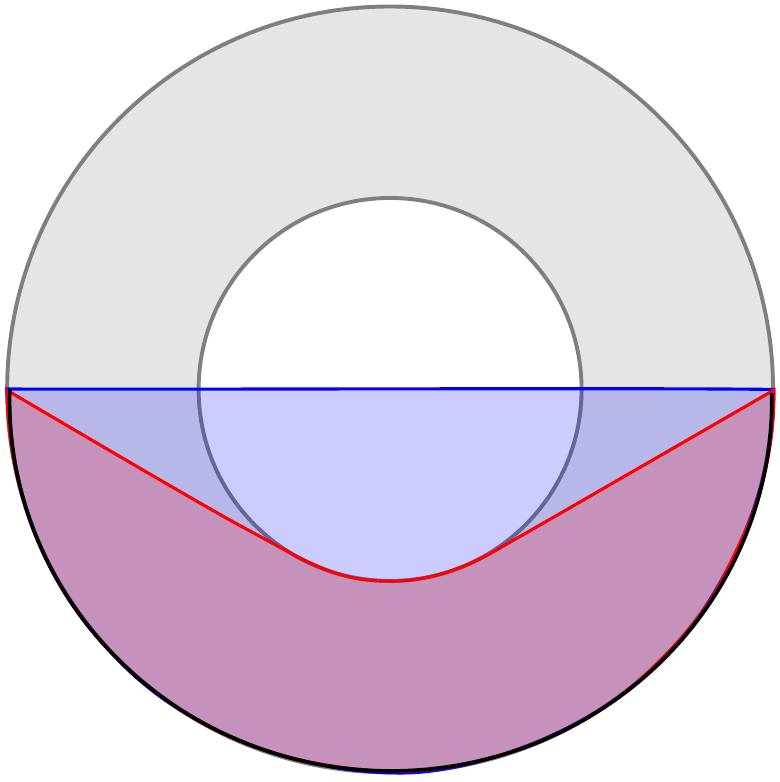}}%
    \put(0.88154214,0.15756813){\color[rgb]{0,0,0}\makebox(0,0)[lt]{\lineheight{1.25}\smash{\begin{tabular}[t]{l}$X$\end{tabular}}}}%
    \put(0.54591006,0.67757864){\color[rgb]{0,0,0}\makebox(0,0)[lt]{\lineheight{1.25}\smash{\begin{tabular}[t]{l}$\textcolor[rgb]{0.5,0.5,0.5}{\Omega}$\end{tabular}}}}%
    \put(0.40870553,0.51770558){\color[rgb]{0,0,0}\makebox(0,0)[lt]{\lineheight{1.25}\smash{\begin{tabular}[t]{l}$\textcolor[rgb]{0,0,1}{\conv X}$\end{tabular}}}}%
    \put(0.47551811,0.30891599){\color[rgb]{0,0,0}\makebox(0,0)[lt]{\lineheight{1.25}\smash{\begin{tabular}[t]{l}$\textcolor[rgb]{1,0,0}{\conv_{\Omega}X}$\end{tabular}}}}%
  \end{picture}%
\endgroup%

\caption{Illustration for Example~\ref{inconv}}\label{pikkinconv}
\end{figure}
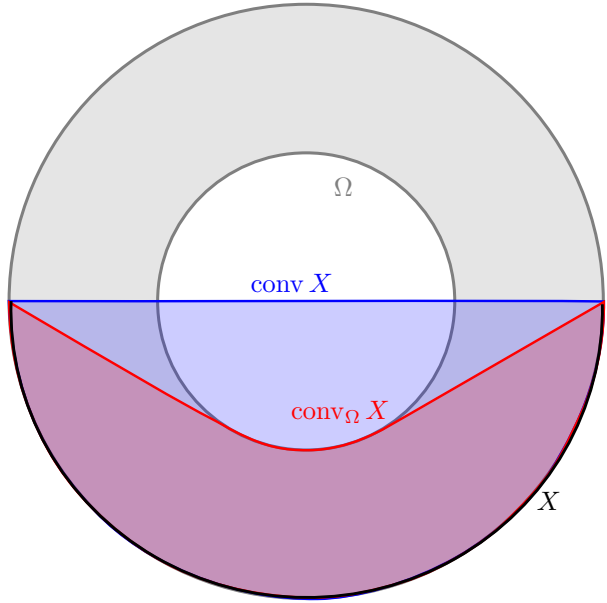

The condition that $\BG$ does not attain the value $-\infty$ at interior points is equivalent to the inclusion $\intX\Omega\subseteq\conv_{\Omega}(C)$. 

The second condition bounds the growth of the function $f$ toward $+\infty$. Its precise form depends on the domain $\Omega$ and on the rate at which the negative part of $f$ tends to $-\infty$. In the case where $\Omega=\Omega_0\setminus\Omega_1\subseteq\mathbb{R}^2$ is the difference of two convex bodies and the obstacle is defined on the outer boundary ($f|_{\Omega\setminus\partial\Omega_0}\equiv-\infty$), sufficient conditions on $f$ ensuring that $\BG$ does not attain the value $+\infty$ are described in Condition 2.1.13 of~\cite{talmudeIII}; the paragraph immediately following Condition 2.1.13 gives necessary conditions on $f$ that are close to sufficient. Theorem 1.16 of~\cite{Novi} gives an sharp criterion under which, for nonnegative boundary obstacles, the minimal biconcave function in the diagonal strip is finite; together with Theorem 1.3.4 of~\cite{StaVaZaMRTR}, this gives an sharp criterion for the finiteness of the minimal locally concave function in the parabolic strip with continuous nonnegative boundary data.

\section{Contact zone.}\label{seczcon}
Since the function $\BG$ is pointwise finite at interior points, it is continuous there (see Theorem 10.1 of~\cite{Rockafellar}). Together with the inequality $\BG\geqslant f$, this gives the estimate that at interior points $\BG$ is bounded from below by the upper semicontinuous majorant of $f$:
\begin{equation}\label{podpor}
\BG(x)\geqslant\supfu{f}(x),
\end{equation}
for any $x\in\intX\Omega$.
\begin{Def}
The contact zone of the minimal locally concave function $\BG$ and the obstacle $f$ is the set of points of $\intX\Omega$ for which inequality~\eqref{podpor} becomes an equality. We denote the contact zone by the symbol
\begin{equation}\label{oprE_0for}
E_0
=
E_0(f,\,\Omega)
=
\left\{
x\in\intX\Omega\big|\,\BG(x)=\supfu{f}(x)
\right\}.
\end{equation}
\end{Def}
\begin{Le}
The set $\partial\Omega\cup E_0$ is closed (that is, $E_0$ is closed in the induced topology of $\intX\Omega$).
\end{Le}
\begin{proof}
The set $\partial\Omega$ is closed, and the set $E_0$ lies in the interior of $\Omega$. Consequently, it suffices to check that if a point $x\in\intX\Omega$ is a limit point of the set $E_0$, then it lies in $E_0$ itself. That is, it suffices to show that 
\begin{equation}
\BG(x)\leqslant\supfu{f}(x), \quad\text{if}\quad x\in\intX\Omega\cap\clos E_0.
\end{equation}
Indeed, $x\in\intX\Omega$, hence, the function $\BG$ is continuous at the point $x$, and we can write the following chain of equalities and inequalities:
\begin{equation}
\BG(x)
=
\lim\limits_{y\to x}\BG(y)
=
\!\!\lim\limits_{y\to x,\,y\in E_0}\!\!\BG(y)
=
\!\!\lim\limits_{y\to x,\,y\in E_0}\!\!\supfu{f}(y)
\leqslant
\supfu{f}(x).
\end{equation}
\end{proof}
The open set $\intX\Omega\setminus E_0$ will be denoted by $E_{>0}$. For this set, the inclusion $\partial E_{>0}\subseteq\partial\Omega\cup E_0$ holds.

\section{The coincidence set of an affine functional and a minimal locally concave function}\label{secmnsop}
\begin{Def}\label{mnozhsopdef}
Let $W\subseteq E_{>0}$ be a closed convex set with nonempty interior. Let $L\colon\mathbb{R}^d\to\mathbb{R}$ be an affine functional. Assume that $\BG|_W\leqslant L|_W$. Denote
\begin{equation}
U(W,\,L)=\{x\in W|\,\BG(x)=L(x)\}.
\end{equation}
We call $U$ the coincidence set of the functional $L$ and the function $\BG$ on the subset~$W$.
\end{Def}
\begin{Rem}
The set $U(W,\,L)$ is closed and convex.
\end{Rem}
\begin{Rem}\label{Ucap}
Let $W$ and $W'$ be closed convex subsets of $E_{>0}$ with nonempty interiors such that $W\cap W'$ also has nonempty interior. Assume that $L\colon\mathbb{R}^d\to \mathbb{R}$ is an affine functional such that $\BG|_W\leqslant L|_W$ and $\BG|_{W'}\leqslant L|_{W'}$. Then the following equalities:
\begin{equation}
U(W,\,L)\cap U\left(W',\,L\right)=W\cap U\left(W',\,L\right)=U(W,\,L)\cap W'=U\left(W\cap W',\,L\right)
\end{equation}
hold.
\end{Rem}
\begin{Rem}\label{sushch}
For any closed convex set $W\subseteq E_{>0}$ with nonempty interior and any point $x\in\intX W$, there exists an affine functional $L\colon\mathbb{R}^d\to\mathbb{R}$ such that $\BG|_W\leqslant L|_W$ and $x\in U(W,\,L)$.
\end{Rem}
\begin{proof}
It suffices to take for $L$ a functional supporting $\BG|_W$ at the point $x$, that is, the functional defining a supporting hyperplane to the subgraph of $\BG|_W$ at the point $(x,\,\BG(x))$.
\end{proof}
\begin{Le}\label{neighborhood inequality}
Let $W\subseteq E_{>0}$ be a closed convex set with nonempty interior, let $x\in\intX W$, and let $L\colon\mathbb{R}^d\to\mathbb{R}$ be an affine functional such that $x\in U(W,\,L)$. Assume also that $y\in\Omega$ is a point such that $[x,\,y]\subseteq\Omega$. Then $\BG(y)\leqslant L(y)$.
\end{Le}
\begin{proof}
Assume the contrary: $\BG(y)>L(y)$. We have, $\BG(x)=L(x)$, and hence, by concavity, $\BG>L$ on the half-interval $(x,\,y]$. This contradicts $L|_W\geqslant\BG|_W$ and $x\in\intX W$.
\end{proof}
\begin{Cor}\label{segment equality}
Under the assumptions of the previous lemmas, let $L(y)=\BG(y)$. Then $L(z)=\BG(z)$ for any point $z\in[x,\,y]$.
\end{Cor}

\begin{Le}\label{simploc}
Let $W\subseteq E_{>0}$ be a compact convex set with nonempty interior, let $x\in\intX W$, and let $L\colon\mathbb{R}^d\to\mathbb{R}$ be an affine functional such that $x\in U(W,\,L)$. Then there exists a simplex $C$ with vertices on $\partial W$ such that $x\in C$ and $C\subseteq U(W,\,L)$.
\end{Le}
The lemma can be reformulated as follows: The extreme points of the set $U(W,\,L)$ lie on the boundary of $W$. This lemma is of similar spirit with Lemma 2 in~\cite{CafNirSpr} and Lemma 3.2 in~\cite{DePhiFig}. However, in~\cite{CafNirSpr} the authors restrict themselves to the case of a boundary obstacle, and in~\cite{DePhiFig} the obstacle may take only finite values.

\begin{proof}
First, note that the function $\BG$ is continuous on the set $W$, and the function $\supfu{f}$ is upper semicontinuous. Thus, the function $\BG-\supfu{f}$ is lower semicontinuous. Since $W\subseteq E_{>0}$, the function $\BG-\supfu{f}$ is positive on $W$. Therefore, $\BG-\supfu{f}$ attains a positive minimum $\delta$ on $W$. Let the function $g\colon W\to\re$ be defined by the formula
\begin{equation}
g(y)
=
\begin{cases}
\BG(y), & \text{if } y\in\partial W,\\
f(y), & \text{if } y\in\intX W.
\end{cases}
\end{equation}
Then, by Theorem~\ref{prinnasotgr}
\begin{equation}
\BG_{f,\,\Omega}|_{_W}=\BG_{g,\,W}.
\end{equation}
Define the set $J$ by the formula $J=\{y\in\partial W|\,\BG(y)=L(y)\}$. Then $J$ is closed, and, by Carath\'eodory's theorem, it suffices to show that $x\in\conv J$. Assume the contrary. Then there exists an affine functional $L'\colon\mathbb{R}^d\to\mathbb{R}$ such that $L'(x)<0$, $L'|_{J}>0$, and $|L'|<1$ on $W$.

Let us show that for sufficiently small $\zeta>0$ and any $y\in W$ the inequality $L(y)+\zeta L'(y)\geqslant g(y)$ holds. Let $\zeta<\delta$. Then, for $y\in\intX W$, the following chain of inequalities:
\begin{equation}
L(y)+\zeta L'(y)\geqslant L(y)-\delta\geqslant \BG(y)-\delta\geqslant f(y)=g(y)
\end{equation}
holds. Let the set $D$ be defined by the formula $D=\{y\in\partial W| L'(y)\leqslant 0\}$. The set $D$ is closed, and the function $L-\BG$ is positive on $D$, since $D$ does not intersect $J$ by the inequalities $L'|_{J}>0$. Let
\begin{equation}
\zeta<\min\{L(y)-\BG(y)|\,y\in D\}.
\end{equation}
Then for $y\in D$ the following chain of inequalities:
\begin{equation}
L(y)+\zeta L'(y)\geqslant L(y)-\min\{L(z)-\BG(z)|\,z\in D\}\geqslant L(y)-(L(y)-\BG(y))= \BG(y)=g(y)
\end{equation}
holds. We also have
\begin{equation}
L(y)+\zeta L'(y)\geqslant L(y)\geqslant \BG(y)=g(y),\quad y\in\partial W\setminus D.
\end{equation}
Consequently, $L(y)+\zeta L'(y)\geqslant g(y)$ for any $y\in W$. However, the function $L+\zeta L'$ is linear and, in particular, concave. Hence $(L+\zeta L')|_{W}\geqslant\BG_{g,\,W}=\BG_{f,\,\Omega}|_{W}$. Thus, the following chain of inequalities:
\begin{equation}
\BG(x)=L(x)>L(x)+\zeta L'(x)\geqslant \BG(x)
\end{equation}
holds. Contradiction.
\end{proof}

\section{Definition of extremal sets and their properties}\label{secekst}
\begin{Def}\label{oprpredex}
Let $L\colon\mathbb{R}^d\to\mathbb{R}$ be an affine functional, and let $l\subseteq\mathbb{R}^d$ be an affine subspace. A point $x\in E_{>0}\cap l$ belongs to the pre-extremal set $\widetilde{E}(L,\,l)$ if and only if the following conditions:
\begin{enumerate}
\item
There exists a closed convex set $W\subseteq E_{>0}$ with nonempty interior such that $x\in\intX W$, $\BG|_W\leqslant L|_W$, and $x\in U=U(W,\,L)$.
\item
The set $l\cap U$ is a face of $U$ such that $\aff(l\cap U)=l$.
\end{enumerate}
hold.
\end{Def}
\begin{Rem}
For any point $x\in\widetilde{E}(L,\,l)$ the equality $\BG(x)=L(x)$ holds.
\end{Rem}
Informally, a pre-extremal set is a set on which the minimal locally concave function is affine. The additional parameter $l$ appears because such a set naturally splits into pieces of different dimensions. It will turn out below that the boundary of an extremal set is essentially an extremal set of smaller dimension, see Lemma~\ref{extremal boundary}. This fact is the key for carrying out induction on the dimension of extremal sets in Theorem~\ref{mainth}.

In the definition of pre-extremal sets, one chooses a convex set $W$. The following lemma says that this choice is not essential: replacing the words ``there exists'' in the definition by ``for any'' does not change the object under consideration.
\begin{Le}\label{usopr}
Let $L\colon\mathbb{R}^d\to \mathbb{R}$ be an affine functional, let $l\subseteq\mathbb{R}^d$ be an affine subspace, and let $x\in\widetilde{E}(L,\,l)$. This means that $x\in l$, $\BG(x)=L(x)$ and there exists a closed convex set $W\subseteq E_{>0}$ such that $x\in\intX W$, $L|_{W}\geqslant\BG|_{W}$, and $l\cap U(W,\,L)$ is a face of $U(W,\,L)$ such that $l=\aff(l\cap U(W,\,L))$. Then, for any closed convex set $W'\subseteq E_{>0}$ such that $x\in\intX W'$, the following conditions: $L|_{W'}\geqslant\BG|_{W'}$, and $l\cap U\left(W',\,L\right)$ is a face of $U(W',\,L)$ such that $l=\aff(l\cap U(W',\,L))$ hold.
\end{Le}
\begin{proof}
Let $W'\subseteq E_{>0}$ be a closed convex set such that $x\in\intX W'$. Since $L|_{W}\geqslant\BG|_{W}$, $x\in\intX W$, and $L(x)=\BG(x)$, Lemma~\ref{neighborhood inequality} implies that $\BG(y)\leqslant L(y)$ for all $y$ such that $[x,\,y]\subset\Omega$. In its turn, $W'$ is convex, and we have $x\in W'\subseteq E_{>0}$. Thus, for any element $y\in W'$ we have the inclusion
\begin{equation}
[x,\,y]\subseteq W'\subseteq E_{>0}\subseteq\Omega.
\end{equation}
Thus $L|_{W'}\geqslant\BG|_{W'}$.

On the other hand, since $x\in\intX W\cap\intX W'=\intX(W\cap W')\ne\varnothing$, Remark~\ref{Ucap} gives the chain of equalities
\begin{equation}
U(W,\,L)\cap U\left(W',\,L\right)=W\cap U\left(W',\,L\right)=U(W,\,L)\cap W'=U\left(W\cap W',\,L\right).
\end{equation}
Since $l\cap U(W,\,L)$ is a face of $U(W,\,L)$ such that $l=\aff(l\cap U(W,\,L))$, Corollary~\ref{lemepdel} implies that the intersection $l\cap\left(U\left(W\cap W',\,L\right)\right)$ is a face of $U\left(W\cap W',\,L\right)$ such that $l=\aff\left(l\cap U\left(W\cap W',\,L\right)\right)$. Consequently, again by Corollary~\ref{lemepdel}, the intersection $l\cap U\left(W',\,L\right)$ is a face of $U\left(W',\,L\right)$ such that $l=\aff\left(l\cap U\left(W',\,L\right)\right)$.
\end{proof}
\begin{Cor}\label{netnulmerextr}
Let $L\colon\mathbb{R}^d\to\mathbb{R}$ be an affine functional, let $l\subseteq\mathbb{R}^d$ be an affine subset, and let $x\in\widetilde{E}(L,\,l)$. Then $\dim l\geqslant 1$.
\end{Cor}
\begin{proof}
We need to check that, for no set $W$ with $x\in\intX W$, the point $x$ is an extreme point of the set $U(W,\,L)$. By Lemma~\ref{usopr} it suffices to verify this for compact $W$, and for compact $W$ this follows from Lemma~\ref{simploc}.
\end{proof}

\begin{Def}\label{visible points}
For $x\in E_{>0}$, denote by $V(x)$ the set of points of $E_{>0}$ visible from $x$:
\begin{equation}
V(x)=\left\{y\in E_{>0}\big|\,[x,\,y]\in E_{>0}\right\}.
\end{equation}
\end{Def}
\begin{Rem}\label{open-set}
For $x\in E_{>0}$, the set $V(x)$ is open.
\end{Rem}
\begin{proof}
If $y\in V(x)$, then $[x,\,y]\subseteq E_{>0}$, and hence there exists $\varepsilon>0$ such that $[x,\,y]+B_{\varepsilon}(0)\subseteq E_{>0}$. Therefore, $B_{\varepsilon}(y)\subseteq V(x)$.
\end{proof}

\begin{Rem}\label{remosushchex}
For any point $x\in E_{>0}$, there exist an affine subspace $l$ and an affine functional $L$ such that $x\in\widetilde{E}(L,\,l)$.
\end{Rem}

\begin{Def}\label{opredext}
The extremal set $E(x,\,L,\,l)$ is defined as the connected component of the point $x$ in the set $\widetilde{E}(L,\,l)$. The dimension of the pre-extremal set $\widetilde{E}(L,\,l)$ and of the extremal set $E(x,\,L,\,l)$ is defined as the dimension of the subspace $l$.
\end{Def}
A standard example of an extremal set is a tangent segment from a point of the fixed boundary to the free boundary, or a chord joining points of the fixed boundary. For example, see formula (2.7) and Figure 3 in~\cite{VasVol}; the function obtained there is linear on the domains $\Omega_1,\,\Omega_2^{\pm},\,\Omega_5$, linear along chords on the domains $\Omega_3^{\pm}$, and linear along tangents on the domains $\Omega_4^{\pm}$. Informally, Theorem~\ref{tangency} can be understood as the statement that there are no other one-dimensional extremal sets. Nontrivial examples of two-dimensional extremal sets are given in~\cite{StaVaZa1}: formula (3.2) of that paper describes an extremal set, and formula (3.3) gives a supporting functional on it.
\begin{Le}\label{rasek}
Let $x\in\widetilde{E}(L,\,l)$, $y\in V(x)\cap l$ and $\BG(y)=L(y)$. Then $y\in E\left(x,\,L,\,l\right)$ and $[x,\,y]\subseteq E\left(x,\,L,\,l\right)$. 
\end{Le}
\begin{proof}
Let $\varepsilon<\dist([x,\,y],\,\partial E_{>0})$. Let $W=[x,\,y]+B_{\varepsilon}(0)$ and $U=U(W,\,L)$. Since $x\in\widetilde{E}(L,\,l)$, Lemma~\ref{usopr} implies that $\BG|_{W}\leqslant L|_{W}$ and that $l\cap U$ is a face of $U$ such that $l=\aff(l\cap U)$. In its turn, by Corollary~\ref{segment equality}, for any point $z\in[x,\,y]$ the equality $\BG(z)=L(z)$ holds, and hence $z\in U$. Moreover, since $x,\,y\in l$, we have $z\in l$; therefore, $z\in\widetilde{E}(L,\,l)$. Thus, $[x,\,y]\subseteq\widetilde{E}(L,\,l)$, and hence $y\in E(x,\,L,\,l)$ and $[x,\,y]\subseteq E(x,\,L,\,l)$.
\end{proof}
\begin{Cor}\label{vipukn}
The sets $\widetilde{E}(L,\,l)$ and $E(x,\,L,\,l)$ are relatevly convex in $E_{>0}$.
\end{Cor}
\begin{Le}\label{extremal boundary}
Let $x\in\widetilde{E}(L,\,l)$ and $y\in\partial_{l} E\left(x,\,L,\,l\right)$. Then one of the following two alternatives:
\begin{enumerate}
\item
$y\in\partial E_{>0}$.
\item
$y\in E(x,\,L,\,l)$ and there exists an affine subspace $l'\subsetneq l$ of $\mathbb{R}^d$ such that $y\in\widetilde{E}\left(L,\,l'\right)$.
\end{enumerate}
holds.
\end{Le}
\begin{proof}
Let $y\in\partial_{l}E(x,\,L,\,l)\setminus\partial E_{>0}$. Then $y$ is a point of continuity of the function $\BG$, and hence $\BG(y)=L(y)$, since $\BG|_{E(x,\,L,\,l)}=L|_{E(x,\,L,\,l)}$. Let $z\in E(x,\,L,\,l)\cap V(y)$. Then $y\in l\cap V(z)$, and by Lemma~\ref{rasek} we have $y\in E(z,\,L,\,l)=E(x,\,L,\,l)$. Let $0<\varepsilon<\dist(y,\,\partial E_{>0})$ and $U=U(B_{\varepsilon}(y),\,L)$. Then $U\cap l$ is a face of $U$ such that $l=\aff(U\cap l)$ and $U\cap l\subseteq E(x,\,L,\,l)$. Since $y\in\partial_l E(x,\,L,\,l)$, we have $y\in\partial_l(U\cap l)$; hence, by Proposition~\ref{bougr}, there exists a face $F$ of $U\cap l$ such that $y\in F$ and $\dim(\aff(F))<\dim l$. Let $l'=\aff F$. Then $l'\subsetneq l$, $U\cap l'=(U\cap l)\cap l'$, and, by Proposition~\ref{grgrgr}, $F$ is a face of $U$. Hence $y\in\widetilde{E}\left(L,\,l'\right)$.
\end{proof}
The following corollary formalizes the principle that a one-dimentional extremal set cannot stop.
\begin{Cor}\label{gronedim}
Let $l$ be a line and $x\in\widetilde{E}(L,\,l)$. Then $\partial_{l}E(x,\,L,\,l)\subseteq\partial E_{>0}$.
\end{Cor}
\begin{proof}
By Lemma~\ref{extremal boundary}, the boundary of an extremal set consists of points of the set $\partial E_{>0}$ and extremal sets of smaller dimension. By Corollary~\ref{netnulmerextr}, extremal sets have dimension at least one, and since $l$ is a line, the desired inclusion $\partial_{l}E(x,\,L,\,l)\subseteq\partial E_{>0}$ follows.
\end{proof}
\subsection{Inductive construction of $E(x,\,L,\,l)$.}\label{inducpostr}
Let $x\in\widetilde{E}(L,\,l)$. We define the sets $E_n(x,\,L,\,l)$ by induction on $n$:
\begin{equation}
E_{0}(x,\,L,\,l)=\{x\},
\end{equation}
\begin{equation}
E_{n+1}(x,\,L,\,l)=\left\{y\in E_{>0}\big|\,\exists z\in E_{n}(x,\,L,\,l)\text{ such that }y\in V(z)\cap l\text{ and }\BG(y)=L(y)\right\}.
\end{equation}
\begin{Le}
One has
\begin{equation}\label{indforext}
E(x,\,L,\,l)=\bigcup\limits_{n\in\mathbb{N}_0}E_{n}(x,\,L,\,l)
\end{equation}
for any $L,\,l,$ and $x$.
\end{Le}
\begin{proof}
Temporarily denote the set $\bigcup_{n\in\mathbb{N}_0}E_{n}(x,\,L,\,l)$ by $E'(x,\,L,\,l)$. The proof of equality of the sets naturally splits into the proof of two inclusions.

{\bf 1. Proof of the inclusion $E(x,\,L,\,l)\supseteq E'(x,\,L,\,l)$.} We prove by induction on $n$ that $E_n(x,\,L,\,l)\subseteq E(x,\,L,\,l)$. The induction base $n=0$ is obvious. 

\noindent Induction step. We know that $E_n(x,\,L,\,l)\subseteq E(x,\,L,\,l)$. Let $y\in E_{n+1}(x,\,L,\,l)$. Then there exists a point $z\in E_n(x,\,L,\,l)$ such that $y\in V(z)\cap l$ and $\BG(y)= L(y)$. Since, by the induction hypothesis, $z\in E(x,\,L,\,l)$, Lemma~\ref{rasek} implies $y\in E(x,\,L,\,l)$.

{\bf 2. Proof of the inclusion $E(x,\,L,\,l)\subseteq E'(x,\,L,\,l)$.} To prove this inclusion, it suffices to show that the set $E'(x,\,L,\,l)$ is open and closed in $\widetilde{E}(L,\,l)$, since $E(x,\,L,\,l)$ is a connected component of $\widetilde{E}(L,\,l)$. This observation splits the proof of the second assertion into two subitems.

{\bf 2.1. Proof of openness of $E'(x,\,L,\,l)$ in $\widetilde{E}(L,\,l)$.} Let $y\in E'(x,\,L,\,l)$. Then there exists a number $n$ such that $y\in E_{n}(x,\,L,\,l)$. Then, the definition of $E_{n+1}(x,\,L,\,l)$ implies that 
\begin{equation}
V(y)\cap\widetilde{E}(L,\,l)\subseteq E_{n+1}(x,\,L,\,l)\subseteq E'(x,\,L,\,l),
\end{equation}
and, by Remark~\ref{open-set}, $y$ has a nonempty neighborhood in $\widetilde{E}(L,\,l)$.

{\bf 2.2. Proof of closedness of $E'(x,\,L,\,l)$ in $\widetilde{E}(L,\,l)$.} Let 
\begin{equation}
y\in
\clos (E'(x,\,L,\,l))\cap \widetilde{E}(L,\,l)
\subseteq
\clos( E'(x,\,L,\,l))\cap E_{>0}.
\end{equation}
Then, by Remark~\ref{open-set}, there exists $z\in E'(x,\,L,\,l)\cap V(y)$. Then, there exists a number $n$ such that $z\in E_{n}(x,\,L,\,l)$, whereas $y\in V(z)\cap l$ and $\BG(y)=L(y)$; hence, $y\in E_{n+1}(x,\,L,\,l)\subseteq E'(x,\,L,\,l)$. Therefore,
\begin{equation}
\clos (E'(x,\,L,\,l))\cap \widetilde{E}(L,\,l)=E'(x,\,L,\,l).
\end{equation}
\end{proof}
\begin{Cor}\label{embedded extremal}
Let $l,\,l'$ be affine subspaces of $\mathbb{R}^d$ and $l'\subseteq l$. Assume also that $x\in\widetilde{E}(L,\,l)$ and $x\in\widetilde{E}\left(L,\,l'\right)$. Then $E\left(x,\,L,\,l'\right)\subseteq E(x,\,L,\,l)$.
\end{Cor}
\begin{proof}
It is easy to see by induction on $n$ that $E_n\left(x,\,L,\,l'\right)\subseteq E_n(x,\,L,\,l)$.
\end{proof}

\section{On the structure of extremal sets.}\label{secsteks}
\subsection{On one-dimensional extremal sets.}
\begin{St}\label{basic proposition}
Let $l$ be a line, and suppose that there exists an affine functional $L\colon\mathbb{R}^d\to\mathbb{R}$ such that $\widetilde{E}(L,\,l)\ne\varnothing$. Let $x\in\widetilde{E}(L,\,l)$. Then $E(x,\,L,\,l)$ is either the whole line $l$, or an open ray with base on $\partial E_{>0}$, or an open interval with endpoints on $\partial E_{>0}$.
\end{St}
\begin{proof}
On the one hand, the set $E(x,\,L,\,l)$ is connected; on the other hand, $E(x,\,L,\,l)\subseteq l$. In this case, by Corollary~\ref{gronedim}, we have $\partial_l E(x,\,L,\,l)\subseteq\partial E_{>0}$. Consequently, $E(x,\,L,\,l)$ is open in $l$. Thus, $E(x,\,L,\,l)$ is either a line, or an open ray, or an open interval whose endpoints lie in $\partial E_{>0}$.
\end{proof}

\subsection{The structural theorem.}
Let $x\in\widetilde{E}(L,\,l)$. Our goal is to show that an analogue of Carath\'eodory's theorem holds for the sets $E(x,\,L,\,l)$. Our final result is the following theorem, which is a direct generalization of Lemma 2 in~\cite{CafNirSpr} and Lemma 3.2 in~\cite{DePhiFig}.
\begin{Th}\label{mainth}
Let $x\in\widetilde{E}(L,\,l)$. Then at least one of the following two statements:
\begin{enumerate}
\item
There exists a line $m$ such that $x\in m\subseteq E(x,\,L,\,l)$.
\item
There exists a generalized simplex $C$ such that $x\in\ri C$, $V_{\fin}(C)\subseteq\partial E_{>0}$ and $C\subseteq \clos(E(x,\,L,\,l))$.
\end{enumerate}
holds. Moreover, if the second condition holds, then the simplex $C$ can be required to satisfy the additional restrictions: $C\setminus V_{\fin}(C)\subseteq E(x,\,L,\,l)$ and $\#(V_{\dir}(C))\leqslant 1$.
\end{Th}

\begin{Cor}
Let $\Omega$ contain no rays and $f|_{\intX\Omega}\equiv-\infty$. Assume also that $x\in\widetilde{E}(L,\,l)$. Then there exists a simplex $C$ with vertices on $\partial\Omega$ such that $x\in\ri C$ and $C\setminus V(C)\subseteq E(x,\,L,\,l)$.
\end{Cor}

Together with Remark~\ref{remosushchex}, this yields the following corollary.
\begin{Cor}
Let $\Omega$ contain no rays and $f|_{\intX\Omega}\equiv-\infty$. Suppose also that $x\in\intX\Omega$. Then there exists a simplex $C$ with vertices on the boundary of $\Omega$ such that $x\in\ri C$ and $C\setminus V(C)\subseteq \intX\Omega$, and the function $\BG$ is affine on $C\setminus V(C)$.
\end{Cor}

The proof of the theorem occupies the whole subsection and naturally splits into three steps.
\begin{enumerate}
\item
Define the points of $\partial\Omega$ and the directions that will serve as candidates for finite and vertices at infinity respectively of the generalized simplex $C$.
\item
Prove that either a line passing through $x$ lies entirely in the extremal set, or $x$ lies in the convex hull of the candidates for the vertices of $C$.
\item
Choose the minimal by inclusion among the generalized simplices containing $x$ for which all vertices are candidates, and show that it satisfies all the properties claimed in the theorem.
\end{enumerate}
\begin{Def}\label{canconver}
Denote
\begin{equation}
V_{\fin}(x,\,L,\,l)=\{y\in\partial E_{>0}|[x,\,y]\subseteq \clos(E(x,\,L,\,l))\}.
\end{equation}
\end{Def}
\begin{Def}\label{cannapver}
Denote
\begin{equation}
V_{\dir}(x,\,L,\,l)=\{\vec{e}\in \vec{S}^{d-1}|[x,\,\vec{e})\subseteq \clos(E(x,\,L,\,l))\}.
\end{equation}
\end{Def}
\begin{Rem}\label{closed-pair remark}
The pair of sets $(V_{\fin}(x,\,L,\,l);\;V_{\dir}(x,\,L,\,l))$ is closed.
\end{Rem}
\begin{proof}
To check this fact, it suffices to show that if a sequence of points and directions $\{v_n\}_{n\in\mathbb{N}}\subseteq V_{\fin}(x,\,L,\,l)\cup V_{\dir}(x,\,L,\,l)$ has a limit $v\in \mathbb{R}^d\cup\vec{S}^{d-1}$, then $v\in V_{\fin}(x,\,L,\,l)\cup V_{\dir}(x,\,L,\,l)$. The fact that $v_n\in V_{\fin}(x,\,L,\,l)\cup V_{\dir}(x,\,L,\,l)$ means that $[x,\,v_n)\subseteq\clos(E(x,\,L,\,l))$. By the convergence $v_n\to v$, the inclusion $[x,\,v)\subseteq\clos(E(x,\,L,\,l))$ holds, and hence $v\in V_{\fin}(x,\,L,\,l)\cup V_{\dir}(x,\,L,\,l)$.
\end{proof}
\begin{Le}\label{finite-vertex lemma}
Let $x\in\widetilde{E}(L,\,l)$, $y\in E(x,\,L,\,l)$ and $y\in V(x)$. Then for any point $z\in V_{\fin}(y,\,L,\,l)$ there exists $\widetilde{z}\in V_{\fin}(x,\,L,\,l)\cap\conv(\{x,\,y,\,z\})$.
\end{Le}
\begin{figure}[h]\centering
\def\svgwidth{7cm}
%% Creator: Inkscape 1.3.2 (091e20e, 2023-11-25, custom), www.inkscape.org
%% Accompanies image file 'stexfris.pdf' (pdf, eps, ps)
\begingroup%
  \makeatletter%
  \providecommand\color[2][]{%
    \errmessage{(Inkscape) Color is used for the text in Inkscape, but the package 'color.sty' is not loaded}%
    \renewcommand\color[2][]{}%
  }%
  \providecommand\transparent[1]{%
    \errmessage{(Inkscape) Transparency is used (non-zero) for the text in Inkscape, but the package 'transparent.sty' is not loaded}%
    \renewcommand\transparent[1]{}%
  }%
  \providecommand\rotatebox[2]{#2}%
  \newcommand*\fsize{\dimexpr\f@size pt\relax}%
  \newcommand*\lineheight[1]{\fontsize{\fsize}{#1\fsize}\selectfont}%
  \ifx\svgwidth\undefined%
    \setlength{\unitlength}{195.59055118bp}%
    \ifx\svgscale\undefined%
      \relax%
    \else%
      \setlength{\unitlength}{\unitlength * \real{\svgscale}}%
    \fi%
  \else%
    \setlength{\unitlength}{\svgwidth}%
  \fi%
  \global\let\svgwidth\undefined%
  \global\let\svgscale\undefined%
  \makeatother%
  \begin{picture}(1,0.52173913)%
    \lineheight{1}%
    \setlength\tabcolsep{0pt}%
    \put(0,0){\includegraphics[width=\unitlength,page=1]{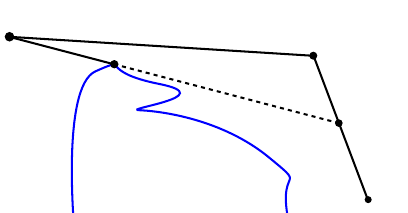}}%
    \put(0.02941814,0.45889409){\color[rgb]{0,0,0}\makebox(0,0)[lt]{\lineheight{1.25}\smash{\begin{tabular}[t]{l}$x$\end{tabular}}}}%
    \put(0.79245703,0.3989496){\color[rgb]{0,0,0}\makebox(0,0)[lt]{\lineheight{1.25}\smash{\begin{tabular}[t]{l}$y$\end{tabular}}}}%
    \put(0.83016404,0.23052122){\color[rgb]{0,0,0}\makebox(0,0)[lt]{\lineheight{1.25}\smash{\begin{tabular}[t]{l}$w$\end{tabular}}}}%
    \put(0.90647921,0.0448866){\color[rgb]{0,0,0}\makebox(0,0)[lt]{\lineheight{1.25}\smash{\begin{tabular}[t]{l}$z$\end{tabular}}}}%
    \put(0.29417137,0.36325455){\color[rgb]{0,0,0}\makebox(0,0)[lt]{\lineheight{1.25}\smash{\begin{tabular}[t]{l}$\widetilde{z}$\end{tabular}}}}%
    \put(0.62325889,0.17224768){\color[rgb]{0,0,0}\makebox(0,0)[lt]{\lineheight{1.25}\smash{\begin{tabular}[t]{l}\textcolor[rgb]{0,0,1}{$\partial E_{>0}$}\end{tabular}}}}%
  \end{picture}%
\endgroup%

\caption{Illustration for Lemma~\ref{finite-vertex lemma}}\label{ilstexf}
\end{figure}
\begin{proof}
The set $V(x)$ is open; hence, the set $[y,\,z]\setminus V(x)$ is closed. Since $y\in V(x)$ and $z\notin V(x)$, because $z\notin E_{>0}$, there exists the closest to $y$ point $w$ of the set $[y,\,z]\setminus V(x)$. Then, for any $\alpha\in (0,\,1)$, the point $(1-\alpha)y+\alpha w$ lies in $[y,\,z]\cap V(x)$. Since $z\in V_{\fin}(y,\,L,\,l)$, we have $[y,\,z]\subseteq \clos(E(y,\,L,\,l))=\clos(E(x,\,L,\,l))$. Also, $V(x)\subseteq E_{>0}$, and hence $(1-\alpha)y+\alpha w\in\clos(E(x,\,L,\,l))\cap E_{>0}=E(x,\,L,\,l)$. Tn such a case, Lemma~\ref{rasek} gives $[x,\,(1-\alpha)y+\alpha w]\subseteq E(x,\,L,\,l)$, and hence $[x,\,w]\subseteq\clos(E(x,\,L,\,l))$. Since $w\notin V(x)$, we have $[x,\,w]\cap\partial E_{>0}\ne\varnothing$. Let $\widetilde{z}\in[x,\,w]\cap\partial E_{>0}$. Then $[x,\,\widetilde{z}]\subseteq\clos(E(x,\,L,\,l))$, that is, $\widetilde{z}\in V_{\fin}(x,\,L,\,l)$. On the other hand, $\widetilde{z}\in\conv(\{x,\,w\})\subseteq\conv(\{x,\,y,\,z\})$.
\end{proof}
\begin{Le}\label{direction-vertex lemma}
Let $x\in\widetilde{E}(L,\,l)$, $y\in E(x,\,L,\,l)$ and $y\in V(x)$. Then for any directions $\vec{e}\in V_{\dir}(y,\,L,\,l)$ if $\vec{e}\notin V_{\dir}(x,\,L,\,l)$, then there exists $\widetilde{z}\in V_{\fin}(x,\,L,\,l)\cap\conv(\{x,\,y\};\;\{\vec{e}\})$.
\end{Le}
\begin{figure}[h]\centering
\def\svgwidth{7cm}
%% Creator: Inkscape 1.3.2 (091e20e, 2023-11-25, custom), www.inkscape.org
%% Accompanies image file 'stexdris.pdf' (pdf, eps, ps)
\begingroup%
  \makeatletter%
  \providecommand\color[2][]{%
    \errmessage{(Inkscape) Color is used for the text in Inkscape, but the package 'color.sty' is not loaded}%
    \renewcommand\color[2][]{}%
  }%
  \providecommand\transparent[1]{%
    \errmessage{(Inkscape) Transparency is used (non-zero) for the text in Inkscape, but the package 'transparent.sty' is not loaded}%
    \renewcommand\transparent[1]{}%
  }%
  \providecommand\rotatebox[2]{#2}%
  \newcommand*\fsize{\dimexpr\f@size pt\relax}%
  \newcommand*\lineheight[1]{\fontsize{\fsize}{#1\fsize}\selectfont}%
  \ifx\svgwidth\undefined%
    \setlength{\unitlength}{212.5984252bp}%
    \ifx\svgscale\undefined%
      \relax%
    \else%
      \setlength{\unitlength}{\unitlength * \real{\svgscale}}%
    \fi%
  \else%
    \setlength{\unitlength}{\svgwidth}%
  \fi%
  \global\let\svgwidth\undefined%
  \global\let\svgscale\undefined%
  \makeatother%
  \begin{picture}(1,0.50666667)%
    \lineheight{1}%
    \setlength\tabcolsep{0pt}%
    \put(0,0){\includegraphics[width=\unitlength,page=1]{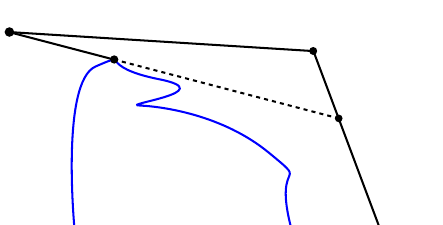}}%
    \put(0.02683539,0.45950699){\color[rgb]{0,0,0}\makebox(0,0)[lt]{\lineheight{1.25}\smash{\begin{tabular}[t]{l}$x$\end{tabular}}}}%
    \put(0.72883067,0.40435729){\color[rgb]{0,0,0}\makebox(0,0)[lt]{\lineheight{1.25}\smash{\begin{tabular}[t]{l}$y$\end{tabular}}}}%
    \put(0.76352122,0.24940114){\color[rgb]{0,0,0}\makebox(0,0)[lt]{\lineheight{1.25}\smash{\begin{tabular}[t]{l}$w$\end{tabular}}}}%
    \put(0.86747233,0.0554273){\color[rgb]{0,0,0}\makebox(0,0)[lt]{\lineheight{1.25}\smash{\begin{tabular}[t]{l}$\vec{e}$\end{tabular}}}}%
    \put(0.27040812,0.37151728){\color[rgb]{0,0,0}\makebox(0,0)[lt]{\lineheight{1.25}\smash{\begin{tabular}[t]{l}$\widetilde{z}$\end{tabular}}}}%
    \put(0.57316864,0.19578878){\color[rgb]{0,0,0}\makebox(0,0)[lt]{\lineheight{1.25}\smash{\begin{tabular}[t]{l}\textcolor[rgb]{0,0,1}{$\partial E_{>0}$}\end{tabular}}}}%
    \put(0,0){\includegraphics[width=\unitlength,page=2]{stexdris.pdf}}%
  \end{picture}%
\endgroup%

\caption{Illustration for Lemma~\ref{direction-vertex lemma}}\label{ilstexd}
\end{figure}
\begin{proof}
The proof is similar to the proof of the previous lemma. Either $[y,\,\vec{e})\subseteq V(x)$; then, for all $\gamma\geqslant 0$, the inclusion $[y+\gamma e,\,x]\subseteq E(x,\,L,\,l)$ holds, whence $[x,\,\vec{e})\subseteq\clos(E(x,\,L,\,l))$ and therefore $\vec{e}\in V_{\dir}(x,\,L,\,l)$. Or there exists the closest to $y$ point $w$ of the set $[y,\,\vec{e})\setminus V(x)$; in this case, the segment $[x,\,w]$ contains a suitable point $\widetilde{z}$.
\end{proof}
\begin{Le}\label{leprovipob}
Let $L\colon\mathbb{R}^d\to\mathbb{R}$ be an affine functional, let $l\subseteq\mathbb{R}^d$ be an affine subspace, and let $x\in\widetilde{E}(L,\,l)$. Then at least one of the following two alternatives:
\begin{enumerate}
\item
There exists a line $m$ such that $x\in m\subseteq E(x,\,L,\,l)$.
\item
The inclusion $x\in\conv(V_{\fin}(x,\,L,\,l);\;V_{\dir}(x,\,L,\,l))$ holds.
\end{enumerate}
holds.
\end{Le}
Let $V(x,\,L,\,l)=V_{\fin}(x,\,L,\,l)\cup V_{\dir}(x,\,L,\,l)$. The proof of the lemma is carried out by induction on the dimension of the extremal set. The base of the induction is essentially Proposition~\ref{basic proposition}. The induction step is a combination of Lemmas~\ref{finite-vertex lemma} and~\ref{direction-vertex lemma} with Corollary~\ref{stiag}, together with a careful analysis of cases.

\begin{proof}
We prove the lemma by induction on the dimension of the extremal set.

{\bf Base of the induction. $\dim l=1$.}

By Proposition~\ref{basic proposition}, the extremal set $E(x,\,L,\,l)$ is either a line, in which case the first condition holds, or a segment or ray with endpoints on the boundary of $E_{>0}$, in which case the second condition holds.

{\bf Step. $\dim l=k+1$.}

If $x\in\partial_{l}(E(x,\,L,\,l))\cap E(x,\,L,\,l)$, then, by Lemma~\ref{extremal boundary}, there exists an affine subspace $\widetilde{l}\subsetneq l$ such that $x\in E(x,\,L,\,\widetilde{l})$. Then, by the induction hypothesis, either there exists a line $m$ such that
\begin{equation}
x\in m\subseteq E(x,\,L,\,\widetilde{l})\overset{\text{Cor.~\ref{embedded extremal}}}{\subseteq} E(x,\,L,\,l),
\end{equation}
or
\begin{equation}
x\in\conv(V_{\fin}(x,\,L,\,\widetilde{l});\;V_{\dir}(x,\,L,\,\widetilde{l}))\overset{\text{Cor.~\ref{embedded extremal}}}{\subseteq}\conv(V_{\fin}(x,\,L,\,l);\;V_{\dir}(x,\,L,\,l)).
\end{equation}
Consequently, in what follows we may assume, without loss of generality, that $x\in\intX_{l}(E(x,\,L,\,l))$. Let $m$ be a line such that $x\in m\subseteq l$. Since $x\in\intX_{l}(E(x,\,L,\,l))$, for sufficiently small $\varepsilon>0$ the inclusion $B_{\varepsilon}(x)\cap m\subseteq\intX_{l}(E(x,\,L,\,l))$ holds. Let $\vec{e}$ and $-\vec{e}$ be the points at infinity of the line $m$. Then $m=[x,\,\vec{e})\cup[x,\,-\vec{e})$. 

First, let us show that at least one of the following three conditions:
\begin{enumerate}
\item[1$\vec{e}$)]
$\vec{e}\in V_{\dir}(x,\,L,\,l)$.
\item[2$\vec{e}$)]
$\exists v\in [x,\,\vec{e})\cap\conv(V_{\fin}(x,\,L,\,l);\;V_{\dir}(x,\,L,\,l))$.
\item[3$\vec{e}$)]
$\exists\vec{f}\in\vec{S}^{d-1}$ such that $\vec{f},\,-\vec{f}\in V_{\dir}(x,\,L,\,l)$.
\end{enumerate}
holds. Either the whole ray $[x,\,\vec{e})$ lies in $E(x,\,L,\,l)$; then $\vec{e}\in V_{\dir}(x,\,L,\,l)$ and condition 1$\vec{e}$) holds. Or the set $[x,\,\vec{e})\setminus\intX_l(E(x,\,L,\,l))$ has a point $y$ nearest to $x$. Then $y\in\partial_{l}(E(x,\,L,\,l))$. By Lemma~\ref{extremal boundary}, there are two possibilities. The first is that $y\in\partial E_{>0}$; then
\begin{equation}
y\in V_{\fin}(x,\,L,\,l)\cap [x,\,\vec{e})\subseteq [x,\,\vec{e})\cap\conv(V_{\fin}(x,\,L,\,l);\;V_{\dir}(x,\,L,\,l))
\end{equation}
and condition 2$\vec{e}$) holds. The second possibility is that $y\in E(x,\,L,\,l)$ and there exists an affine subspace $l'\subsetneq l$ such that $y\in\widetilde{E}(L,\,l')$. We may assume, without loss of generality, that we are in the latter case; in particular, $y\in V(x)$. Then, by the induction hypothesis, at least one of the following two alternatives:
\begin{enumerate}
\item[1$y$)]
There exists a line $\widetilde{m}$ such that $y\in\widetilde{m}\subseteq E(y,\,L,\,l')\subseteq E(y,\,L,\,l)=E(x,\,L,\,l)$. Then there exists a direction $\vec{f}\in\vec{S}^{d-1}$ such that $\vec{f},\,-\vec{f}\in V_{\dir}(y,\,L,\,l)$.
\item[1$y$)]
$y\in\conv(V_{\fin}(y,\,L,\,l');\;V_{\dir}(y,\,L,\,l'))\subseteq \conv(V_{\fin}(y,\,L,\,l);\;V_{\dir}(y,\,L,\,l)).$
\end{enumerate}
holds. By Lemmas~\ref{finite-vertex lemma} and~\ref{direction-vertex lemma}, the points $x,\,y$ and the sets $V=V_{\fin}(y,\,L,\,l)$, $V'=V_{\fin}(x,\,L,\,l)$, $E=V_{\dir}(y,\,L,\,l)$, and $E'=V_{\dir}(x,\,L,\,l)$ satisfy the assumptions of Corollary~\ref{stiag}. Hence the conclusion of Corollary~\ref{stiag} gives either condition 2$\vec{e}$) or condition 3$\vec{e}$). 

Thus, at least one of the three conditions holds; moreover, at least one of the three analogous conditions for $-\vec{e}$. 
\begin{enumerate}
\item[1$-\vec{e}$)]
$-\vec{e}\in V_{\dir}(x,\,L,\,l)$.
\item[2$-\vec{e}$)]
$\exists v'\in [x,\,-\vec{e})\cap\conv(V_{\fin}(x,\,L,\,l);\;V_{\dir}(x,\,L,\,l))$.
\item[3$-\vec{e}$)]
$\exists\vec{f}'\in\vec{S}^{d-1}$ such that $\vec{f}',\,-\vec{f}'\in V_{\dir}(x,\,L,\,l)$.
\end{enumerate}
holds.

If there exists a direction $\vec{f}'\in\vec{S}^{d-1}$ such that $\vec{f}',\,-\vec{f}'\in V_{\dir}(x,\,L,\,l)$, then there exists a line $m'=\conv(\{x\};\;\{\vec{f}',\,-\vec{f}'\})$ such that $x\in m'\subseteq\clos(E(x,\,L,\,l))$. Then either $m'\subseteq E(x,\,L,\,l)$, or there exists
\begin{equation}
y\in m'\setminus E(x,\,L,\,l)\subseteq\clos(E(x,\,L,\,l))\setminus E(x,\,L,\,l)\subseteq \partial E_{>0}.
\end{equation}
Then $y\in V_{\fin}(x,\,L,\,l)$ and
\begin{equation}
x\in m'=\conv(\{y\};\;\{\vec{f},\,-\vec{f}\})\subseteq\conv(V_{\fin}(x,\,L,\,l);\;V_{\dir}(x,\,L,\,l)).
\end{equation}
Consequently, without loss of generality, we may assume that, for $\vec{e}$ and for $-\vec{e}$, at least one of the conditions 1$\vec{e}$), 2$\vec{e}$) and at least one of the conditions 1$-\vec{e}$), 2$-\vec{e}$), respectively, holds. If conditions 1$\vec{e}$) and 1$-\vec{e}$) hold simultaneously, then we are in the case analogous to the third condition. Therefore, without loss of generality, condition 2$\vec{e}$) holds, that is,
\begin{equation}
\exists v\in [x,\,\vec{e})\cap\conv(V_{\fin}(x,\,L,\,l);\;V_{\dir}(x,\,L,\,l)).
\end{equation}
If condition 1$-\vec{e}$) holds for $-\vec{e}$, then
\begin{equation}
x\in [v,\,-\vec{e})\subseteq\conv(V_{\fin}(x,\,L,\,l);\;V_{\dir}(x,\,L,\,l)).
\end{equation}
If condition 2$-\vec{e}$) holds, then
\begin{equation}
x\in [v,\,v']\subseteq\conv(V_{\fin}(x,\,L,\,l);\;V_{\dir}(x,\,L,\,l)).
\end{equation}
\end{proof}

\begin{proof}[Proof Theorem~\ref{mainth}]
By Remark~\ref{closed-pair remark} and Lemma~\ref{leprovipob}, the pair of sets 
\begin{equation}
(V_{\fin}(x,\,L,\,l);\; V_{\dir}(x,\,L,\,l))
\end{equation}
is closed, and at least one of the two alternatives: Either there exists a line $m$ such that $x\in m\subseteq E(x,\,L,\,l)$, or
\begin{equation}
x\in\conv(V_{\fin}(x,\,L,\,l);\;V_{\dir}(x,\,L,\,l))
\end{equation}
holds. Assume, without loss of generality, that the second alternative holds. Then, by Corollary~\ref{tsorsim}, there exists an inclusion-minimal generalized simplex $C$ such that $x\in C$ and
\begin{equation}\label{vk1id}
V_{\fin}(C)\subseteq V_{\fin}(x,\,L,\,l)
\end{equation}
and
\begin{equation}\label{vk2id}
V_{\dir}(C)\subseteq V_{\dir}(x,\,L,\,l).
\end{equation}
Let us show that $C$ satisfies the assumptions of the theorem. First note that $x\in\ri C$; otherwise $C$ could be reduced to one of its faces. From the inclusions~\eqref{vk1id} and~\eqref{vk2id} it follows that, for any $v\in V_{\fin}(C)$, the inclusion $[x,\,v]\subseteq\clos(E(x,\,L,\,l))$ holds, and, for any $\vec{e}\in V_{\dir}(C)$, the inclusion $[x,\,\vec{e})\subseteq\clos(E(x,\,L,\,l))$ holds as well. Let us show that, in fact, $[x,\,v)\subseteq E(x,\,L,\,l)$ and $[x,\,\vec{e})\subseteq E(x,\,L,\,l)$. Indeed, if there exists a point $y$ lying in $[x,\,v)\setminus E(x,\,L,\,l)$ or in $[x,\,\vec{e})\setminus E(x,\,L,\,l)$, then $y\in\partial E_{>0}$ and $[x,\,y]\subseteq \clos(E(x,\,L,\,l))$. Thus, $y\in V_{\fin}(x,\,L,\,l)$, and the generalized simplex $C$ can be reduced by replacing the corresponding vertex ($v$ or $\vec{e}$) by $y$. 

Let $v\in V_{\fin}(C)$. For $\alpha\in [0,\,1)$, denote by $v_{\alpha}$ the point defined by the formula
\begin{equation}
v_{\alpha}=\alpha v+(1-\alpha)x.
\end{equation}
Let $\vec{e}\in V_{\dir}(C)$. For $\alpha\in [0,\,1)$, denote by $e_{\alpha}$ the point given by the formula
\begin{equation}
e_{\alpha}=x+\frac{\alpha}{1-\alpha} e.
\end{equation}
Let 
\begin{equation}
C_{\alpha}=\conv(\{v_{\alpha}|\,v\in V_{\fin}(C)\}\cup\{e_{\alpha}|\,\vec{e}\in V_{\dir}(C)\}).
\end{equation}
As $\alpha$ increases from zero to $1$, the simplex $C_{\alpha}$ continuously expands from the set $\{x\}$ to the generalized simplex $C$.

If $C_{\alpha}\cap\partial E_{>0}=\varnothing$, then $C_{\alpha}\subseteq E_{>0}$, and by Corollary~\ref{vipukn} we obtain $C_{\alpha}\subseteq E(x,\,L,\,l)$. Note that, for $\beta\geqslant\alpha$, the inclusion $C_{\alpha}\subseteq C_{\beta}$ is true. Thus, either $C_{\alpha}\subseteq E(x,\,L,\,l)$ for any $\alpha\in [0,\,1)$, or there exists a number $\alpha_0\in [0,\,1)$ such that, for $\alpha<\alpha_0$, the inclusion $C_{\alpha}\subseteq E(x,\,L,\,l)$ holds, and, for $\alpha>\alpha_0$, the inequality $C_{\alpha}\cap\partial E_{>0}\ne\varnothing$ holds. We show that the second case is impossible.

Note that if $C_{\alpha_0}\cap\partial E_{>0}=\varnothing$, then $\dist(C_{\alpha_0},\,\partial E_{>0})>\delta>0$. In this case, for sufficiently small $\varepsilon>0$ the inequalities $\dist(v_{\alpha_0+\varepsilon},\,v_{\alpha_0})<\frac{\delta}{2}$ for any $v\in V_{\fin}(C)$ and $\dist(e_{\alpha_0+\varepsilon},\,e_{\alpha_0})<\frac{\delta}{2}$ for any $\vec{e}\in V_{\dir}(C)$ hold. Then $C_{\alpha_0+\varepsilon}\subseteq C_{\alpha_0}+B_{\frac{\delta}{2}}(0)\subseteq E_{>0}$, a contradiction. In particular, $\alpha_0\ne 0$.

Thus, $C_{\alpha_0}\cap\partial E_{>0}\ne\varnothing$. Note that $\ri(C_{\alpha_0})=\bigcup_{\alpha<\alpha_0}C_{\alpha}$. Let $y\in C_{\alpha_0}\cap\partial E_{>0}$. Then $[x,\,y)\subseteq\ri C_{\alpha_0}\subseteq E(x,\,L,\,l)$. Thus, $y\in V_{\fin}(x,\,L,\,l)$. In this case, $y\in C_{\alpha_0}\subseteq\ri C$. Therefore, $y$ splits $C$ into $d+1$ simplices (by replacing one of its vertices by $y$). One of them contains $x$, which contradicts the minimality of~$C$.

Thus, for any $\alpha\in [0,\,1)$ the inclusion $C_{\alpha}\subseteq E(x,\,L,\,l)$ holds. That is,
\begin{equation}
\ri C=\bigcup\limits_{\alpha\in[0,\,1)}C_{\alpha}\subseteq E(x,\,L,\,l).
\end{equation}
Thus, in particular, $C\subseteq\clos(E(x,\,L,\,l))$. If $C\cap\partial E_{>0}\ne V_{\fin}(C)$, then there exists $y\in C\cap\partial E_{>0}\setminus V_{\fin}(C)$. Then $y\in V_{\fin}(x,\,L,\,l)$ and $C$ can be reduced by replacing some vertex by $y$. Thus, $C\cap\partial E_{>0}=V_{\fin}(C)$, and since $C\subseteq\clos(E(x,\,L,\,l))$, this means that $C\setminus V_{\fin}(C)\subseteq E(x,\,L,\,l)$.

It remains to show that $\#(V_{\dir}(C))\leqslant 1$. Let $V_{\fin}(C)=\{v_1,\dots,\,v_k\}$, $V_{\dir}(C)=\{\vec{e}_1,\dots,\,\vec{e}_r\}$. The point $x$ is expressed as a convex combination of the vertices of $C$:
\begin{equation}
x=\sum\limits_{j=1}^{k}\alpha_j v_j+\sum\limits_{i=1}^{r}\beta_i e_i.
\end{equation}
Since $x\in\ri C$, the inequalities $\alpha_j>0$ hold for every $j=1,\dots,k$ and $\beta_i>0$ for every $i=1,\dots,r$. However if $r>1$ then $C$ can be reduced to a simplex 
\begin{equation}
\widetilde{C}
=
\conv\left(\{v_1,\dots,\,v_k\};\;\left\{\overrightarrow{\left(\frac{\sum\limits_{i=1}^{r}\beta_i e_i}{\left|\sum\limits_{i=1}^{r}\beta_i e_i\right|}\right)}\right\}\right),
\end{equation}
which contradicts the minimality of $C$.
\end{proof}

\section{On one-dimensional extremal sets}\label{emprul}
This section is devoted to a precise formulation and proof of the following empirical observation: a one-dimensional extremal set may approach the free boundary only tangentially, not transversally. By the free boundary we mean the part of the boundary where the obstacle is not defined or does not coincide with the value of the function. This principle was first observed in the paper~\cite{SlVanir}. It plays an important role in construction of extremal set functions, so-called optimizers, in Bellman function theory, see the fifth section of~\cite{r6} and the fifth section of~\cite{talmudeIII}.
\begin{Th}\label{tangency}
Let $l\subseteq\mathbb{R}^d$ be a line, let $L\colon\mathbb{R}^d\to\mathbb{R}$ be an affine functional, let $x\in\widetilde{E}(L,\,l)$, and let $y\in\partial_l E(x,\,L,\,l)$. Suppose that the pair of points $x,\,y$ satisfies the following cone condition: there exists $\varepsilon>0$ such that, for any point $z\in B_{\varepsilon}(y)\cap\clos\Omega$, the inclusion $[x,\,z)\subseteq\intX\Omega$ holds. Then $L(y)=\supfu{f}(y)$.
\end{Th}
In the case where the domain is closed and the function $f$ is upper semicontinuous, the latter condition is replaced by $L(y)=f(y)$. If the domain $\Omega$ is $C^1$-smooth, then the cone condition turns into the condition that the segment $[x,\,y]$ is not tangent to $\partial\Omega$ at the point $y$. Thus, informally, this theorem says that a one-dimensional extremal set is either tangent to the boundary at its endpoint or its endpoint is an ``obstacle point''. 

Before passing to the proof of the theorem, we introduce auxiliary constructions and describe their basic properties. Let $z\in B_{\varepsilon}(y)\cap\Omega$. Then, by the cone condition, $[x,\,z]\subseteq\Omega$, and hence, by Lemma~\ref{neighborhood inequality}, $f(z)\leqslant \BG(z)\leqslant L(z)$. Thus, $L(y)\geqslant\supfu{f}(y)$, and it suffices to check that $L(y)\leqslant\supfu{f}(y)$.

If $y\in\intX\Omega$, then, since $y\in\partial_l E(x,\,L,\,l)$, Corollary~\ref{gronedim} gives $y\in\partial E_{>0}\setminus\partial\Omega\subseteq E_0$. Thus, $L(y)=\BG(y)=\supfu{f}(y)$. Therefore, in what follows we may assume, without loss of generality, that $y\in\partial\Omega$.

Assume the contrary: let $L(y)>\supfu{f}(y)$. Further, without loss of generality, we will assume that $L\equiv 0$, this can be achieved by replacing $f$ by $f-L$.

Choose a number $\xi>0$ so small that the inequality $|\xi y-\xi x|<\frac{\varepsilon}{2}$ holds. Then $(1+\xi)y-\xi x,\,(1-\xi)y+\xi x\in B_{\frac{\varepsilon}{2}}(y)$. By the cone condition and since $y\in\partial\Omega$, the inclusions $(1-\xi)y+\xi x\in\intX\Omega$ and $(1+\xi)y-\xi x\in\mathbb{R}^d\setminus\clos\Omega=\intX(\mathbb{R}^d\setminus\Omega)$ hold. Let $\tau>0$ be a number satisfying the following inequality:
\begin{equation}\label{taudef}
\tau<\min\left\{\frac{\varepsilon}{2},\,\dist((1-\xi)y+\xi x,\,\partial\Omega),\,\dist((1+\xi)y-\xi x,\,\partial\Omega)\right\}.
\end{equation}
Then $B_{\tau}((1+\xi)y-\xi x)\subseteq\intX(\mathbb{R}^d\setminus\Omega)$ and $B_{\tau}((1-\xi)y+\xi x)\subseteq\intX\Omega$. Let $M$ be the affine hyperplane perpendicular to $l$ and passing through the point $(1+\xi)y-\xi x$, and let $M'$ be the analogous hyperplane passing through the point $(1-\xi)y+\xi x$. Let $W$ be the cone
\begin{equation}\label{defW}
W=\left\{z\in\mathbb{R}^d\big|\,\exists v\in B_{\tau}((1+\xi)y-\xi x)\cap M\text{ such that }z\in[x,\,v]\right\}=\conv\left(\{x\}\cup(B_{\tau}((1+\xi)y-\xi x)\cap M)\right),
\end{equation}
see Fig. \ref{Figmnw}.
\begin{figure}[h]\centering
\def\svgwidth{10cm}
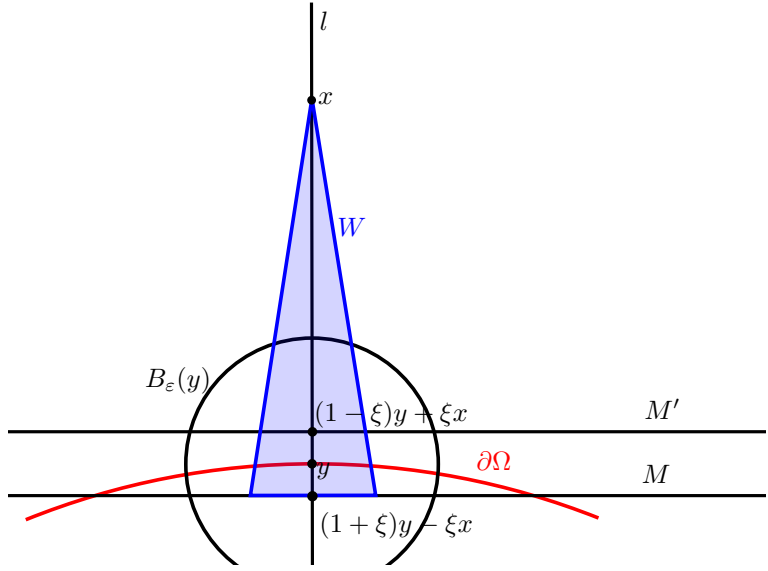
\caption{Construction of the set $W$ in~\eqref{defW}}\label{Figmnw}
\end{figure}
Let the sets $U$ and $F$ be given by the following formulas:
\begin{equation}
U=W\cap\Omega=\left\{z\in\Omega\big|\,\exists v\in B_{\tau}((1+\xi)y-\xi x)\cap M\text{ such that }z\in[x,\,v]\right\},
\end{equation}
\begin{equation}
F=\partial W\cap\Omega=\left\{z\in\Omega\big|\,\exists v\in (\partial B_{\tau}((1+\xi)y-\xi x))\cap M\text{ such that }z\in[x,\,v]\right\}.
\end{equation}
Since $F=(\partial W)\cap\Omega$ and $U=W\cap\Omega$, the inclusion $\partial_{\Omega}U\subseteq F$ holds; and since $W$ is closed, $U$ is closed in $\Omega$. Let the function $g\colon U\to\mathbb{R}\cup\{-\infty,\,+\infty\}$ be defined by the formula
\begin{equation}
g(z)=
\begin{cases}
\BG(z), & z\in F,\\
f(z), & z\in U\setminus F.
\end{cases}
\end{equation}
By the hereditary principle for minimal locally concave functions (Theorem~\ref{prinnasotgr}), the equality $\BG_{f,\,\Omega}|_U=\BG_{g,\,U}$ holds.
\begin{St}\label{closure proposition}
If $z\in\clos U$, then $[x,\,z)\subseteq U\cap \intX\Omega$.
\end{St}
\begin{proof}
Since $\clos U\subseteq W\cap\clos\Omega$, there exists a point $v\in B_{\tau}((1+\xi)y-\xi x)\cap M$ such that $z\in[v,\,x]$. Let $v'$ be the point of intersection of the segment $[v,\,x]$ and the hyperplane $M'$. By the definition of $\xi$ and the choice of $\tau$, the inclusion $v'\in\intX\Omega$ holds and $|v-y|<\varepsilon$. Thus, by the cone condition, $\left[x,\,v'\right)\subseteq\intX\Omega$. This proves the proposition in the case $z\in\left[x,\,v'\right]$. In the remaining case, if $z\in\left[v,\,v'\right]$, then
\begin{equation}
|z-y|\leqslant\max\left\{|v-y|,\,\left|v'-y\right|\right\}<\varepsilon,
\end{equation}
and the proposition follows from the cone condition.
\end{proof}

Let the set $H$ be defined by the formula
\begin{equation}
H=\left\{z\in\clos U\middle|\,\supfu{g}(z)=0\right\}.
\end{equation}
Then $H$ is compact.
\begin{St}\label{positive proposition}
Let $z\in H$. Then $\BG|_{[x,\,z)}\equiv 0$.
\end{St}
\begin{proof}
By the definition of $H$, for any $\delta>0$ there exists a point $v_{\delta}\in U$ such that $|z-v_{\delta}|<\delta$ and $g(v_{\delta})>-\delta$. Then, by Proposition~\ref{closure proposition}, $[x,\,v_{\delta}]\subseteq U$, and by the concavity of $\BG$ the inequality $\BG|_{[x,\,v_{\delta}]}>-\delta$ holds. By the continuity of $\BG$ at interior points, we obtain $\BG|_{[x,\,z)}\geqslant 0$. Since $\BG|_U\leqslant 0$, the equality $\BG|_{[x,\,z)}\equiv 0$ holds.
\end{proof}
\begin{Rem}
The inclusion $x\in H$ holds.
\end{Rem}
\begin{proof}[Proof of Theorem~\ref{tangency}.]
Let $z=\frac{x+y}{2}$. Consider two cases.

{\bf 1)$z\notin\conv H$.}

In this case, there exists an affine hyperplane $N$ passing through the point $z$ that does not intersect the set $\conv H$. The hyperplane $N$ divides the space $\mathbb{R}^d$ into two closed half-spaces $N^-$ and $N^+$. We may assume, without loss of generality, that $x\in N^+$. Then $H\subseteq N^+$. Let $\zeta=\max\{\supfu{g}(v)|\,v\in N^{-}\cap \clos U\}<0$; otherwise, there would be a point $w\in N^{-}\cap\clos U$ at which $\supfu{g}(w)=0$. Thus, $w\in N^{-}\cap H=\varnothing$.

Now let $L'\colon\mathbb{R}^d\to \mathbb{R}$ be an affine functional that vanishes on $N$, is positive on $N^+\setminus N$, and satisfies $L'|_{U}>\zeta$. Then $L'|_{U}\geqslant g$, and hence $L'|_{U}\geqslant\BG|_{U}$. In this case $L'(x)>0$ and $L'\left(\frac{x+y}{2}\right)=0$, and hence $\BG|_{\left(y,\,\frac{x+y}{2}\right)}\leqslant L'|_{\left(y,\,\frac{x+y}{2}\right)}<0$, which contradicts $[x,\,y)\subseteq E(x,\,0,\,l)$.

{\bf 2)$z\in\conv H$.}

\begin{figure}[h]\centering
\def\svgwidth{7cm}
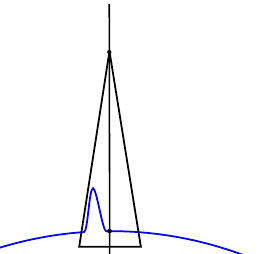
\caption{Illustration of the case $z\in\conv H$}
\end{figure}
In this case, there exists a simplex $C=\conv(\{v_1,\dots,\,v_k\})$ such that $v_1,\dots,\,v_k\in~H$ and $z\in\ri(C)$. Let $\lambda\in(0,\,1)$ be so small that
\begin{equation}
\conv(\lambda C+(1-\lambda)x)\subseteq E_{>0}.
\end{equation}
Such a $\lambda$ exists since $x\in E_{>0}$ and the set $E_{>0}$ is open. By Proposition~\ref{positive proposition}, $\BG((1-\lambda)x+\lambda v_k)=0$, and hence, by the concavity and nonpositivity of $\BG$ on $U$, we have
\begin{equation}
\BG|_{\conv(\lambda C+(1-\lambda)x)}
\equiv
0.
\end{equation}
Let $\gamma>0$ be so small that $\gamma<\dist(\conv(\lambda C+(1-\lambda)x),\,\partial E_{>0})$. Let $V=\conv(\lambda C+(1-\lambda)x)+B_{\gamma}(0)$. On the one hand, $\lambda z+(1-\lambda)x\in\ri(\lambda C+(1-\lambda)x)$ and $\lambda C+(1-\lambda)x\subseteq U(V,\,0)$, where $U(V,\,0)$ is taken from Definition~\ref{mnozhsopdef}. On the other hand, $\lambda z+(1-\lambda)x\in l$, and by Definition~\ref{oprpredex} of pre-extremal sets and Lemma~\ref{usopr}, the set $l\cap U(V,\,0)$ is a face of $U(V,\,0)$. Thus, by Proposition~\ref{grvipvnutgr}, $\lambda C+(1-\lambda)x\subseteq l$, and therefore $v_1,\dots,\,v_k\in l$ and $z\in\conv(l\cap H)$.

It remains to show that $l\cap H=\{x\}$. Indeed, $l\cap H\subseteq l\cap\clos U=[x,\,y]$. In this case, $y\notin H$ by the assumption $L(y)>\supfu{f}(y)$. If, for some $\alpha\in(0,\,1)$, the inclusion $\alpha x+(1-\alpha) y\in H$ holds, then
\begin{equation}
\supfu{g}(\alpha x+(1-\alpha) y)
=
\supfu{f}(\alpha x+(1-\alpha) y)
=
0
=
\BG(\alpha x+(1-\alpha) y),
\end{equation}
which means that $\alpha x+(1-\alpha) y\in E_0$. This contradicts $\alpha x+(1-\alpha) y\in E(x,\,0,\,l)\subseteq E_{>0}$.

\end{proof}

\subsection{Transversal extremal sets of higher dimensions}
The tangency principle for extremal sets relies essentially on their one-dimensionality and is not true for extremal sets of dimension at least $2$. This can be seen from the following example. 

\begin{Ex}\label{transversal example}
Let $d\geqslant 0$ and $k\geqslant 2$ be integers, let $g\colon\mathbb{R}^d\to\mathbb{R}$ be a concave function, and let $\Omega$ be the spherical annulus
\begin{equation}
\Omega=\left\{x\in\mathbb{R}^{d+k}\big|\,|x|\in [1-\varepsilon,\,1]\right\}.
\end{equation}
Let the function $f\colon\Omega\to\re$ be defined by the formula
\begin{equation}
f(x)=
\begin{cases}
g(x_1,\dots,\,x_d), &\text{if } |x|=1,\\
-\infty, &\text{if } |x|<1.
\end{cases}
\end{equation}
Then $\BG_{f,\,\Omega}(x)=g(x_1,\dots,\,x_d)$. If, in addition, the function $g$ is strictly concave, then the extremal sets of the function $\BG_{f,\,\Omega}$ are the sets $\{y\in\intX\Omega|\,y_1=x_1,\dots,\,y_d=x_d\}$. 
\end{Ex}
\begin{proof}
To see that $\BG(x)=g(x_1,\dots,\,x_d)$, first note that the function $x\mapsto g(x_1,\dots,\,x_d)$ is concave and is greater than or equal to $f$ on $\Omega$; hence $\BG(x)\leqslant g(x_1,\dots,\,x_d)$. The reverse inequality is obtained as follows. Let $x\in\Omega$. Then the section $\{y\in\intX\Omega|\,y_1=x_1,\dots,\,y_d=x_d\}$ is a spherical annulus of dimension $k\geqslant 2$, and we can draw in it a segment $[y,\,z]$ with vertices on the outer boundary and passing through $x$. Then
\begin{equation}
\BG(x)\geqslant\alpha\BG(y)+(1-\alpha)\BG(z)\geqslant \alpha f(y)+(1-\alpha)f(z)=g(x_1,\dots,\,x_d).
\end{equation}
In particular, in this case $E_0=\varnothing$.

Now let $g$ be a strictly concave function. Let $x\in\widetilde{E}(L,\,l)$, $0<\delta<\dist(x,\,\partial\Omega)$, and $y\in B_{\delta}(x)\cap E(x,\,L,\,l)$. Then, by Corollary~\ref{segment equality}, $\BG|_{[x,\,y]}=L|_{[x,\,y]}$, and the function $g|_{[(x_1,\dots,\,x_d),\,(y_1,\dots,\,y_d)]}$ is affine. By the strict concavity of $g$, this means that $(x_1,\dots,\,x_d)=(y_1,\dots,\,y_d)$. By Lemma~\ref{usopr},
\begin{equation}
l=\aff(l\cap U(B_{\delta}(x),\,L))\subseteq\aff(\{z\in\intX\Omega|\,z_1=x_1,\dots,\,z_d=x_d\}).
\end{equation}
On the other hand, $\BG$ is constant on the section $\{z\in\intX\Omega|\,z_1=x_1,\dots,\,z_d=x_d\}$. Therefore, any affine functional $L$ locally supporting $\BG$ at $x$ coincides with $\BG$ on the whole section. It follows that $l=\aff(\{z\in\intX\Omega|\,z_1=x_1,\dots,\,z_d=x_d\})$ and $E(x,\,L,\,l)=\{z\in\intX\Omega|\,z_1=x_1,\dots,\,z_d=x_d\}$.
\end{proof}
In particular, in this example one of the extremal sets is the set
\begin{equation}
\left\{x\in\mathbb{R}^{d+k}\big|\,x_1=\dots=x_d=0,\,|(x_{d+1},\dots,\,x_{d+k})|\in(1-\varepsilon,\,1)\right\}.
\end{equation}
This diametral section of the spherical annulus has dimension $k$ and approaches the inner sphere transversally. The obstacle $f$ is equal to $-\infty$ there.
\begin{figure}[h]\centering
\def\svgwidth{6cm}
%% Creator: Inkscape 1.3.2 (091e20e, 2023-11-25, custom), www.inkscape.org
%% Accompanies image file 'sfv11.pdf' (pdf, eps, ps)
\begingroup%
  \makeatletter%
  \providecommand\color[2][]{%
    \errmessage{(Inkscape) Color is used for the text in Inkscape, but the package 'color.sty' is not loaded}%
    \renewcommand\color[2][]{}%
  }%
  \providecommand\transparent[1]{%
    \errmessage{(Inkscape) Transparency is used (non-zero) for the text in Inkscape, but the package 'transparent.sty' is not loaded}%
    \renewcommand\transparent[1]{}%
  }%
  \providecommand\rotatebox[2]{#2}%
  \newcommand*\fsize{\dimexpr\f@size pt\relax}%
  \newcommand*\lineheight[1]{\fontsize{\fsize}{#1\fsize}\selectfont}%
  \ifx\svgwidth\undefined%
    \setlength{\unitlength}{364.5bp}%
    \ifx\svgscale\undefined%
      \relax%
    \else%
      \setlength{\unitlength}{\unitlength * \real{\svgscale}}%
    \fi%
  \else%
    \setlength{\unitlength}{\svgwidth}%
  \fi%
  \global\let\svgwidth\undefined%
  \global\let\svgscale\undefined%
  \makeatother%
  \begin{picture}(1,0.99794239)%
    \lineheight{1}%
    \setlength\tabcolsep{0pt}%
    \put(0,0){\includegraphics[width=\unitlength,page=1]{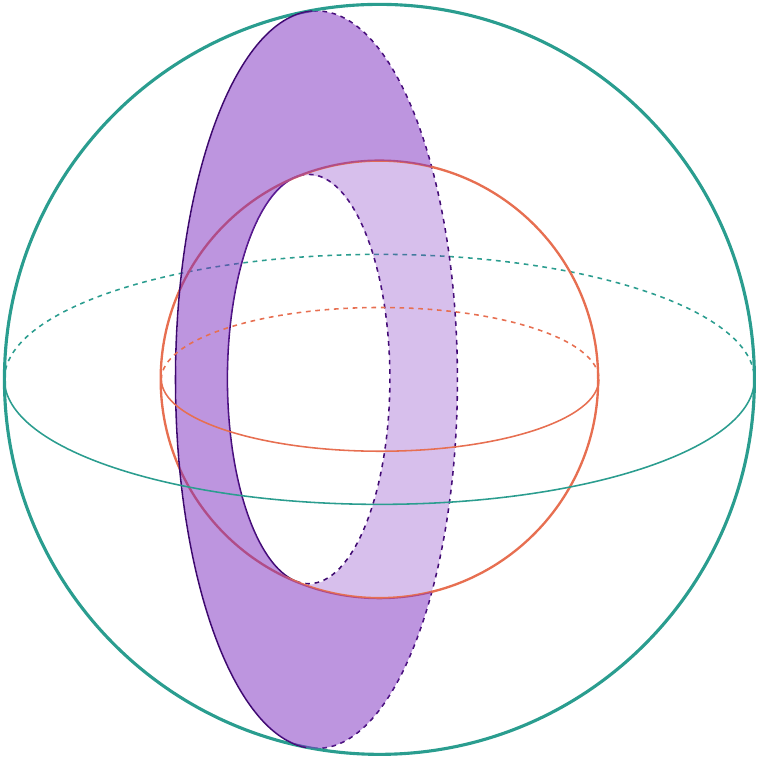}}%
    \put(0.50775178,0.92317659){\color[rgb]{0,0,0}\makebox(0,0)[lt]{\lineheight{1.25}\smash{\begin{tabular}[t]{l}\textcolor[rgb]{0.484375, 0.17578125, 0.75}{$E(x,\,L,\,l)$}\end{tabular}}}}%
  \end{picture}%
\endgroup%

\caption{Illustration for Example~\ref{transversal example}. Case $d=0$, $k=2$.}\label{sphere figure11}
\end{figure}

\section{On the case where a strictly concave $C^2$ majorant exists}\label{c2slmn}

Let the domain $\Omega\subseteq\mathbb{R}^d$ be open and bounded and have sufficiently smooth boundary, and let $f\colon\partial\Omega\to\mathbb{R}$ be such that there exists a strictly locally concave function $W\in C^2(\clos\Omega)$ satisfying $W|_{\partial\Omega}=f$. In this subsection, by $\BG$ we mean $\BG_{f,\clos\Omega}$; by $C^2(\clos\Omega)$ we mean functions that extend to a neighborhood of $\clos\Omega$ as $C^2$-smooth functions; and by a strictly locally concave function we mean a locally concave function without segments of linearity. This case was considered by Guan in~\cite{Guan}. In this subsection, we show which additional restrictions on the structure of a minimal locally concave function are implied by the existence of such a majorant, as compared with the general case. The main difference is that the existence of the majorant guarantees the absence of tangent extremal sets.
\begin{Le}
Let the function $\BG$ have a strictly locally concave majorant $W\in C^1(\clos\Omega)$. Suppose that the boundary of $\Omega$ is $C^1$-smooth, $[x,y)\subseteq\Omega$, $y\in\partial\Omega$, and the segment $[x,\,y]$ is tangent to $\partial\Omega$ at $y$. Then the half-interval $[x,\,y)$ cannot be a part of any extremal set.
\end{Le}
\begin{proof}
We argue by contradiction. Suppose that such an extremal set exists. Without loss of generality, let it be $E(x,\,0,\,l)$ for some subspace $l$ containing the points $x,\,y$. Let $e$ be the outer normal to~$\Omega$ at $y$. Since $[x,\,y]$ is tangent to $\Omega$, the ray $(t x+(1-t)y,\,\vec{e})$ intersects $\partial\Omega$ for sufficiently small $t\in(0,\,1)$. Let $v_t$ be the first point of $\partial\Omega$ on this ray. Then, by Lemma~\ref{neighborhood inequality}, $\BG(v_t)\leqslant 0$. Hence $W(v_t)\leqslant 0$, since the boundary values of both $\BG$ and $W$ coincide with $f$. Since $[x,\,y]$ is tangent, the derivative of $W$ at $y$ in the direction $x-y$ can be written as
\begin{equation}
\frac{\partial W}{\partial(x-y)}(y)=\lim\limits_{t\to 0_+}\frac{W(v_t)-W(y)}{t}\leqslant 0.
\end{equation}
However, since $W$ is strictly locally concave, it follows that $W(x)<0$, which contradicts the fact that $\BG(x)=0$, and $W$ is a majorant of $\BG$.
\end{proof}

Guan proved the following lemma in his paper, which is an analogue of Lemma 2 of~\cite{CafNirSpr}, Lemma 3.2 of~\cite{DePhiFig} and Theorem~\ref{mainth} of this paper.
\begin{Le}[Lemma 3.1 of~\cite{Guan}]\label{LeGuan}
Let the function $\BG$ have a strictly locally concave majorant $W\in C^2(\clos\Omega)$, let $x\in\Omega$, $\BG(x)=0$, and suppose that, in a neighborhood of $x$, the function $\BG$ is nonpositive. Let $S_x$ be the connected component of $x$ in the set $\{y\in\clos\Omega|\,\BG(y)=0\}$. Then
\begin{equation}
S_x=\conv_{\clos\Omega}(S_x\cap\partial\Omega).
\end{equation}
\end{Le}
The proof of the lemma is based on the following proposition.
\begin{St}\label{StGuan}
The set $S_x$ is relatively convex in $\clos(\Omega)$ \textup{(}see Definition~\ref{otnositvip}\textup{)}.
\end{St}
This proposition is not stated as a separate lemma in Guan's paper, but is postulated at the beginning of the proof of Lemma 3.1. As a result, in this case Guan proves the following theorem.
\begin{Th}[Theorem 1.2 of~\cite{Guan}]\label{ThGuan}
Let the function $\BG$ have a strictly locally concave majorant $W\in C^2(\clos\Omega)$, let the boundary of the domain $\Omega$ be $C^{3,1}$-smooth, and let $f\in C^{3,1}(\partial\Omega)$. Then $\BG\in C^{1,1}(\clos\Omega)$.
\end{Th}
In the case where the majorant $W$ does not exist, the proposition, lemma, and theorem may fail. The key idea is the choice of the correct part $S_x$, which is precisely the construction of extremal sets. A counterexample to all three propositions is the function computed in the paper~\cite{SlVanir}. Let 
\begin{equation}
\Omega_{\varepsilon}=\{(x_1,\,x_2)\in\mathbb{R}^2|\,x_1^2\leqslant x_2\leqslant x_1^2+\varepsilon^2\}.
\end{equation}
Let the function $f$ be defined by the formula $f=e^{x_1}|_{\{x_2=x_1^2\}}$, and $B_{\varepsilon}=\BG_{f,\,\Omega_{\varepsilon}}$. In~\cite{SlVanir}, the following formula was proved
\begin{equation}
B_{\varepsilon}(x_1,\,x_2)=\frac{1-\sqrt{\varepsilon^2+x_1^2-x_2}}{1-\varepsilon}\exp\left(x_1+\sqrt{\varepsilon^2+x_1^2-x_2}-\varepsilon\right),
\end{equation}
\begin{equation}
\frac{\partial^2B_{\varepsilon}}{\partial x_2^2}(x_1,\,x_2)=-\frac{1}{4(1-\varepsilon)\sqrt{\varepsilon^2+x_1^2-x_2}}\exp\left(x_1+\sqrt{\varepsilon^2+x_1^2-x_2}-\varepsilon\right).
\end{equation}
The first of these is formula (3.5) of~\cite{SlVanir}, the second is the formula immediately preceding formula (4.8) in the same paper.
The upper boundary of the domain $\Omega_{\varepsilon}$ is the parabola $\{x_2=x_1^2+\varepsilon^2\}$. As the point $(x_1,\,x_2)$ tends to the upper boundary, the second derivative of the function with respect to $x_2$ tends to infinity, which shows that the function cannot be $C^{1,1}$-smooth up to the boundary. A more careful study of the formula shows that, although away from the free boundary the function is infinitely smooth, it approaches the free boundary only $C^{1,\frac{1}{2}}_{loc}$-smoothly. In this case, the restriction of the function to the free boundary is given by
\begin{equation}
B_{\varepsilon}(x_1,\,x_1^2+\varepsilon^2)=\frac{1}{1-\varepsilon}\exp\left(x_1-\varepsilon\right),
\end{equation}
which is infinitely smooth. To see that the example with the function $B_{\varepsilon}$ shows the necessity of a $C^2$-majorant for Lemma~\ref{LeGuan} and Proposition~\ref{StGuan}, one has to note that the extremal sets of the function are tangent intervals
\begin{equation}
((u,\,u^2),\,(u+\varepsilon,\,(u+\varepsilon)^2+\varepsilon^2))=\{(u+t,\,u+(2u+2\varepsilon)t)|\,t\in(0,\,\varepsilon)\}.
\end{equation}
However the set $\{\BG=L\}$ is the extremal set together with the curve $\clos(E(x,\,L,\,l))\cup \gamma$; that is, it is connected but not relatively convex (see Fig.~\ref{nonconvex example}).
\begin{figure}[h]\centering
\def\svgwidth{6cm}
%% Creator: Inkscape 1.3.2 (091e20e, 2023-11-25, custom), www.inkscape.org
%% Accompanies image file 'prodzaeks.pdf' (pdf, eps, ps)
\begingroup%
  \makeatletter%
  \providecommand\color[2][]{%
    \errmessage{(Inkscape) Color is used for the text in Inkscape, but the package 'color.sty' is not loaded}%
    \renewcommand\color[2][]{}%
  }%
  \providecommand\transparent[1]{%
    \errmessage{(Inkscape) Transparency is used (non-zero) for the text in Inkscape, but the package 'transparent.sty' is not loaded}%
    \renewcommand\transparent[1]{}%
  }%
  \providecommand\rotatebox[2]{#2}%
  \newcommand*\fsize{\dimexpr\f@size pt\relax}%
  \newcommand*\lineheight[1]{\fontsize{\fsize}{#1\fsize}\selectfont}%
  \ifx\svgwidth\undefined%
    \setlength{\unitlength}{1500bp}%
    \ifx\svgscale\undefined%
      \relax%
    \else%
      \setlength{\unitlength}{\unitlength * \real{\svgscale}}%
    \fi%
  \else%
    \setlength{\unitlength}{\svgwidth}%
  \fi%
  \global\let\svgwidth\undefined%
  \global\let\svgscale\undefined%
  \makeatother%
  \begin{picture}(1,1)%
    \lineheight{1}%
    \setlength\tabcolsep{0pt}%
    \put(0,0){\includegraphics[width=\unitlength,page=1]{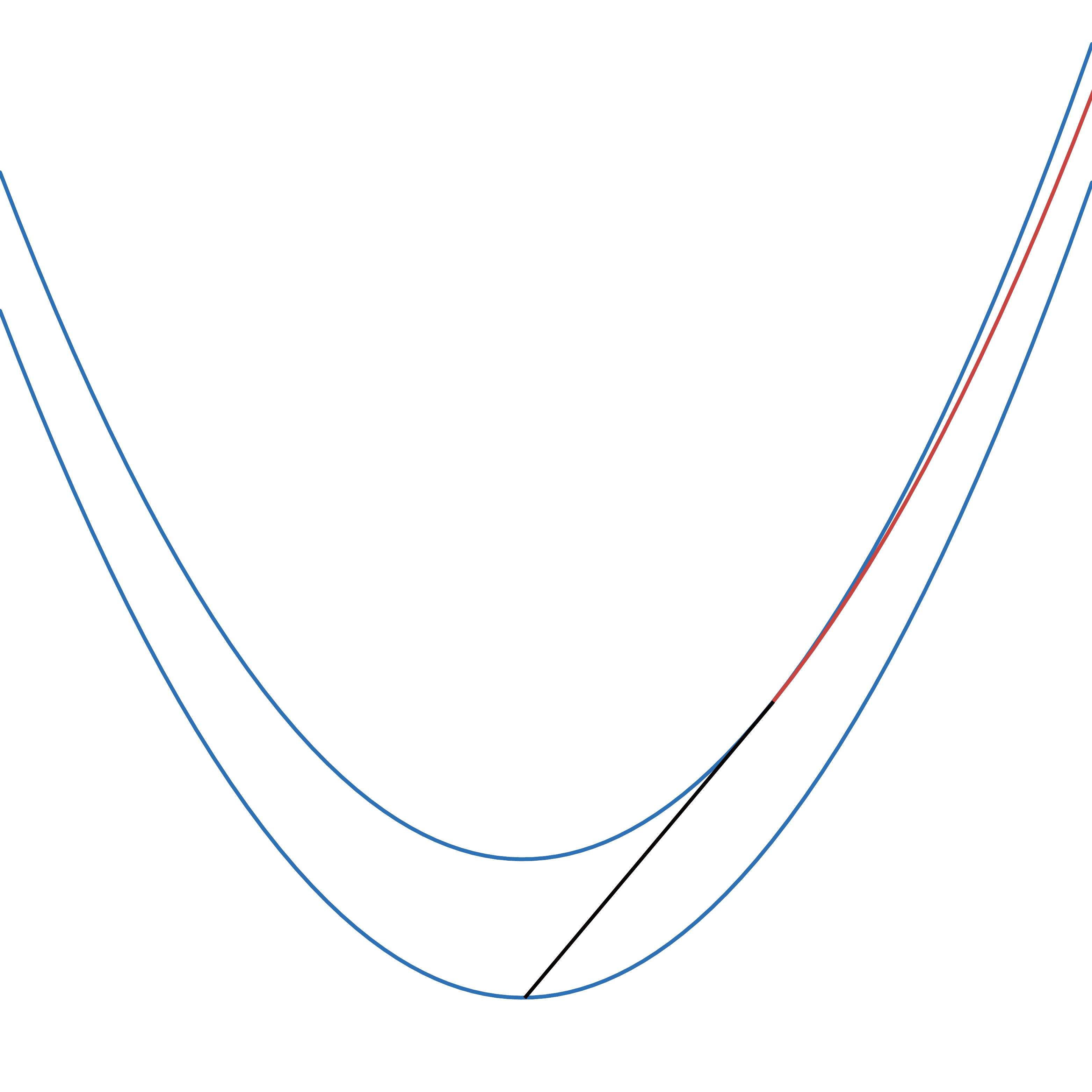}}%
    \put(0.28930255,0.16432373){\color[rgb]{0,0,0}\makebox(0,0)[lt]{\lineheight{1.25}\smash{\begin{tabular}[t]{l}$E(x,\,L,\,l)$\end{tabular}}}}%
    \put(0.78229187,0.42381409){\color[rgb]{0.78039216,0.26666667,0.25098039}\makebox(0,0)[lt]{\lineheight{1.25}\smash{\begin{tabular}[t]{l}\textcolor[rgb]{0.777,0.266,0.25}{$\gamma$}\end{tabular}}}}%
  \end{picture}%
\endgroup%

\caption{Example of nonconvexity of a component of the set where the function is equal to an affine functional, containing an extremal set}\label{nonconvex example}
\end{figure}
One can give an analogous example in a compact domain. Moreover, one can arrange that $\BG\in C^2(\clos(\Omega))$. Thus, for Lemma~\ref{LeGuan} and Proposition~\ref{StGuan}, the existence of a strictly locally concave majorant is important, and there is a counterexample even in the presence of a $C^2$ locally concave majorant.

\section{On relative convexity}

The theory constructed here admits several generalizations. Using the methods developed in this paper, one can transfer several classical results of convex analysis to the relatively convex setting. The following theorem may be regarded as an analogue of the corollary of Carath\'eodory's theorem~\ref{carat} and as a generalization of the Krein--Milman theorem (see Theorem 18.5 in~\cite{Rockafellar}) to unbounded convex sets.
\begin{Th}
A closed convex set $U$ that does not contain lines is a union of generalized simplices whose finite vertices are extreme points of $U$ and whose vertices at infinity are directions of extreme rays of $U$.
\end{Th}
The theory constructed above allows one to prove an analogous result for relatively convex relatively closed sets.
\begin{Th}
Let $U\subseteq\mathbb{R}^d$ be an open set, and let $X\subseteq U$ be a relatively convex relatively closed subset. Then $X$ is the union of lines lying entirely in $X$, their extreme points, and sets of the form $C\setminus V_{fin}(C)$, where $C$ ranges over generalized simplices such that $C\setminus V_{fin}(C)\subseteq X$ and whose finite vertices are either extreme points of $X$ or points of $\partial U$.
\end{Th}
The main Theorem~\ref{mainth} of this paper is a corollary of this theorem and of the following proposition.
\begin{St}
An extremal set $E(x,\,L,\,l)$ is a relatively closed, relatively convex subset of the open set $E_{>0}$ with no extreme points.
\end{St}

{\small
 
 St. Petersburg State University, Department of Mathematics and Computer Science;\\
 29, 14th Line of Vasilievsky Island, St. Petersburg, Russia;\\
 e-mail: e.dobronravov@spbu.ru\\
 \bigskip
 
 }

\end{document}